\documentclass{cmip-l}

\usepackage{amssymb,amsfonts,amsmath}
\usepackage{color,graphicx,epic,eepic}
\usepackage{hyperref}								
\usepackage{boxedminipage}

\newtheorem{theorem}{Theorem}[section]
\newtheorem{lemma}[theorem]{Lemma}

\newtheorem{conjecture}[theorem]{Conjecture}

\newtheorem{coro}[theorem]{Corollary}

\theoremstyle{definition}
\newtheorem{definition}[theorem]{Definition}

\theoremstyle{remark}

\numberwithin{equation}{section}

\def \N {{\mathbb N}}
\def \R {{\mathbb R}}
\def \Z {{\mathbb Z}}

\def \P {{\mathbb P}}
\def \E {{\mathbb E}}

\def \cW {{\mathcal W}}
\def \cL {{\mathcal L}}
\def \cD {{\mathcal D}}
\def \cP {{\mathcal P}}
\def \cC {{\mathcal C}}
\def \mP {{\mathbb P}}
\def \mE {{\mathbb E}}
\def \mZ {{\mathbb Z}}

\def \cD {{\mathcal D}}
\def \cL {{\mathcal L}}

\def \cP {{\mathcal P}}

\def \cF {{\mathcal F}}

\def \cC {{\mathcal C}}

\def \erom {{\mathrm{e}}}
\def \drom {{\mathrm{d}}}
\def \SRW {{\rm SRW}}
\def \SAW {{\rm SAW}}

\def \WD {{\rm WD}}
\def \SD {{\rm SD}}

\def \ba {\begin{array}}
\def \ea {\end{array}}
\def \ind {{1}}
\def \gep {{\varepsilon}}

\begin{document}

\setcounter{page}{319}

\title{Lectures on Random Polymers}

\author[Caravenna]{Francesco Caravenna} 
\address{Dipartimento di Matematica e Applicazioni, Universit\`a degli Studi
di Milano - Bicocca, Via Cozzi 53, 20125 Milano, Italy}
\email{francesco.caravenna@unimib.it}

\author[den Hollander]{Frank den Hollander}
\address{Mathematical Institute, Leiden University, P.O.\ Box 9512,
2300 RA Leiden, The Netherlands}
\email{denholla@math.leidenuniv.nl}

\author[P\'etr\'elis]{Nicolas P\'etr\'elis}
\address{Laboratoire de Math\'ematiques Jean Leray, 2 rue de la 
Houssini\`ere - BP 92208, F-44322 Nantes Cedex 3, France}
\email{petrelis@univ-nantes.fr}


\subjclass[2010]{Primary: 60F10, 60K37, 82B26, 82D60; Secondary: 60F67, 82B27, 82B44, 92D20}

\date{\today}

\begin{abstract}
These lecture notes are a guided tour through the fascinating world 
of polymer chains interacting with themselves and/or with their 
environment. The focus is on the mathematical description of 
a number of physical and chemical phenomena, with particular emphasis 
on phase transitions and space-time scaling. The topics covered, 
though only a selection, are typical for the area. Sections~\ref{S1}--\ref{S3} 
describe models of polymers without disorder, Sections~\ref{S4}--\ref{S6} 
models of polymers with disorder. Appendices~\ref{appA}--\ref{appE} 
contain tutorials in which a number of key techniques are explained 
in more detail. 
\end{abstract}

\maketitle

\tableofcontents


\section*{Foreword}

These notes are based on six lectures by Frank den Hollander and 
five tutorials by Francesco Caravenna and Nicolas P\'etr\'elis. 
The final manuscript was prepared jointly by the three authors. 
A large part of the material is drawn from the monographs by 
Giambattista Giacomin~\cite{Gi07} and Frank den Hollander~\cite{dHo09}. 
Links are made to some of the lectures presented elsewhere in this 
volume. In particular, it is argued that in two dimensions the 
Schramm-Loewner Evolution (SLE) is a natural candidate for the 
scaling limit of several of the ``exotic lattice path'' models 
that are used to describe self-interacting random polymers. Each 
lecture provides a snapshot of a particular class of models and 
ends with a formulation of some open problems. The six lectures 
can be read independently.     

Random polymers form an exciting, highly active and challenging field of 
research that lies at the crossroads between mathematics, physics, chemistry 
and biology. DNA, arguably the most important polymer of all, is subject 
to several of the phenomena that are described in these lectures: 
\emph{folding} (= collapse), \emph{denaturation} (= depinning due to 
temperature), \emph{unzipping} (= depinning due to force), \emph{adsorption} 
(= localization on a substrate).


\section{Background, model setting, free energy, two basic models}
\label{S1}

In this section we describe the physical and chemical background of 
random polymers (Sections~\ref{S1.1}--\ref{S1.4}), formulate the model 
setting in which we will be working (Section~\ref{S1.5}), discuss the 
central role of free energy (Section~\ref{S1.6}), describe two basic models 
of random polymer chains: the simple random walk and the self-avoiding 
walk (Section~\ref{S1.7}), and formulate a key open problem for the latter 
(Section~\ref{S1.8}).


\subsection{What is a polymer?}
\label{S1.1}

A polymer is a large molecule consisting of mono\-mers that are tied together 
by \emph{chemical bonds}. The monomers can be either small units (such as 
${\rm CH_2}$ in polyethylene; Fig.~\ref{fig-Poly}) or larger units with an 
internal structure (such as the adenine-thymine and cytosine-guanine base 
pairs in the DNA double helix; Fig.~\ref{fig-DNA}). Polymers abound in nature 
because of the \emph{multivalency} of atoms like carbon, oxygen, nitrogen and 
sulfur, which are capable of forming long concatenated structures. 

\vspace{1.7cm}
\begin{figure}[htbp]
\begin{center}
\includegraphics[width=.40\hsize]{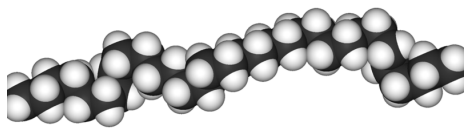}
\end{center}
\vspace{-1.7cm}
\caption{\small Polyethylene.}
\label{fig-Poly}
\end{figure}

\begin{figure}[htbp]
\includegraphics[width=.40\hsize]{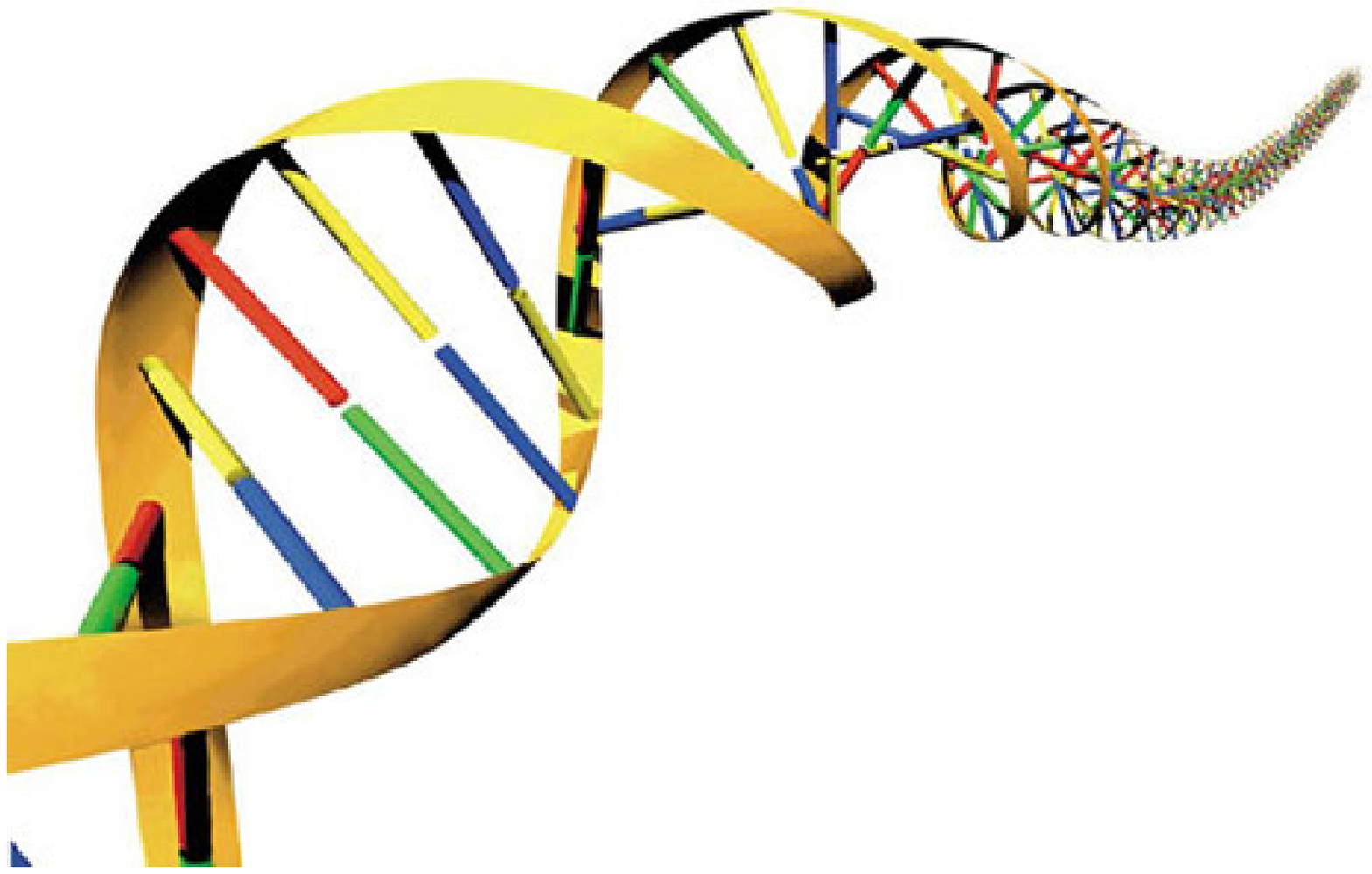}
\caption{\small DNA.}
\label{fig-DNA}
\end{figure}


\subsection{What types of polymers occur in nature?} 
\label{S1.2}

Polymers come in two varieties: \emph{homopolymers}, with all their 
monomers identical (such as polyethylene), and \emph{copolymers}, with 
two or more different types of monomers (such as DNA). The order of the 
monomer types in copolymers can be either periodic (e.g.\ in agar) or 
random (e.g.\ in carrageenan).  

Another classification is into \emph{synthetic polymers} (like nylon, 
polyethylene and polystyrene) and \emph{natural polymers} (also called 
biopolymers). Major subclasses of the latter are: (a) proteins (strings 
of amino-acids; Fig.~\ref{fig-protfold}); (b) nucleic acids (DNA, RNA;
Fig.~\ref{fig-DNA}); (c) polysaccharides (like agar, alginate, amylopectin, 
amylose, carrageenan, cellulose); (d) lignin (plant cement); (e) rubber. 
Apart from (a)--(e), which are organic materials, clays and minerals are 
inorganic examples of natural polymers. Synthetic polymers typically are 
homopolymers, while natural polymers typically are copolymers (with notable 
exceptions). Bacterial polysaccharides tend to be periodic, while plant 
polysaccharides tend to be random.

\begin{figure}[htbp]
\includegraphics[width=.30\hsize]{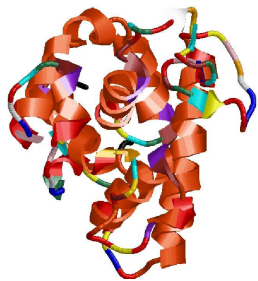}
\caption{\small A folded-up protein.}
\label{fig-protfold}
\end{figure}

Yet another classification is into \emph{linear polymers} and \emph{branched 
polymers}. In the former, the monomers have one reactive group (such as ${\rm 
CH_2}$), leading to a linear organization as a result of the polymerization 
process. In the latter, the monomers have two or more reactive groups (such 
as hydroxy acid), leading to a network organization with multiple cross connections. 
Most natural polymers are linear, like proteins, DNA, RNA, and the polysaccharides 
agar, alginate, amylose, carrageenan and cellulose. Some polysaccharides are 
branched, like amylopectin. Many synthetic polymers are linear, and many are 
branched. An example of a branched polymer is rubber, both natural and synthetic.
The network structure of rubber is what gives it both strength and flexibility!


\subsection{What are the size and shape of a polymer?}
\label{S1.3}

Size and shape are two key properties of a polymer.

\medskip\noindent
{\bf Size:}
The chemical process of building a polymer from monomers is called 
\emph{polymerization}. The size of a polymer may vary from $10^3$ 
up to $10^{10}$ (shorter chains do not deserve to be called a polymer, 
longer chains have not been recorded). Human DNA has $10^9-10^{10}$ 
base pairs, lignin consists of $10^6-10^7$ phenyl-propanes, while 
polysaccharides carry $10^3-10^4$ sugar units.

Both in synthetic and in natural polymers, the \emph{size distribution} 
may either be broad, with numbers varying significantly from polymer 
to polymer (e.g.\ nylons, polysaccharides), or be narrow (e.g.\ proteins, 
DNA). In synthetic polymers the size distribution can be made narrow through 
specific polymerization methods. 

The length of the monomer units varies from $1.5\,\mbox{\AA}$ 
(for ${\rm CH_2}$ in polyethylene) to $20\,\mbox{\AA}$ (for 
the base pairs in DNA), with $1\,\mbox{\AA} = 10^{-10}\,\mbox{m}$. 

\medskip\noindent
{\bf Shape:}
The chemical bonds in a polymer are flexible, so that the polymer can 
arrange itself in \emph{many different shapes}. The longer the chain, 
the more involved these shapes tend to be. For instance, the polymer 
may wind around itself to form a \emph{knot} (Fig.~\ref{fig-polknot}), 
may expand itself to form a \emph{random coil} due to repulsive forces 
caused by excluded-volume (e.g.\ when a good solvent surrounds the 
monomers and prevents them from coming close to each other), or may 
collapse on itself to form a \emph{compact ball} due to attractive van 
der Waals forces between the monomers (or repulsive forces between the 
monomers and a poor solvent causing the polymer to fold itself up).

\begin{figure}[htbp]
\includegraphics[width=.30\hsize]{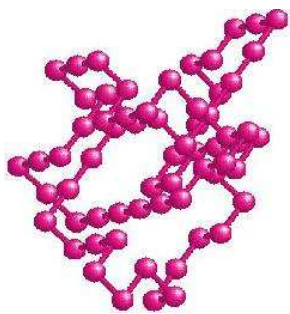}
\caption{\small A knotted polymer.}
\label{fig-polknot}
\end{figure}

In addition, the polymer may interact with a \emph{surface} or with two fluids 
separated by an \emph{interface}, may interact with a field of \emph{random charges} 
in which it is immersed, or may be subjected to a \emph{force} applied to one of its 
endpoints. Many models have been invented to describe such situations. In 
Sections~\ref{S2}--\ref{S6} we take a look at some of these models.


\subsection{What questions may a mathematician ask and hope to answer?}
\label{S1.4}

The majority of mathematical research deals with \emph{linear polymers}. Examples 
of quantities of interest are: number of different spatial configurations, 
end-to-end distance (subdiffusive/diffusive/superdiffusive), fraction of monomers 
adsorbed onto a surface, force needed to pull an adsorbed polymer off a surface,
effect of randomness in the interactions, all typically in the limit as the polymer 
gets long (so that techniques from probability theory and statistical physics can 
be used). In these lectures special attention is given to the \emph{free energy} 
of the polymer, and to the presence of \emph{phase transitions} as a function of 
underlying model parameters. Recent surveys are the monographs by Giacomin~\cite{Gi07}
and den Hollander~\cite{dHo09}, and \emph{references therein}.


\subsection{What is the model setting?}
\label{S1.5}

In mathematical models polymers often live on a lattice, like $\Z^d$, $d \geq 1$, 
and are modelled as random paths, where the monomers are the vertices in the path, 
and the chemical bonds connecting the monomers are the edges in the path 
(Fig.~\ref{fig-latpath}). 

\begin{figure}[htbp]
\includegraphics[width=.25\hsize]{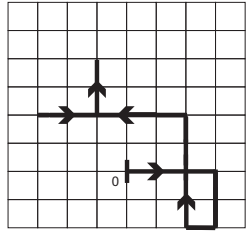}
\caption{\small A lattice path.}
\label{fig-latpath}
\end{figure}

\medskip\noindent
{\bf I.\ Paths and energies:}
Choosing a polymer model amounts to fixing for each $n\in\N_0 = \N \cup \{0\}$:
\begin{itemize}
\item[(1)]
$\cW_n$, a set of allowed $n$-step paths on $\Z^d$,
\item[(2)]
$H_n$, a Hamiltonian function that associates an energy to each 
path in $\cW_n$.
\end{itemize}
The choice of $\cW_n$ may allow for directed or undirected paths, possibly with 
some geometric constraints (see Fig.~\ref{fig-dirpath}).

\begin{figure}[htbp]
\begin{center}
\setlength{\unitlength}{0.4cm}
\begin{picture}(18,6)(2,-1)
{\thicklines
\qbezier(0,0)(.5,.5)(1,1)
\qbezier(1,1)(1.5,1.5)(2,2)
\qbezier(2,2)(2.5,1.5)(3,1)
\qbezier(3,1)(3.5,.5)(4,0)
\qbezier(4,0)(4.5,-.5)(5,-1)
\qbezier(5,-1)(5.5,-.5)(6,0)
\qbezier(6,0)(6.5,.5)(7,1)
\qbezier(9,0)(9.5,.5)(10,1)
\qbezier(10,1)(10.5,1)(11,1)
\qbezier(11,1)(11.5,1.5)(12,2)
\qbezier(12,2)(12.5,1.5)(13,1)
\qbezier(13,1)(13.5,1)(14,1)
\qbezier(14,1)(14.5,.5)(15,0)
\qbezier(17,-1)(17,-.5)(17,0)
\qbezier(17,0)(17,.5)(17,1)
\qbezier(17,1)(17,1.5)(17,2)
\qbezier(17,2)(17.5,2)(18,2)
\qbezier(18,2)(18,2.5)(18,3)
\qbezier(18,3)(18.5,3)(19,3)
\qbezier(19,3)(19.5,3)(20,3)
\qbezier(20,3)(20,2.5)(20,2)
\qbezier(20,2)(20,1.5)(20,1)
\qbezier(20,1)(20,.5)(20,0)
\qbezier(20,0)(20.5,0)(21,0)
}
\put(0,0){\circle*{0.25}}
\put(1,1){\circle*{0.25}}
\put(2,2){\circle*{0.25}}
\put(3,1){\circle*{0.25}}
\put(4,0){\circle*{0.25}}
\put(5,-1){\circle*{0.25}}
\put(6,0){\circle*{0.25}}
\put(7,1){\circle*{0.25}}
\put(9,0){\circle*{0.25}}
\put(10,1){\circle*{0.25}}
\put(11,1){\circle*{0.25}}
\put(12,2){\circle*{0.25}}
\put(13,1){\circle*{0.25}}
\put(14,1){\circle*{0.25}}
\put(15,0){\circle*{0.25}}
\put(17,-1){\circle*{0.25}}
\put(17,0){\circle*{0.25}}
\put(17,1){\circle*{0.25}}
\put(17,2){\circle*{0.25}}
\put(18,2){\circle*{0.25}}
\put(18,3){\circle*{0.25}}
\put(19,3){\circle*{0.25}}
\put(20,3){\circle*{0.25}}
\put(20,2){\circle*{0.25}}
\put(20,1){\circle*{0.25}}
\put(20,0){\circle*{0.25}}
\put(21,0){\circle*{0.25}}
\end{picture}
\end{center}
\caption{\small Three examples of directed paths on $\Z^2$.}
\label{fig-dirpath}
\end{figure}
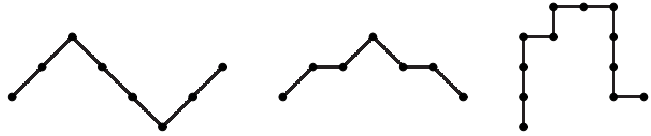

\noindent
The choice of $H_n$ captures the interaction of the polymer with itself and/or its 
environment. Typically, $H_n$ depends on one or two parameters, including temperature.
Sections~\ref{S2}--\ref{S6} will provide many examples.

\medskip\noindent
{\bf II.\ Path measure:}
For each $n\in\N_0$, the law of the polymer of length $n$ is defined by assigning to 
each $w\in\cW_n$ a probability given by
\[
P_n(w) = \frac{1}{Z_n}\,\erom^{-H_n(w)}, \qquad w \in \cW_n,
\]
where $Z_n$ is the normalizing partition sum. This is called the \emph{Gibbs measure} 
associated with the pair $(\cW_n,H_n)$, and it describes the polymer \emph{in 
equilibrium} with itself and/or its environment, at a fixed length $n$. Paths with
a low (high) energy have a large (small) probability under the Gibbs measure. \emph{Note:} 
In the physics and chemistry literature, $H_n/kT$ is put into the exponent instead of
$H_n$, with $T$ the absolute temperature and $k$ the Boltzmann constant. Since $kT$ 
has the dimension of energy, $H_n/kT$ is a dimensionless quantity. In our notation, 
however, we absorb $kT$ into $H_n$.
 
\medskip\noindent
{\bf III. Random environment:} 
In some models $H_n$ also depends on a
\[ 
\mbox{random environment } \omega
\]
describing e.g.\ a random ordering of the monomer types or a random field of charges 
in which the polymer is immersed. In this case the Hamiltonian is written as 
$H_n^\omega$, and the path measure as $P_n^\omega$. The law of $\omega$ is denoted 
by $\mP$. (Carefully distinguish between the symbols $w$ and $\omega$.)

Three types of path measures with disorder are of interest:  
\begin{itemize}
\item[(1)] 
The \emph{quenched} Gibbs measure   
\[
P_n^\omega(w) = \frac{1}{Z_n^\omega}\,\erom^{-H_n^\omega(w)}, \qquad 
w \in \cW_n.
\]
\item[(2)] 
The \emph{average quenched} Gibbs measure 
\[
\mE(P_n^\omega(w)) = \int P_n^\omega(w)\,\mP(\drom\omega), \qquad 
w \in \cW_n.
\]
\item[(3)] 
The \emph{annealed} Gibbs measure
\[
\mP_n(w) = \frac{1}{\mZ_n}\,
\int \erom^{-H_n^\omega(w)}\,\mP(\drom\omega), \qquad
w \in \cW_n.
\]
\end{itemize}
These are used to describe a polymer whose random environment is frozen [(1)+(2)],
respectively, takes part in the equilibration [(3)]. Note that in (3), unlike
in (2), the normalizing partition sum does not (!) appear under the integral.

It is also possible to consider models where the length or the configuration 
of the polymer changes with time (e.g.\ due to growing or shrinking), or to 
consider a Metropolis dynamics associated with the Hamiltonian for an appropriate 
choice of allowed transitions. These \emph{non-equilibrium} situations are very 
interesting and challenging, but so far the available mathematics is rather limited. 
Two recent references are Caputo, Martinelli and Toninelli~\cite{CaMaTo08},
Caputo, Lacoin, Martinelli, Simenhaus and Toninelli~\cite{CaLaMaSiTopr}.


\subsection{The central role of free energy}
\label{S1.6}

The \emph{free energy} of the polymer is defined as
\[
f = \lim_{n\to\infty} \frac{1}{n} \log Z_n
\]
or, in the presence of a random environment, as
\[
f = \lim_{n\to\infty} \frac{1}{n} \log Z_n^\omega \qquad \omega\text{-a.s.}
\]
If the limit exists, then it typically is constant $\omega$-a.s., a property referred 
to as \emph{self-averaging}. We next discuss existence of $f$ and some of its 
properties. 

\begin{figure}[htbp]
\includegraphics[width=.30\hsize]{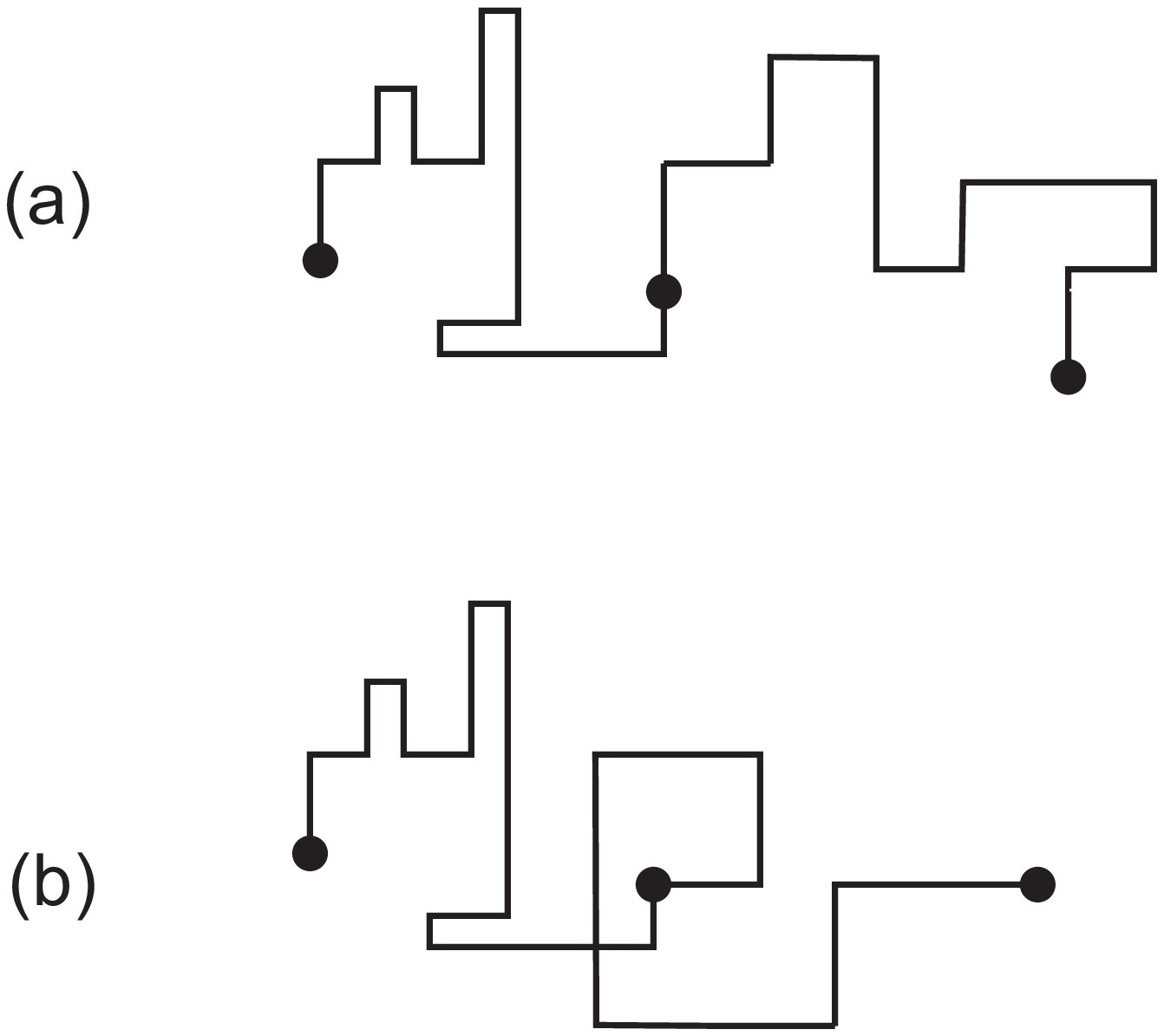}
\caption{\small Concatenation of two self-avoiding paths: (a) the 
concatenation is self-avoiding; (b) the concatenation is not 
self-avoiding.}
\label{fig-concat}
\end{figure}

\medskip\noindent
{\bf I.\ Existence of the free energy:}
When $H_n$ assigns a repulsive self-interaction to the polymer, the partition sum 
$Z_n$ satisfies the inequality
\[
Z_n \leq Z_m\,Z_{n-m} \qquad \forall\,0 \leq m \leq n.
\]
(See Fig.~\ref{fig-concat} for an example involving the counting of self-avoiding 
paths, i.e., $Z_n=|\cW_n|$.) Consequently, 
\[
n \mapsto nf_n = \log Z_n
\]
is a \emph{subadditive sequence}, so that
\[
f = \lim_{n\to\infty} f_n = \inf_{n\in\N} f_n \in [-\infty,\infty).
\]
(See the tutorial in Appendix A.1 of Bauerschmidt, Duminil-Copin, Goodman 
and Slade~\cite{BaDuGoSl11}.) If, moreover, $\inf_{w\in\cW_n} H_n(w) \leq Cn$ 
for all $n\in\N$ and some $C<\infty$, then $f \neq - \infty$. A similar result 
holds when $H_n$ assigns an attractive self-interaction to the polymer, in 
which case the inequalities are reversed, $f \in (-\infty,\infty]$, and $f 
\neq \infty$ when $|\cW_n|\leq \erom^{Cn}$ and $\inf_{w\in\cW_n} H_n(w) \geq -Cn$ 
for all $n\in\N$ and some $C<\infty$. 

When $H_n$ assigns both repulsive and attractive interactions to the polymer, 
then the above argument is generally not available, and the existence of the 
free energy either remains open or has to be established by other means. 
Many examples, scenarios and techniques are available. {\bf Tutorial 1 in 
Appendix \ref{appA}} describes two techniques to prove existence of free 
energies, in the context of the model of a polymer near a random interface
that is the topic of Section~\ref{S4}.

In the presence of a random environment $\omega$, it is often possible to 
derive a random form of subadditivity. When applicable, 
\[
n\mapsto nf_n^\omega = \log Z_n^\omega
\] 
becomes a \emph{subadditive random process}, and Kingman's subadditive ergodic 
theorem implies the existence of
\[
f = \lim_{n\to\infty} f_n^\omega \qquad \omega\text{-a.s.} 
\]
(as explained in {\bf Tutorial 1 in Appendix \ref{appA}}). This fact is of key 
importance for polymers with disorder.

\medskip\noindent
{\bf II.\ Convexity of the free energy:}
Suppose that the Hamiltonian depends linearly on a single parameter $\beta\in\R$, 
which is pulled out by writing $\beta H_n$ instead of $H_n$. Then, by the H\"older 
inequality, $\beta \mapsto f_n(\beta)$ is convex for all $n\in\N_0$ and hence so 
is $\beta\mapsto f(\beta)$. Convexity and finiteness imply continuity, and also 
monotonicity on either side of a minimum. Moreover, at those values of $\beta$ where 
$f(\beta)$ is differentiable, convexity implies that 
\[
f'(\beta) = \lim_{n\to\infty} f_n'(\beta).
\]
The latter observation is important because
\[
\begin{aligned}
f_n'(\beta) 
&= \left[\frac{1}{n}\,\log Z_n(\beta)\right]'
= \frac{1}{n}\,\frac{Z_n'(\beta)}{Z_n(\beta)}\\
&= \frac{1}{n}\,\frac{1}{Z_n(\beta)}\,
\frac{\partial}{\partial\beta} \left(\sum_{w\in\cW_n} 
\erom^{-\beta H_n(w)}\right)
=\frac{1}{n} \sum_{w\in\cW_n} [-H_n(w)]\,P_n^\beta(w).
\end{aligned}
\]
What this says is that $-\beta f'(\beta)$ is the limiting energy per monomer 
under the Gibbs measure as $n\to\infty$. At those values of $\beta$ where 
the free energy fails to be differentiable this quantity is discontinuous, 
signalling the occurrence of a first-order \emph{phase transition}. (Several
examples will be given later on.) Higher-order phase transitions correspond 
to discontinuity of higher-order derivatives of $f$.


\subsection{Two basic models}
\label{S1.7}

The remainder of this section takes a brief look at two basic models for a 
polymer chain: (1) the \emph{simple random walk}, a polymer without self-interaction; 
(2) the \emph{self-avoiding walk}, a polymer with excluded-volume self-interaction. 
In some sense these are the ``plain vanilla'' and ``plain chocolate'' versions 
of a polymer chain. The self-avoiding walk is the topic of the lectures by 
Bauerschmidt, Duminil-Copin, Goodman and Slade~\cite{BaDuGoSl11}.

\medskip\noindent
{\bf (1) Simple random walk:}
$\SRW$ on $\Z^d$ is the random process $(S_n)_{n\in\N_0}$ defined by
\[
S_0 = 0, \qquad S_n = \sum_{i=1}^n X_i, \quad n \in \N,
\]
where $X = (X_i)_{i\in\N}$ is an i.i.d.\ sequence of random variables 
taking values in $\Z^d$ with marginal law ($\|\cdot\|$ is the Euclidean 
norm)
\[
P(X_1=x) = \left\{\begin{array}{ll}
\tfrac{1}{2d}, &x\in\mZ^d \mbox{ with } \|x\|=1,\\[0.2cm]
0, &\mbox{otherwise}.
\end{array}
\right.
\]
Think of $X_i$ as the orientation of the chemical bond between the $(i-1)$-th and 
$i$-th monomer, and of $S_n$ as the location of the end-point of the polymer of 
length $n$. $\SRW$ corresponds to choosing
\[
\begin{aligned}
\cW_n &= \big\{w=(w_i)_{i=0}^n\in(\mZ^d)^{n+1}\colon\\
&\qquad\qquad\qquad w_0=0,\,\|w_{i+1}-w_i\|=1\,\,\forall\,0\leq i<n\big\},\\ 
H_n &\equiv 0,
\end{aligned}
\]
so that $P_n$ is the uniform distribution on $\cW_n$. In this correspondence, think 
of $(S_i)_{i=0}^n$ as the realization of $(w_i)_{i=0}^n$ drawn according to $P_n$. 

\begin{figure}[htbp]
\begin{center}
\includegraphics[width=.30\hsize]{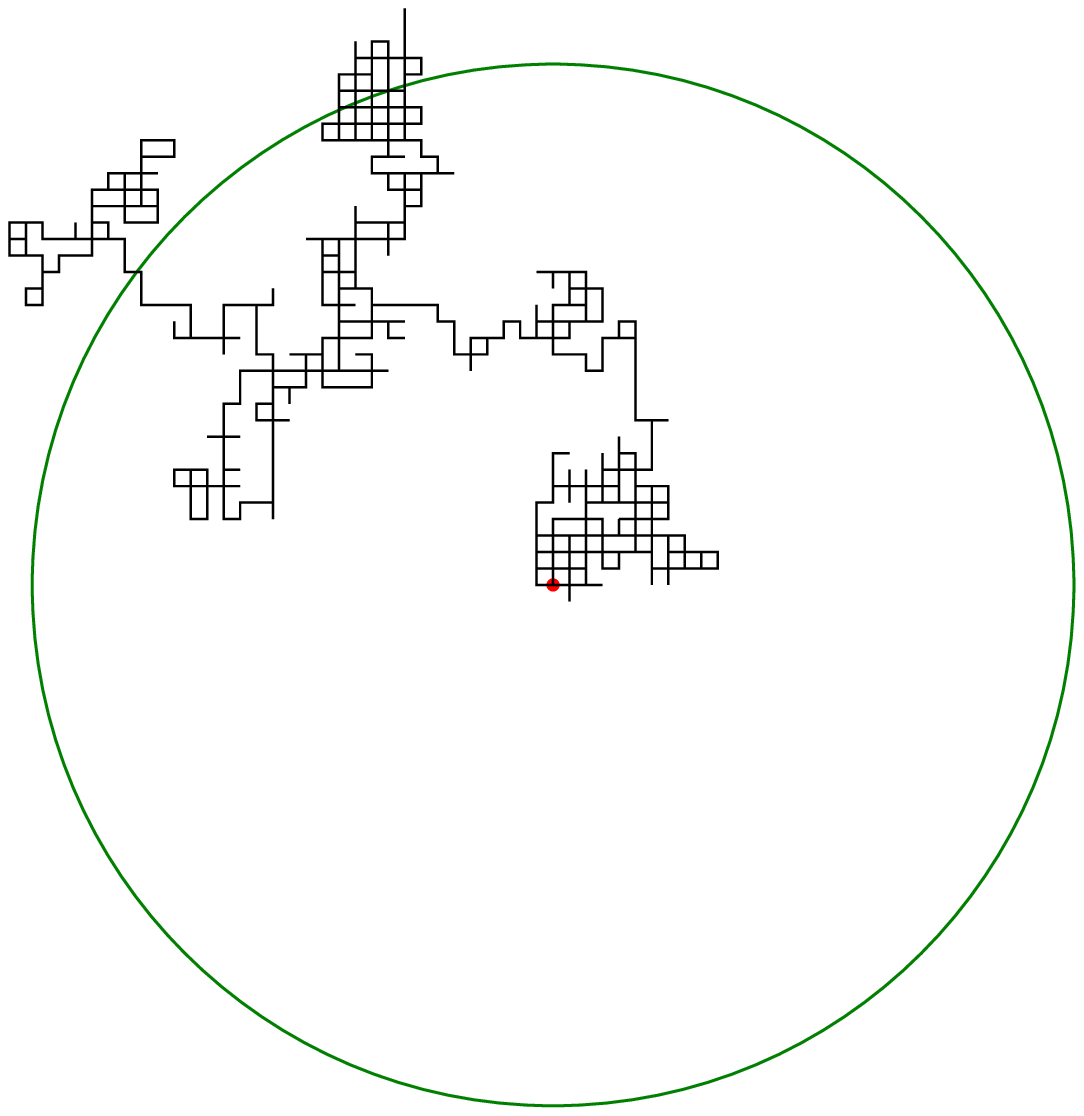}
\includegraphics[width=.30\hsize]{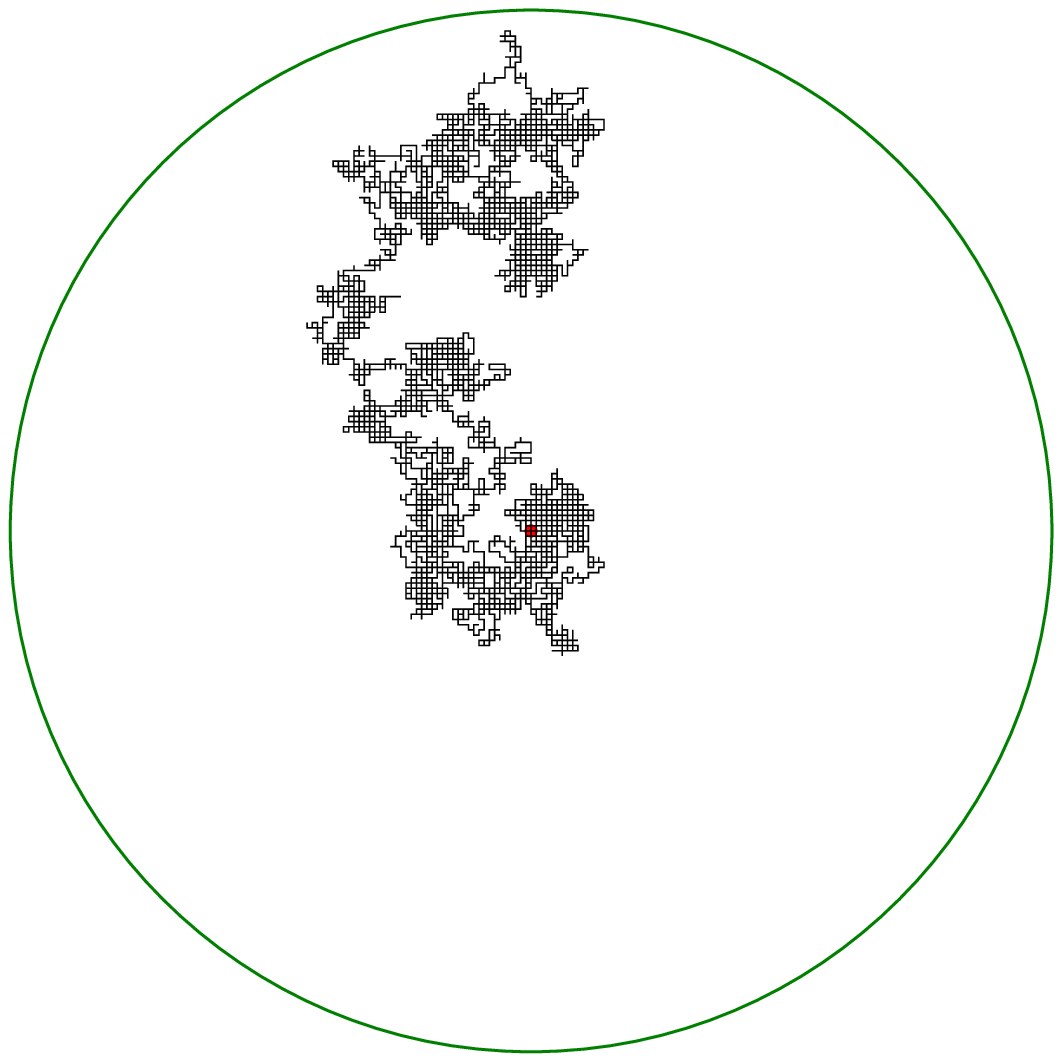}
\includegraphics[width=.30\hsize]{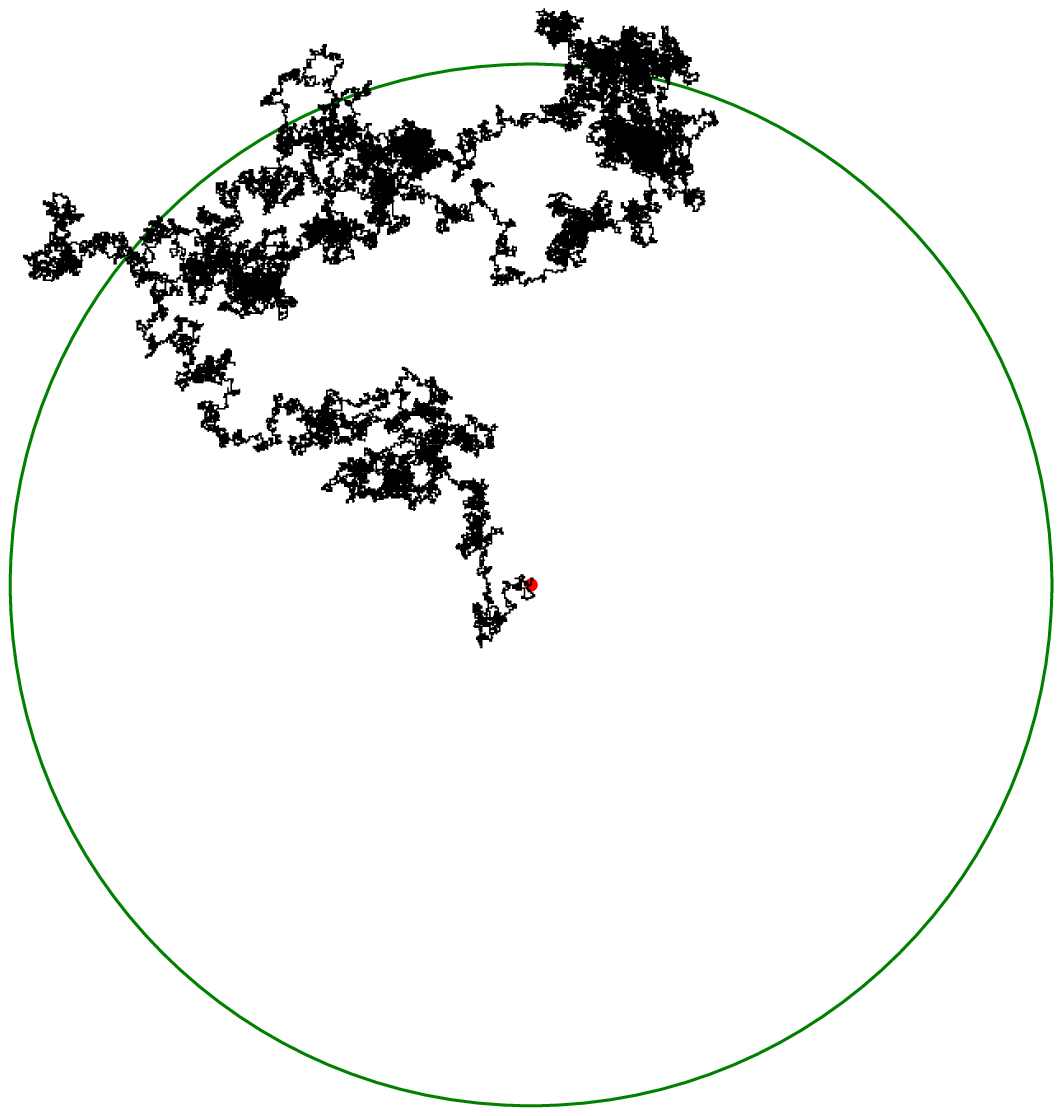}
\end{center}
\caption{\small Simulation of $\SRW$ on $\Z^2$ with $n=10^3$, $10^4$ and $10^5$ 
steps. The circles have radius $n^{1/2}$ in units of the step size. [Courtesy of
Bill Casselman and Gordon Slade.]}
\label{fig-SRW}
\end{figure}

A distinctive feature of $\SRW$ is that it exhibits \emph{diffusive behavior}, i.e.,
\[
E_n(S_n)=0 \quad \mbox{ and } \quad E_n(\|S_n\|^2) = n \qquad \forall\,n\in\N_0
\]
and
\[
\left(\frac{1}{n^{1/2}}\,S_{\lfloor nt \rfloor}\right)_{0 \leq t \leq 1}
\quad \Longrightarrow \quad (B_t)_{0 \leq t \leq 1} \qquad \mbox{ as } n \to \infty,
\]
where the right-hand side is Brownian motion on $\R^d$, and $\Longrightarrow$ denotes 
convergence in distribution on the space of c\`adl\`ag paths endowed with the Skorohod 
topology (see Fig.~\ref{fig-SRW}). 

\medskip\noindent
{\bf (2) Self-avoiding walk:}
$\SAW$ corresponds to choosing
\[
\begin{aligned}
\cW_n &= \big\{w=(w_i)_{i=0}^n\in(\mZ^d)^{n+1}\colon\\
&\qquad\qquad\qquad w_0=0,\,\|w_{i+1}-w_i\|=1\,\,\forall\,0 \leq i < n,\\
&\qquad\qquad\qquad w_i \neq w_j\,\,\forall\,0 \leq i < j \leq n\big\},\\
H_n &\equiv 0,
\end{aligned}
\]
so that $P_n$ is the uniform distribution on $\cW_n$. Again, think of $(S_i)_{i=0}^n$ as 
the realization of $(w_i)_{i=0}^n$ drawn according to $P_n$. 

\begin{figure}[htbp]
\begin{center}
\includegraphics[width=.30\hsize]{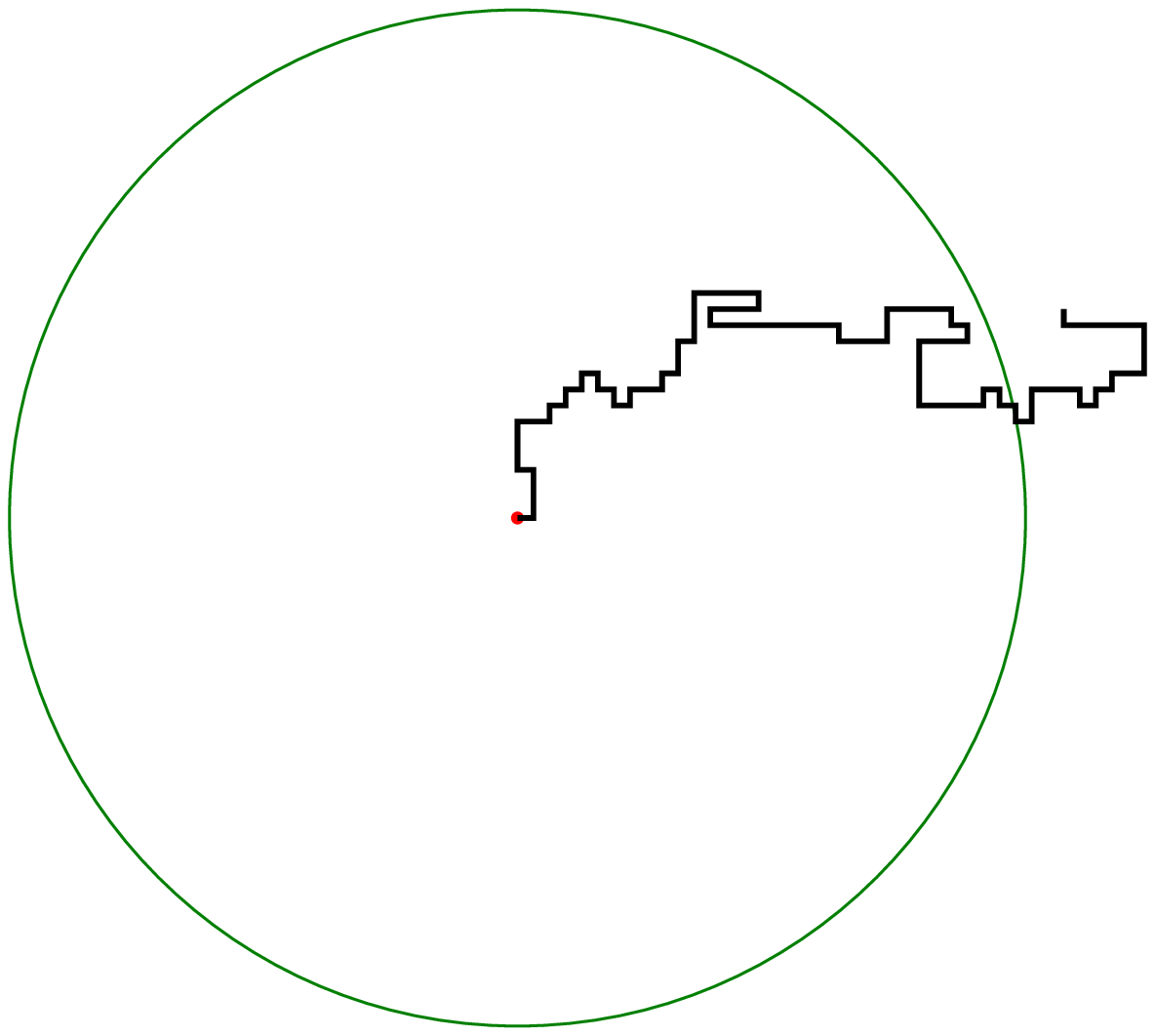}
\includegraphics[width=.30\hsize]{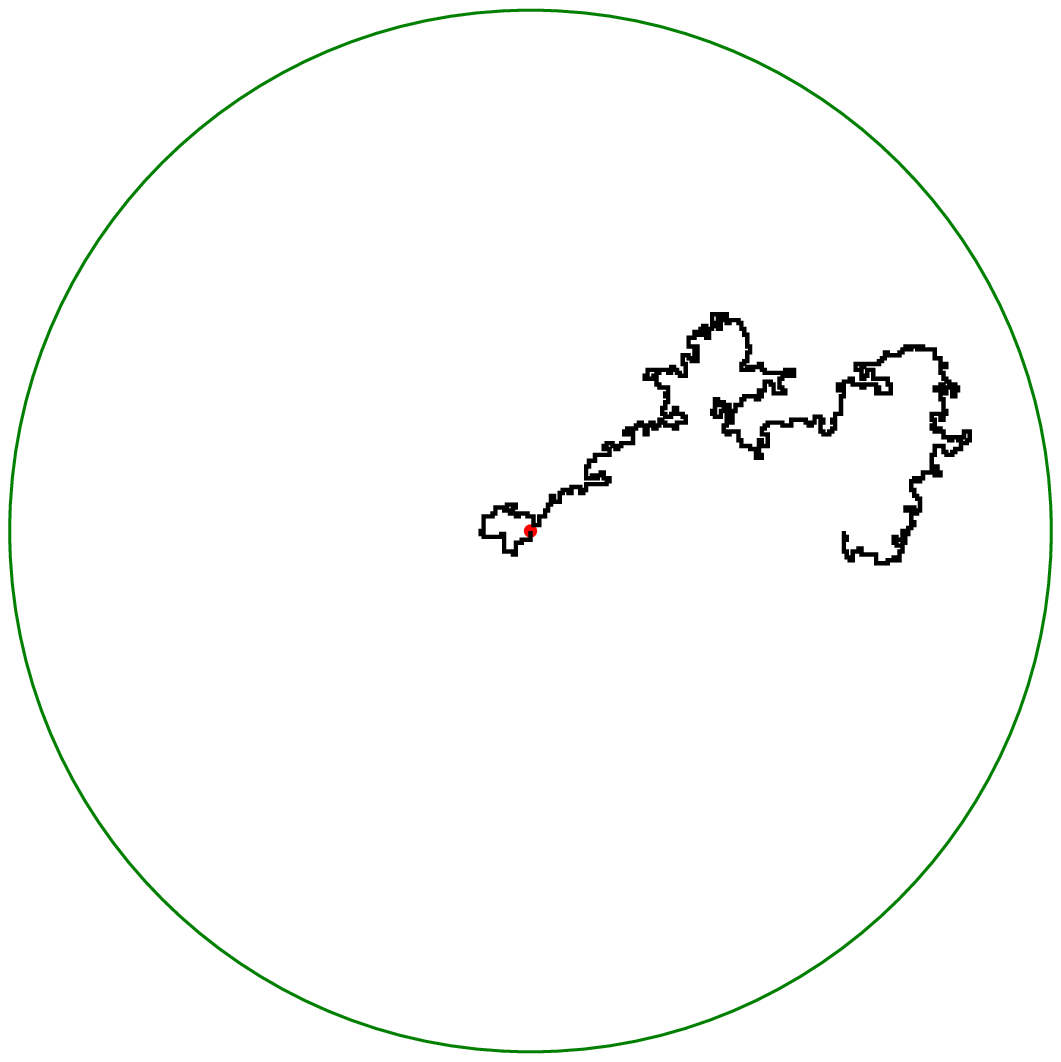}
\includegraphics[width=.30\hsize]{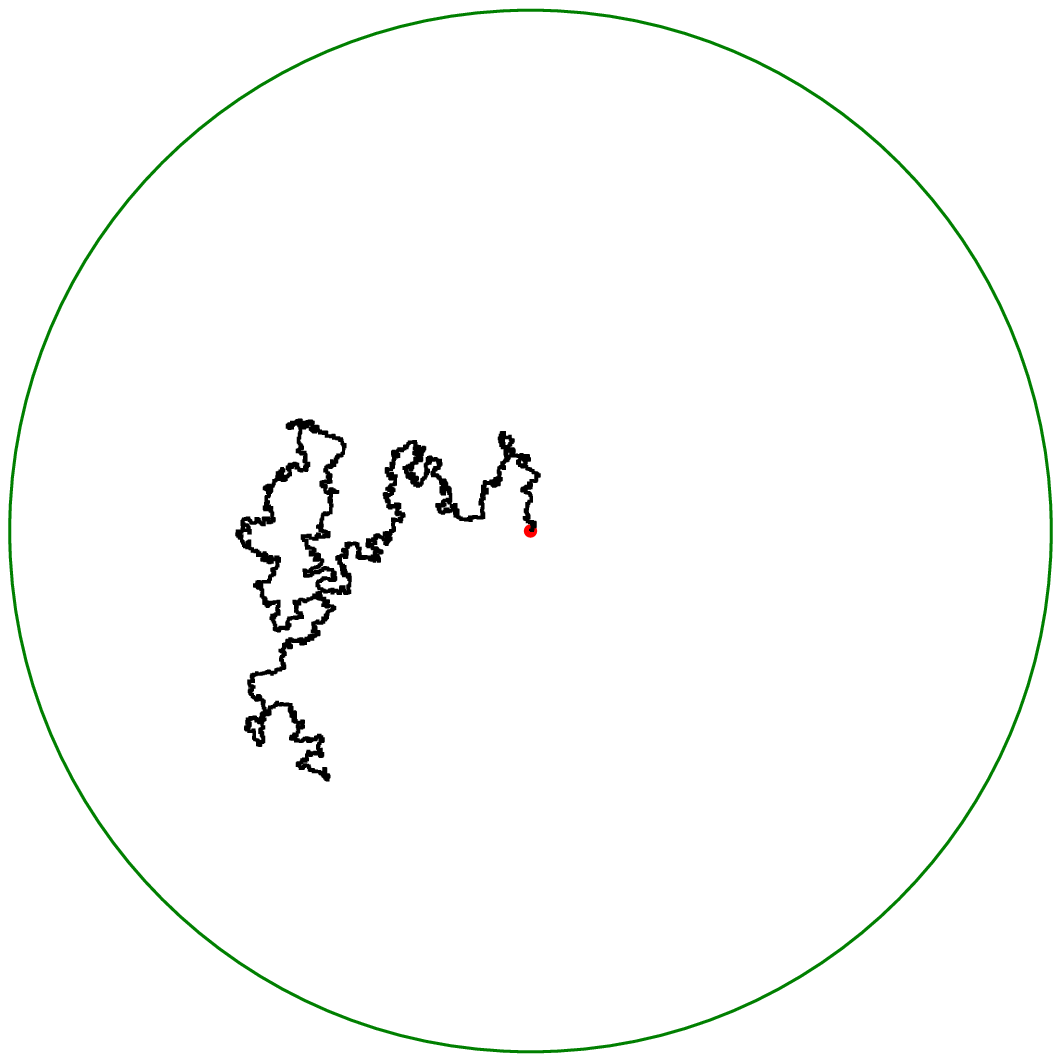}
\end{center}
\caption{\small Simulation of $\SAW$ on $\Z^2$ with $n=10^2$, $10^3$ and $10^4$ steps. 
The circles have radius $n^{3/4}$ in units of the step size. [Courtesy of Bill 
Casselman and Gordon Slade.]}
\label{fig-SAW}
\end{figure}

$\SAW$ in $d=1$ is trivial. In $d \geq 2$ no closed form expression is available for 
$E_n(\|S_n\|^2)$, but for small and moderate $n$ it can be computed via \emph{exact 
enumeration methods}. The current record is: $n=71$ for $d=2$ (Jensen~\cite{Jewww}); 
$n=36$ for $d=3$ (Schram, Barkema and Bisseling~\cite{ScBaBi11}); $n=24$ for $d \geq 4$ 
(Clisby, Liang and Slade~\cite{ClLiSl07}). Larger $n$ can be handled either via 
numerical simulation (presently up to $n=2^{25} \approx 3.3\times 10^{7}$ in $d=3$) 
or with the help of extrapolation techniques. 

The mean-square displacement is predicted to scale like
\[
E_n(\|S_n\|^2) = \left\{\begin{array}{ll}
D\,n^{2\nu}\,[1+o(1)], &\quad d \neq 4,\\[0.2cm]
D\,n(\log n)^{\frac14}\,[1+o(1)], &\quad d=4, 
\end{array} \quad \text{ as } n\to\infty, 
\right.
\]
with $D$ a non-universal diffusion constant and $\nu$ a universal \emph{critical 
exponent}. Here, universal refers to the fact that $\nu$ is expected to depend only 
on $d$, and to be independent of the fine details of the model (like the choice of 
the underlying lattice or the choice of the allowed increments of the path). 

The value of $\nu$ is predicted to be
\[
\nu = 1 \,\,(d=1), \quad \tfrac{3}{4}\,\, (d=2), \quad
0.588\dots\,\, (d=3), \quad \tfrac{1}{2}\,\, (d \geq 5).
\]
Thus, $\SAW$ is \emph{ballistic} in $d=1$, \emph{subballistic and superdiffusive} 
in $d=2,3,4$, and \emph{diffusive} in $d \geq 5$. 

For $d=1$ the above scaling is trivial. For $d \geq 5$ a proof has been given by 
Hara and Slade~\cite{HaSl92a,HaSl92b}. These two cases correspond to ballistic, 
respectively, diffusive behavior. The claim for $d=2,3,4$ is open.
\begin{itemize}
\item
For $d=2$ the scaling limit is predicted to be $\mathrm{SLE}_{8/3}$ (the Schramm 
Loewner Evolution with parameter $8/3$; see Fig.~\ref{fig-SAW}). 
\item
For $d=4$ a proof is under construction by Brydges and Slade (work in progress). 
\end{itemize}
See the lectures by Bauerschmidt, Duminil-Copin, Goodman and Slade~\cite{BaDuGoSl11},
Beffara~\cite{Be11} and Duminil-Copin and Smirnov~\cite{DuSm11} for more details.
$\SAW$ in $d \geq 5$ scales to Brownian motion,
\[
\left(\frac{1}{Dn^{1/2}}\,S_{\lfloor nt \rfloor}\right)_{0 \leq t \leq 1}
\Longrightarrow (B_t)_{0 \leq t \leq 1} \qquad \mbox{ as } n \to \infty,
\]
i.e., $\SAW$ is in the \emph{same universality class} as $\SRW$. Correspondingly, $d=4$ 
is called the upper critical dimension. The intuitive reason for the crossover at $d=4$ 
is that in low dimension long loops are dominant, causing the effect of the self-avoidance 
constraint in $\SAW$ to be long-ranged, whereas in high dimension short loops are dominant, 
causing it to be short-ranged. Phrased differently, since $\SRW$ in dimension $d \geq 2$ 
has Hausdorff dimension $2$, it tends to intersect itself frequently for $d<4$ and 
not so frequently for $d>4$. Consequently, the self-avoidance constraint in $\SAW$ 
changes the qualitative behavior of the path for $d<4$ but not for $d>4$.   


\subsection{Open problems}
\label{S1.8}

A version of $\SAW$ where self-intersections are not forbidden but are nevertheless
discouraged is called the \emph{weakly self-avoiding walk}. Here, $\cW_n$ is the same
as for $\SRW$, but $H_n(w)$ is chosen to be $\beta$ times the number of self-intersections
of $w$, with $\beta \in (0,\infty)$ a parameter referred to as the strength of 
self-repellence. It is predicted that the weakly self-avoiding walk is in the same 
universality class as $\SAW$ (the latter corresponds to $\beta=\infty$). This has been 
proved for $d=1$ and $d \geq 5$, but remains open for $d=2,3,4$. The scaling limit 
of the weakly self-avoiding walk in $d=2$ is again predicted to be $\mathrm{SLE}_{8/3}$, 
despite the fact that $\mathrm{SLE}_{8/3}$ does not intersect itself. The reason is 
that the self-intersections of the weakly self-avoiding walk typically occur close 
to each other, so that when the scaling limit is taken these self-intersections are
lost in the limit. This loss, however, does affect the \emph{time-parametrization} of 
the limiting $\mathrm{SLE}_{8/3}$, which is predicted to be $\beta$-dependent. It is 
a challenge to prove these predictions. For more details on $\mathrm{SLE}$, we refer to 
the lectures by Beffara~\cite{Be11}. 


\section{Polymer collapse}
\label{S2}

In this section we consider a polymer that receives a penalty for each 
\emph{self-intersection} and a reward for each \emph{self-touching}. This 
serves as a model of a polymer subject to screened van der Waals forces, 
or a polymer in a poor solvent. It will turn out that there are three 
phases: \emph{extended}, \emph{collapsed} and \emph{localized}.

An example is polystyrene dissolved in cyclohexane. At temperatures above 35 degrees 
Celsius the cyclohexane is a good solvent, at temperatures below 30 it is a poor solvent. 
When cooling down, the polystyrene collapses from a random coil to a compact ball
(see Fig.~\ref{fig-colpol}).

\begin{figure}[htbp]
\includegraphics[width=.40\hsize]{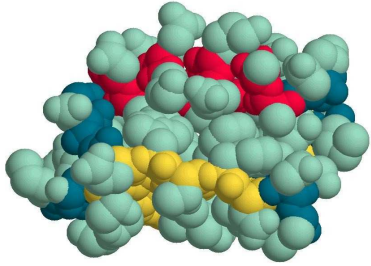}
\caption{\small A collapsed polymer.}
\label{fig-colpol}
\end{figure}

In Sections~\ref{S2.1}--\ref{S2.3} we consider a model with \emph{undirected paths}, 
in Sections~\ref{S2.4}--\ref{S2.5} a model with \emph{directed paths}. In 
Section~\ref{S2.6} we look at what happens when a force is applied to the 
endpoint of a collapsed polymer. In Section~\ref{S2.7} we formulate open
problems.


\subsection{An undirected polymer in a poor solvent}
\label{S2.1}

Our choice for the set of allowed paths and for the interaction Hamiltonian is 
\[
\begin{aligned}
\cW_n &= \big\{w=(w_i)_{i=0}^n\in(\Z^d)^{n+1}\colon\\
&\qquad\qquad\qquad w_0=0,\,\|w_{i+1}-w_i\|=1\,\,
\forall\,0 \leq i < n\big\},\\[0.4cm]
H_n^{\beta,\gamma}(w) &= \beta I_n(w) - \gamma J_n(w),
\end{aligned}
\]
where $\beta,\gamma\in (0,\infty)$, and
\[
\begin{aligned} 
I_n(w) &= \sum_{ {i,j=0} \atop {i<j} }^n 1_{\{\|w_i-w_j\|=0\}},\\
J_n(w) &= \tfrac{1}{2d} \sum_{ {i,j=0} \atop {i<j-1} }^n 1_{\{\|w_i-w_j\|=1\}},
\end{aligned}
\]
count the number of \emph{self-intersections}, respectively, \emph{self-touchings}
of $w$ (see Fig.~\ref{fig-polsintstouch}). The factor $\frac{1}{2d}$ is added to 
account for the fact that each site has $2d$ neighboring sites where the polymer
can achieve a self-touching. The path measure is
\[
P_n^{\beta,\gamma}(w) = \frac{1}{Z_n^{\beta,\gamma}}\,
\erom^{-H_n^{\beta,\gamma}(w)}\,P_n(w), \qquad w\in\cW_n,
\]
where $P_n$ is the law of the $n$-step $\SRW$ and $Z_n^{\beta,\gamma}$ is the 
normalizing partition sum. 

\begin{figure}[htbp]
\includegraphics[width=.60\hsize]{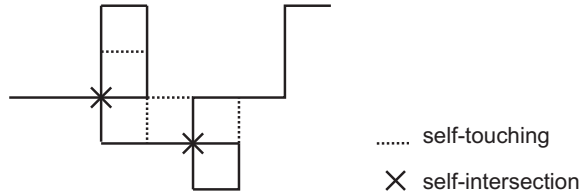}
\caption{\small A polymer with self-intersections and self-touchings.}
\label{fig-polsintstouch}
\end{figure}

Under the law $P_n^{\beta,\gamma}$, self-intersections are penalized while 
self-touchings are rewarded. The case $\gamma=0$ corresponds to weakly 
self-avoiding walk, which falls in the same universality class as $\SAW$ 
as soon as $\beta>0$ (recall Section~\ref{S1.8}). We expect that for 
$\beta\gg\gamma$ the polymer is a \emph{random coil}, while for $\gamma\gg\beta$ 
it is a \emph{compact ball}. A crossover is expected to occur when $\beta$ and
$\gamma$ are comparable. In the next two sections we identify two phase 
transition curves.


\subsection{The localization transition}
\label{S2.2}

For $L\in\N$, abbreviate $\Lambda(L) = [-L,L]^d \cap \Z^d$.

\begin{theorem}
\label{thm:infl}
{\rm [van der Hofstad and Klenke~\cite{vdHoKl01}]}
If $\beta>\gamma$, then the polymer is inflated, i.e., there exists an 
$\epsilon_0=\epsilon_0(\beta,\gamma)>0$ such that for all $0<\epsilon\leq
\epsilon_0$ there exists a $c=c(\beta,\gamma,\epsilon)>0$ such that
\[
P_n^{\beta,\gamma}\big(S_i \in \Lambda(\epsilon n^{1/d})
\,\,\forall\,0 \leq i \leq n\big)
\leq \erom^{-cn} \qquad \forall\,n\in\N.
\]
\end{theorem}

\begin{theorem}
\label{thm:loc}
{\rm [van der Hofstad and Klenke~\cite{vdHoKl01}]}
If $\gamma>\beta$, then the polymer is localized, i.e., there exist 
$c=c(\beta,\gamma)>0$ and $L_0 = L_0(\beta,\gamma) \in \N$ such that
\[
P_n^{\beta,\gamma}\big(S_i \in \Lambda(L)\,\,\forall\,0 \leq i \leq n\big)
\geq 1-\erom^{-cLn} \qquad \forall\,n\in\N,\,L \geq L_0.
\]
\end{theorem}

\noindent
Thus, at $\gamma=\beta$ a phase transition takes place, from a phase in which 
the polymer exits a box of size $n^{1/d}$ to a phase in which it is confined 
to a finite box. (In Section~\ref{S2.3} we will see that the inflated phase 
splits into two subphases: a collapsed phase and an extended phase.)

\begin{figure}[htbp]
\vspace{1cm}
\begin{center}
\setlength{\unitlength}{0.6cm}
\begin{picture}(7,6)(1,-0.5)
\put(0,0){\line(8,0){8}}
\put(0,0){\line(0,6){6}}
{\thicklines
\qbezier(0,0)(3,3)(6,6)
}
\put(-.7,-.7){$0$}
\put(8.5,-.2){$\beta$}
\put(-.2,6.5){$\gamma$}
\put(4,2){\tiny inflated}
\put(1,5){\tiny localized}
\end{picture}
\end{center}
\caption{\small Two phases: inflated and localized.}
\label{fig-inflloc}
\end{figure}
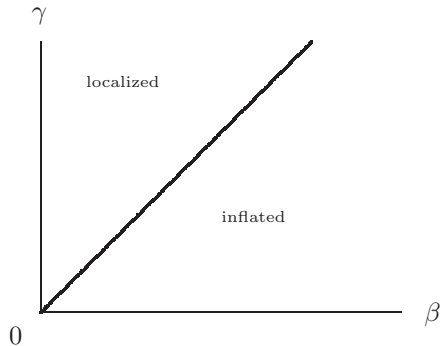

The main ideas behind the proof of Theorems~\ref{thm:infl}--\ref{thm:loc} are:
\begin{itemize}
\item[$\blacktriangleright$]
\emph{Inflated phase:} For $\epsilon$ small, most $n$-step paths that are 
folded up inside $\Lambda(\epsilon n^{1/d})$ have many self-intersections 
and many self-touchings. Since $\beta>\gamma$, the former produce more positive
energy than the latter produce negative energy, and so the total energy is
positive, making such paths unlikely. 
\item[$\blacktriangleright$]
\emph{Localized phase:} Two key ingredients are important:
\begin{itemize} 
\item[$\bullet$] 
An estimate showing that, since $\gamma>\beta$, the minimum of the Hamiltonian is 
achieved by a localized path. 
\item[$\bullet$] 
An estimate showing that, if $L$ is so large that $\Lambda(L)$ contains a minimizing 
path, then the penalty for leaving $\Lambda(L)$ is severe.
\end{itemize}
\end{itemize}
The proof uses a geometric argument based on \emph{folding of paths}, in the 
spirit of what is done in Section 2.1 of Bauerschmidt, Duminil-Copin, Goodman and 
Slade~\cite{BaDuGoSl11}. It is not known whether or not the minimizing path is 
unique modulo the symmetries of $\Z^d$. 

In terms of the mean-square displacement it is predicted that 
\[
E_n^{\beta,\gamma}(\|S_n\|^2) \asymp n^{2\nu} \quad \text{ as } n\to\infty,
\] 
where $\asymp$ stands for ``asymptotially the same modulo logarithmic factors'' (i.e., 
$E_n^{\beta,\gamma}(\|S_n\|^2)=n^{2\nu+o(1)}$). Theorems~\ref{thm:infl}--\ref{thm:loc} 
show that $\nu=0$ in the localized phase and $\nu \geq 1/d$ in the inflated phase. It 
is conjectured in van der Hofstad and Klenke~\cite{vdHoKl01} that on the critical line 
$\gamma=\beta$, 
\[
\nu = \nu_{\mathrm{loc}} = 1/(d+1).
\]
For $d=1$, this conjecture is proven in van der Hofstad, Klenke and 
K\"onig~\cite{vdHoKlKo02}. For $d \geq 2$ it is still open. The key 
simplification that can be exploited when $\beta=\gamma$ is the relation
\[
I_n(w)-J_n(w) = - \frac{n+1}{2} + \frac{1}{8d} \sum_{ \{x,y\} \in \Z^d\times\Z^d } 
|\ell_n(x) - \ell_n(y)|^2,  
\]
where the sum runs over all unordered pairs of neighboring sites, and $\ell_n(x) 
= \sum_{i=0}^n \ind_{\{w_i=x\}}$ is the local time of $w$ at site $x$. Since
the factor $-\frac{n+1}{2}$ can be absorbed into the partition sum, the model
at $\beta=\gamma$ effectively becomes a model where the energy is $\beta/4d$
times the sum of the squares of the gradients of the local times.  


\subsection{The collapse transition}
\label{S2.3}

It is predicted that there is a second phase transition at a critical value $\gamma_c
=\gamma_c(\beta)<\beta$ at which the inflated polymer moves from scale $n^{1/d}$ to 
scale $n^{\nu_\mathrm{SAW}}$, with $\nu_\mathrm{SAW}$ the critical exponent for $\SAW$. 
Thus, it is predicted that the inflated phase splits into two subphases: a \emph{collapsed 
phase} and an \emph{extended phase}, separated by a second critical curve at which a 
\emph{collapse transition} takes place. At the second critical curve, the critical 
exponent is predicted to be
\[
\nu = \nu_{\mathrm{coll}} = \left\{\begin{array}{ll}
\frac47, &\mbox{ if } d=2,\\[0.2cm]
\frac12, &\mbox{ if } d\geq 3.
\end{array}
\right.
\]

Thus, the phase diagram for $d \geq 2 $ is conjectured to have the shape in 
Fig.~\ref{fig-phdiacolpol}. The free energy is known to be $\infty$ in the localized 
phase, and is expected to lie in $(-\infty,0)$ in the two other phases. However, not 
even the existence of the free energy has been proven in the latter two phases.


Although these predictions are supported by heuristic theories (Duplantier and 
Saleur~\cite{DuSa87}, Seno and Stella~\cite{SeSt88}) and by extensive simulations, 
a mathematical proof of the existence of the collapse transition and a mathematical
verification of the values of the critical exponent have remained open for more than 20 
years. For $d=1$ there is no collapse transition because $\nu_\mathrm{SAW}=1$. Indeed, 
Theorem~\ref{thm:infl} says that below the critical line $\gamma=\beta$ the polymer 
is ballistic like $\SAW$. 

\begin{figure}[htbp]
\vspace{1cm}
\begin{center}
\setlength{\unitlength}{0.7cm}
\begin{picture}(7,6)(0,0)
\put(0,0){\line(8,0){8}}
\put(0,0){\line(0,6){6}}
{\thicklines
\qbezier(0,0)(3,3)(6,6)
\qbezier(0,0)(3,2.4)(7,2.7)
}
\qbezier[60](0,3)(4,3)(8,3)
\put(-.4,-.4){$0$}
\put(8.5,-.2){$\beta$}
\put(-.2,6.5){$\gamma$}
\put(-1,3){$\gamma_c$}
\put(1.5,4){$\nu=0$}
\put(1.5,3.5){$\mbox{\tiny localized}$} 
\put(5.5,4){$\nu=\frac1d$}
\put(5.5,3.5){$\mbox{\tiny collapsed}$}
\put(6.3,6.3){$\nu=\nu_{\mathrm{loc}} = \frac{1}{d+1}$}
\put(7.5,2.5){$\nu=\nu_{\mathrm{coll}}$}
\put(4.5,1){$\nu=\nu_{\mbox{\tiny SAW}}$}
\put(4.5,.5){$\mbox{\tiny extended}$}
\end{picture}
\end{center}
\caption{\small Conjectured phase diagram.}
\label{fig-phdiacolpol}
\end{figure}
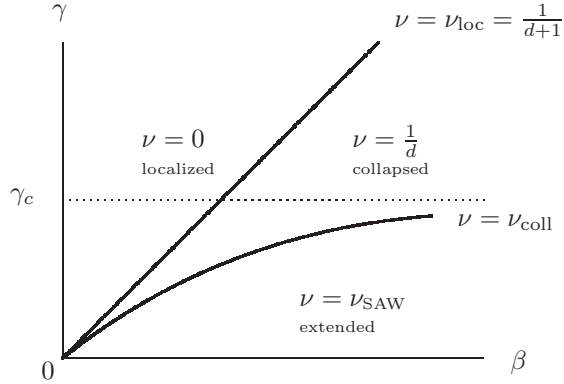

In $d=3$, simulations by Tesi, Janse van Rensburg, Orlandini and 
Whittington~\cite{TeJvReOrWh96} for $\SAW$ with attraction (corresponding 
to $\beta=\infty$ and $\gamma\in (0,\infty)$) yield $\gamma_c=\gamma_c(\infty) 
\in [0.274,0.282]$ and $\nu_\mathrm{coll}\in [0.48,0.50]$, the latter in 
accordance with the prediction mentioned above.


\subsection{A directed polymer in a poor solvent}
\label{S2.4}

In order to deal with the collapse transition mathematically, it is necessary 
to turn to a directed version of the model. The results to be described below 
are taken from Brak, Guttmann and Whittington~\cite{BrGuWh92}, with refinements 
carried out in various later papers.

Our choice for the set of allowed paths and the interaction Hamiltonian is
(see Fig.~\ref{fig-dircolpol})
\[
\begin{aligned} 
\cW_n &= 
\big\{w=(w_i)_{i=0}^n\in(\N_0\times\Z)^{n+1}\colon\\
&\qquad\qquad\qquad w_0=0,\,w_1-w_0 =\,\rightarrow,\\
&\qquad\qquad\qquad w_{i+1}-w_i \in \{\uparrow,\downarrow,\rightarrow\}
\,\,\forall\,0<i<n,\\
&\qquad\qquad\qquad w_i \neq w_j\,\,\forall\,0 \leq i < j \leq n\big\},\\[0.4cm]
H_n^\gamma(w) &= -\gamma J_n(w),
\end{aligned}
\]
where $\uparrow$, $\downarrow$ and $\rightarrow$ denote steps between neighboring 
sites in the north, south and east direction, respectively, $\gamma\in\R$ and
\[
J_n(w) = \sum_{ {i,j=0} \atop {i<j-1} }^n 1_{\{\|w_i-w_j\|=1\}}.
\]
The path measure is
\[
P_n^\gamma(w) = \frac{1}{Z_n^\gamma}\,\,\erom^{-H^\gamma_n(w)}, \qquad w \in \cW_n,
\]
with counting measure as the reference law (instead of the uniform measure $P_n$ used
in Sections~\ref{S2.1}--\ref{S2.3}) and with normalizing partition sum $Z_n^\gamma$. 
Thus, each self-touching is rewarded when $\gamma>0$ (= attractive) and penalized 
when $\gamma<0$ (= repulsive). Note that, because the path is self-avoiding ($I_n(w)=0$), 
the directed model is to be compared with the undirected model at $\beta=\infty$. Also 
note that the model lives in dimension $1+1$ and that no factor $\frac12$ is needed in 
front of the sum defining $J_n(w)$ because the path is directed. The choice that the 
first step of $w$ must be to the right is made for convenience only. (In the undirected 
model studied in Sections~\ref{S2.1}--\ref{S2.3} we did not consider the case $\gamma<0$ 
because of the presence of $\beta$.)

\begin{figure}[htbp]
\includegraphics[width=.30\hsize]{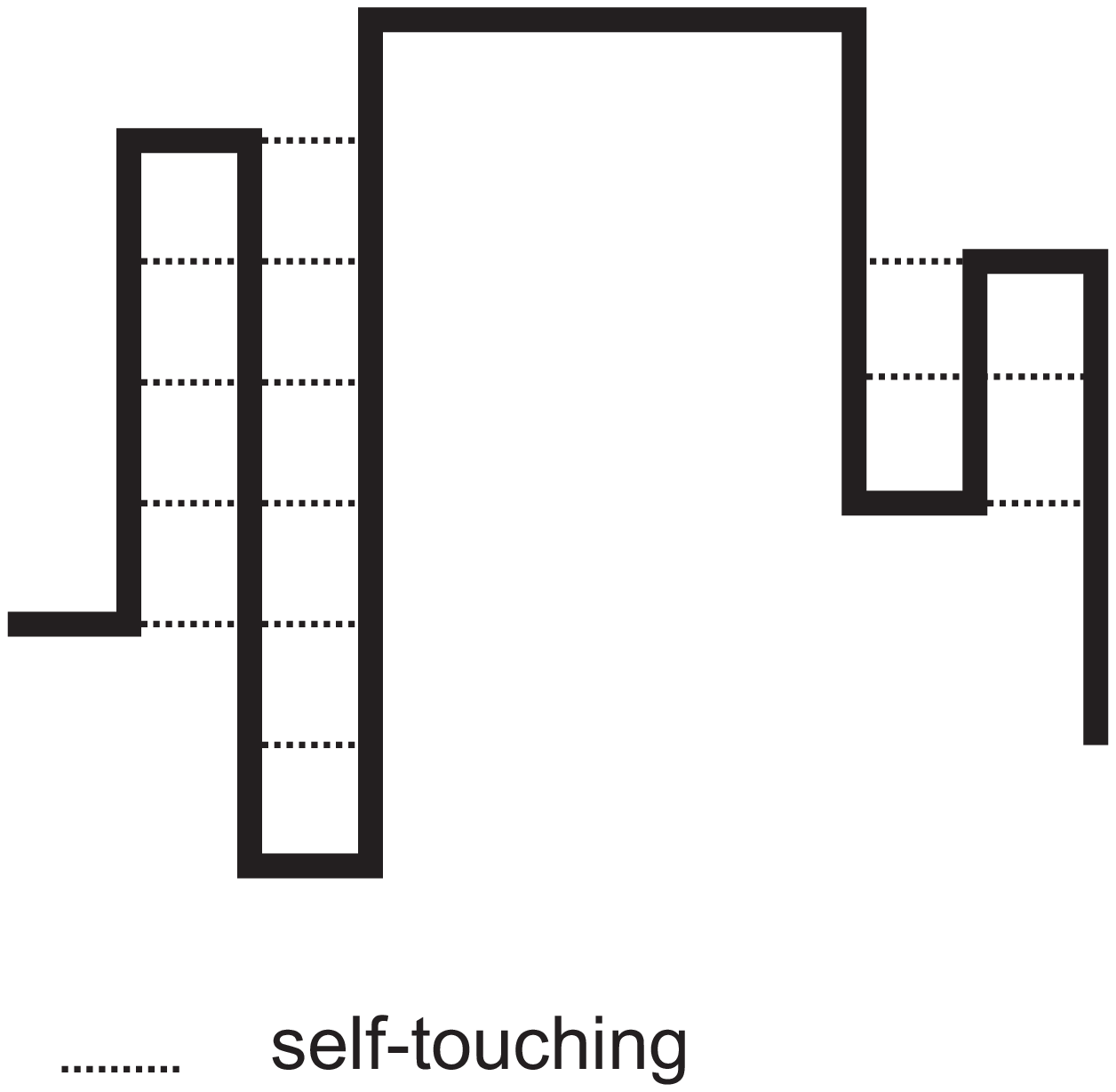}
\caption{\small A directed $\SAW$ with self-touchings.}
\label{fig-dircolpol}
\end{figure}


\subsection{Generating functions}
\label{S2.5}

The free energy of the directed polymer is given by
\[
f(\gamma)=\lim_{n\to\infty}\frac1n \log Z_n^\gamma,
\]
whenever the limit exists. The following theorem establishes existence and shows that 
there are two phases: a collapsed phase and an extended phase (see 
Fig.~\ref{fig-dircolpolphtr}).

\begin{theorem}
\label{thm:dircolpoltr}
{\rm [Brak, Guttmann and Whittington~\cite{BrGuWh92}]}
The free energy exists, is finite, and has a collapse transition at $\gamma_c=\log x_c$, 
with $x_c \approx 3.382975$ the unique positive solution of the cubic equation 
$x^3-3x^2-x-1=0$. The collapsed phase corresponds to $\gamma>\gamma_c$, the extended 
phase to $\gamma<\gamma_c$. 
\end{theorem}

\begin{figure}[htbp]
\vspace{-.3cm}
\begin{center}
\setlength{\unitlength}{0.6cm}
\begin{picture}(12,3)(0,-1.3)
\put(0,0){\line(12,0){12}}
\put(6,-.5){\line(0,1){1}}
\put(5.6,-1.3){$\gamma_c$}
\put(2,.5){extended}
\put(8,.5){collapsed}
\end{picture}
\end{center}
\caption{\small Collapse transition for the directed model.}
\label{fig-dircolpolphtr}
\end{figure}

Below we sketch the proof of Theorem~\ref{thm:dircolpoltr} in 5 Steps. The proof 
makes use of \emph{generating functions}. The details are worked out in 
{\bf Tutorial 2 in Appendix~\ref{appB}}. In Section~\ref{S3} we will encounter 
another model where generating functions lead to a full description of a 
phase transition. 

\medskip\noindent
{\bf 1.}
The partition sum $Z_n^\gamma=\sum_{w\in\cW_n} \erom^{\gamma J_n(w)}$ can 
be written as $Z_n^\gamma=Z_n(e^\gamma)$ with the power series
\[
Z_n(x) = \sum_{m\in\N_0} c_n(m) x^m, \quad x \in [0,\infty),\,n\in\N_0,
\]
where
\[
\begin{aligned}
c_n(m) &= |\{w\in\cW_n\colon\,J_n(w)=m\}|\\
&= \mbox{the number of $n$-step paths with $m$ self-touchings.}
\end{aligned} 
\]

\medskip\noindent
{\bf 2.} 
The existence of the free energy can be proved with the help of 
a subadditivity argument applied to the coefficients $c_n(m)$, 
based on concatenation of paths (as in Section 2 in Bauerschmidt,
Duminil-Copin, Goodman and Slade~\cite{BaDuGoSl11}.) 

\medskip\noindent
{\bf 3.} The finiteness of the free energy follows from the 
observation that $c_n(m)=0$ for $m \geq n$ and $\sum_{m=0}^\infty c_n(m) 
\leq 3^n$, which gives $f(\gamma) \leq \log[ 3(\erom^\gamma \vee 1)]
= \log 3 +(\gamma \vee 0)$.

\medskip\noindent
{\bf 4.} The following lemma gives a closed form expression for the 
generating function ($x=\erom^\gamma$)
\[
\begin{aligned}
G = G(x,y) &= \sum_{n\in\N_0} Z_n(x)\,y^n\\
&= \sum_{n\in\N_0} \Big[\sum_{m=0}^n c_n(m)\,x^m\Big]\,y^n,
\quad x,y \in [0,\infty). 
\end{aligned}
\]

\begin{lemma}
\label{lem:genfct}
For $x,y\in [0,\infty)$ the generating function is given by the formal power 
series
\[
G(x,y) = - \frac{aH(x,y)-2y^2}{bH(x,y)-2y^2},
\]
 where
\[
a = y^2(2+y-xy),\,\,b = y^2(1+x+y-xy),\,\,H(x,y) = y\,\frac{\bar{g}_0(x,y)}{\bar{g}_1(x,y)},
\]
 with
\[
\begin{aligned}
&\bar{g}_r(x,y) = y^r\,\left(1 + \sum_{k\in\N}
\frac{(y-q)^k\,y^{2k}\, q^{\frac12 k(k+1)}}
{\prod_{l=1}^k (yq^l-y)(yq^l-q)}\,q^{kr}\right),\\
&q=xy,\,r=0,1.
\end{aligned}
\]
\end{lemma}

\noindent
The function $H(x,y)$ is a quotient of two $q$-hypergeometric functions (which are singular at least along the curve 
$q=xy=1$). As shown in Brak, Guttmann and Whittington~\cite{BrGuWh92}, the latter can be expressed as continued fractions 
and therefore can be properly analyzed (as well as computed numerically).

\begin{figure}[htbp]
\vspace{1cm}
\begin{center}
\setlength{\unitlength}{0.6cm}
\begin{picture}(7,6)(0,0)
\put(0,0){\line(8,0){8}}
\put(0,0){\line(0,6){6}}
{\thicklines
\qbezier(5,2)(7,1.2)(8,1)
\qbezier(0,3.5)(2,2.5)(5,2)
}
\qbezier[20](5,0)(5,1)(5,2)
\qbezier[40](0.5,6)(2.5,3)(5,2.1)
\put(4.8,-.6){$x_c$}
\put(-.3,-.6){$0$}
\put(8.5,-.2){$x$}
\put(-.6,6.5){$y_c(x)$}
\put(5,2){\circle*{0.2}}
\put(0,3.5){\circle*{0.2}}
\end{picture}
\end{center}
\vspace{0.3cm}
\caption{\small The domain of convergence of the generating function $G(x,y)$ 
lies below the critical curve (= solid curve). The dotted line is the hyperbola 
$xy=1$ (corresponding to $q=1$). The point $x_c$ is identified with the collapse 
transition, because this is where the free energy is non-analytic.}
\label{fig-dircolpolcr}
\end{figure}
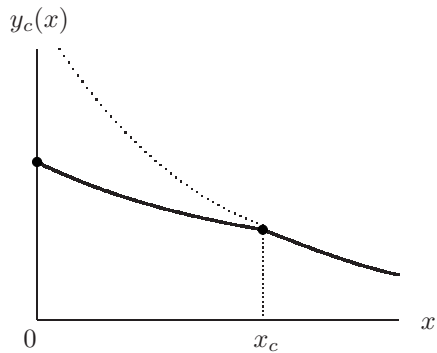

\medskip\noindent
{\bf 5.}
By analyzing the singularity structure of $G(x,y)$ it is possible to compute $f(\gamma)$. 
Indeed, the task is to identify the critical curve $x \mapsto y_c(x)$ in the $(x,y)$-plane 
below which $G(x,y)$ has no singularities and on or above which it does, because this 
identifies the free energy as
\[
f(\gamma) = - \log y_c(\erom^\gamma), \quad \gamma \in \R.
\]
It turns out that the critical curve has the shape given in Fig.~\ref{fig-dircolpolcr},
which implies that the free energy has the shape given in Fig.~\ref{fig-dircolpolfe}.

\begin{figure}[htbp]
\vspace{1cm}
\begin{center}
\setlength{\unitlength}{0.6cm}
\begin{picture}(7,6)(-3,-1.8)
\put(-6,0){\line(12,0){12}}
\put(-1,-2){\line(0,7){7}}
{\thicklines
\qbezier(1,2)(3,4)(4,5)
\qbezier(-5,.5)(-1,.5)(1,2)
}
\qbezier[20](-1,0)(0,1)(1,2)
\qbezier[20](1,2)(1,1)(1,0) 
\qbezier[40](-6,.4)(-3,.4)(-1,.4)
\put(.8,-.6){$\gamma_c$}
\put(-1.5,-.6){$0$}
\put(6.4,-.1){$\gamma$}
\put(-1.3,5.4){$f(\gamma)$}
\put(1,2){\circle*{0.25}}
\end{picture}
\end{center}
\caption{\small Plot of the free energy per monomer. The collapse transition 
occurs at $\gamma_c=\log x_c$. The limiting value at $\gamma=-\infty$ equals
$\log(1/y_c(0))$ with $y_c(0) \approx 0.453397$ the solution of the cubic 
equation $y^3+2y-1=0$, and is the entropy per step of the directed polymer 
that avoids self-touchings altogether, i.e., $\lim_{n\to\infty} \tfrac{1}{n}
\log c_n(0)$.}
\label{fig-dircolpolfe}
\end{figure}
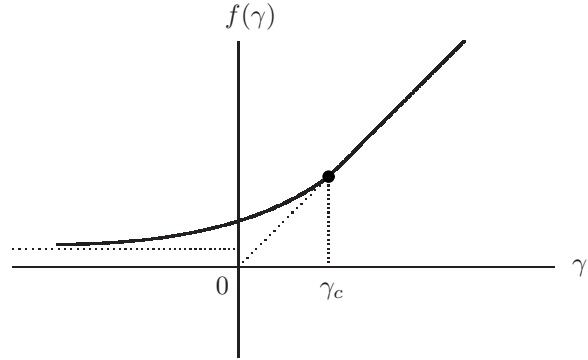

The derivative of the free energy is the limiting number of self-touchings 
per monomer, as plotted in Fig.~\ref{fig-dircolpolfeder}:
\[
f'(\gamma) = \lim_{n\to\infty} \frac{1}{n} \sum_{w\in\cW_n} J_n(w)\,P_n^\gamma(w). 
\]

\begin{figure}[htbp]
\vspace{1cm}
\begin{center}
\setlength{\unitlength}{0.6cm}
\begin{picture}(7,6)(-3,-2)
\put(-6,0){\line(12,0){12}}
\put(-1,-2){\line(0,6){6}}
{\thicklines
\qbezier(1,3)(3,3)(5,3)
\qbezier(-5,0.3)(0,1)(1,3)
}
\qbezier[20](1,3)(0,3)(-1,3)
\qbezier[30](1,3)(1,1.5)(1,0) 
\put(0.8,-.6){$\gamma_c$}
\put(-1.5,2.9){$1$}
\put(-1.5,-.6){$0$}
\put(6.4,-.1){$\gamma$}
\put(-1.3,4.4){$f'(\gamma)$}
\put(1,3){\circle*{0.25}}
\end{picture}
\end{center}
\caption{\small Plot of the number of self-touchings per monomer. Since $\gamma
\mapsto f'(\gamma)$ is continuous but not differentiable at $\gamma_c$, the phase 
transition is second order.}
\label{fig-dircolpolfeder}
\end{figure}
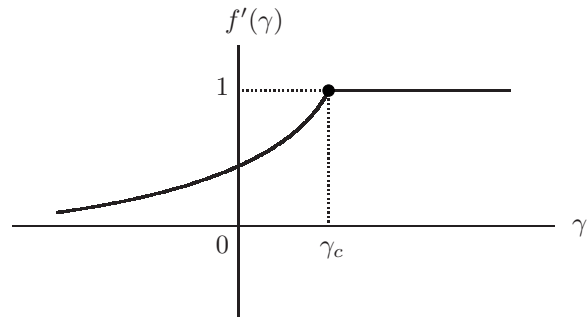


\subsection{Pulling at a collapsed polymer}
\label{S2.6}

It is possible to induce a collapse transition by applying a \emph{force} 
to the endpoint of a polymer rather than changing its interaction strength. 
The force can be applied, for instance, with the help of optical tweezers. 
A focused laser beam is used, containing a narrow region -- called the beam 
waist -- in which there is a strong electric field gradient. When a dielectric 
particle, a few nanometers in diameter, is placed in the waist, it feels a 
strong attraction towards the center of the waist. It is possible to chemically 
attach such a particle to the end of the polymer and then pull on the particle 
with the laser beam, thereby effectively exerting a force on the polymer itself. 
Current experiments allow for forces in the range of $10^{-12}-10^{-15}$ Newton. 
With such microscopically small forces the structural, mechanical and elastic 
properties of polymers can be probed. We refer to Auvray, Duplantier, Echard 
and Sykes~\cite{AuDuEcSy08}, Section 5.2, for more details. The force is the
result of transversal fluctuations of the dielectric particle, which can
be measured with great accuracy. 

Ioffe and Velenik~\cite{IoVe10,IoVe11,IoVepra,IoVeprb} consider a version of 
the undirected model in which the Hamiltonian takes the form
\[
H_n^{\psi,\phi}(w) = \sum_{x\in\Z^d} \psi\big(\ell_n(x)\big) - (\phi,w_n),
\qquad w \in \cW_n,
\] 
where $\cW_n$ is the set of allowed $n$-step paths for the \emph{undirected}
model considered in Sections~\ref{S2.1}--\ref{S2.3}, $\ell_n(x)=\sum_{i=0}^n 
\ind_{\{w_i=x\}}$ is the local time of $w$ at site $x \in \Z^d$, $\psi\colon\,
\N_0\to [0,\infty)$ is non-decreasing with $\psi(0)=0$, and $\phi\in\R^d$ is 
a \emph{force} acting on the endpoint of the polymer. Note that $(\phi,w_n)$ 
is the \emph{work} exerted by the force $\phi$ to move the endpoint of the 
polymer to $w_n$. The path measure is 
\[
P_n^{\psi,\phi}(w) = \frac{1}{Z_n^{\psi,\phi}}\,\erom^{-H_n^{\psi,\phi}(w)}\,P_n(w),
\qquad w\in\cW_n,
\] 
with $P_n$ the law of $\SRW$. 

Two cases are considered:
\begin{itemize}
\item[(1)] 
$\psi$ is \emph{superlinear} (= repulsive interaction).
\item[(2)] 
$\psi$ is \emph{sublinear} with $\lim_{\ell\to\infty} \psi(\ell)/\ell=0$ 
(= attractive interaction).
\end{itemize} 
Typical examples are:
\begin{itemize}
\item[(1)] 
$\psi(\ell)=\beta\ell^2$ (which corresponds to the weakly self-avoiding walk).
\item[(2)] 
$\psi(\ell) = \sum_{k=1}^\ell \beta_k$ with $k\mapsto\beta_k$ non-increasing 
such that $\lim_{k\to\infty} \beta_k =0$ (which corresponds to the annealed 
version of the model of a polymer in a random potential described in Section~\ref{S6}, 
for the case where the potential is non-negative).
\end{itemize} 
It is shown in Ioffe and Velenik~\cite{IoVe10,IoVe11,IoVepra,IoVeprb} (see also
references cited therein) that:
\begin{itemize}
\item[(1)] 
The polymer is in an extended phase for all $\phi \in \R^d$.
\item[(2)] 
There is a compact convex set $K=K(\psi)\subset\R^d$, with $\mathrm{int}(K) 
\ni 0$, such that the polymer is in a collapsed phase (= subballistic) when 
$\phi \in \mathrm{int}(K)$ and in an extended phase (= ballistic) when $\phi 
\notin K$. 
\end{itemize}
The proof uses \emph{coarse-graining arguments}, showing that in the extended 
phase large segments of the polymer can be treated as directed. For $d \geq 2$, 
the precise shape of the set $K$ is not known. It is known that $K$ has the 
symmetries of $\Z^d$ and has a locally analytic boundary $\partial K$ with a 
uniformly positive Gaussian curvature. It is predicted not to be a ball, but 
this has not been proven. The phase transition at $\partial K$ is first order.


\subsection{Open problems}
\label{S2.7}

The main challenges are:
\begin{itemize}
\item
Prove the conjectured phase diagram in Fig.~\ref{fig-phdiacolpol} for the undirected 
$(\beta,\gamma)$-model studied Sections~\ref{S2.1}--\ref{S2.3} and determine the order 
of the phase transitions. 
\item
Extend the analysis of the directed $\gamma$-model studied in Sections~\ref{S2.4}--\ref{S2.5} 
to $1+d$ dimensions with $d \geq 2$.
\item
Find a closed form expression for the set $K$ of the undirected $\psi$-model studied 
in Section~\ref{S2.6}. 
\end{itemize}

For the undirected model in $d=2$, the scaling limit is predicted to be:

\medskip\noindent
\begin{tabular}{lll}
&(1) &$\mathrm{SLE}_8$ in the collapsed phase (between the two critical curves),\\
&(2) &$\mathrm{SLE}_6$ at the collapse transition (on the lower critical curve),\\
&(3) &$\mathrm{SLE}_{8/3}$ in the extended phase (below the lower critical curve),
\end{tabular}

\medskip\noindent
all three with a time parametrization that depends on $\beta$ and $\gamma$ (see
the lectures by Beffara~\cite{Be11} for an explanation of the time parametrization). 
Case (1) is plausible because $\mathrm{SLE}_8$ is space filling, while we saw in 
Section~\ref{S2.2} that the polymer rolls itself up inside a ball with a volume 
equal to the polymer length. Case (2) is plausible because on the hexagonal lattice 
the exploration process in critical percolation has a path measure that, apart from 
higher order terms, is equal to that of the $\SAW$ with a critical reward for 
self-touchings (numerical simulation shows that $\gamma_c \approx \log 2.8$), and 
this exploration process has been proven to scale to $\mathrm{SLE}_6$ (discussions 
with Vincent Beffara and Markus Heydenreich). Case (3) is plausible because 
$\mathrm{SLE}_{8/3}$ is predicted to be the scaling limit of $\SAW$ (see 
Section~\ref{S1.7}).


\section{A polymer near a homogeneous interface}
\label{S3}

This section considers a polymer in the vicinity of a linear interface. Each 
monomer that touches the interface feels a \emph{binding energy}, resulting in 
an attractive interaction between the polymer and the interface. The focus is 
on the occurrence of a phase transition between a \emph{localized phase}, 
where the polymer stays close to the interface, and a \emph{delocalized phase}, 
where it wanders away from the interface (see Fig.~\ref{fig-hompin}). In 
Sections~\ref{S3.1}--\ref{S3.3} we look at the \emph{pinning} version of the 
model, where the polymer can move on both sides of the interface, and in 
Section~\ref{S3.4} at the \emph{wetting} version, where the polymer is constrained 
to stay on one side of the interface (which acts like a hard wall). In 
Sections~\ref{S3.5}--\ref{S3.6} we study how a pinned polymer can be pulled 
off an interface by applying a force to one of its endpoints. Section~\ref{S3.7} 
lists some open problems.

\begin{figure}[htbp]
\vspace{-3cm}
\includegraphics[width=.60\hsize]{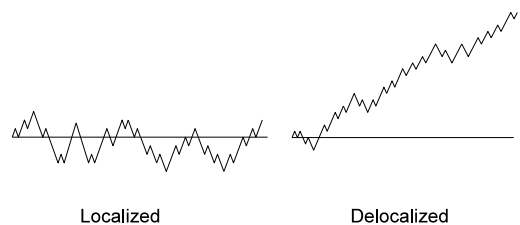}
\vspace{-5cm}
\caption{\small Path behavior in the two phases.}
\label{fig-hompin}
\end{figure}

Polymers are used as surfactants, foaming and anti-foaming agents, etc. The wetting
version of the model considered in the present section can be viewed as describing 
``paint on a wall''.


\subsection{Model}
\label{S3.1}

Our choices for the set of paths and for the interaction Hamiltonian are
\[
\begin{aligned}
\cW_n &= \big\{w=(i,w_i)_{i=0}^n\colon\,w_0=0,\,
w_i \in \Z \,\,\,\forall\,0 \leq i \leq n\big\},\\
H_n^\zeta(w) &= -\zeta L_n(w),
\end{aligned} 
\] 
with $\zeta\in\R$ and
\[
L_n(w) = \sum_{i=1}^n 1_{\{w_i=0\}}, \qquad w \in \cW_n,
\]
the \emph{local time of $w$ at the interface}. The path measure is 
\[
P_n^\zeta(w) = \frac{1}{Z_n^\zeta}\,e^{-H_n^\zeta(w)}\,P_n(w),
\qquad w\in\cW_n,
\]
where $P_n$ is the projection onto $\cW_n$ of the path measure $P$ of 
an \emph{arbitrary} directed irreducible random walk. This models a 
$(1+1)$-dimensional directed polymer in $\N_0 \times \Z$ in which 
each visit to the interface $\N\times\{0\}$ contributes an energy 
$-\zeta$, which is a reward when $\zeta>0$ and a penalty when $\zeta<0$
(see Fig.~\ref{fig-hompolexam}). 

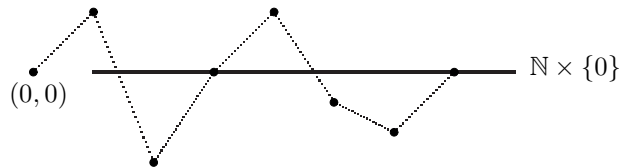
\begin{figure}[htbp]
\vspace{1cm}
\begin{center}
\setlength{\unitlength}{0.4cm}
\begin{picture}(0,0)(3,0)
{\thicklines
\qbezier(-4,0)(0,0)(10,0)
}
\qbezier[20](-6,0)(-5,1)(-4,2)
\qbezier[30](-4,2)(-3,-.5)(-2,-3)
\qbezier[30](-2,-3)(-1,-1.5)(0,0)
\qbezier[20](0,0)(1,1)(2,2)
\qbezier[25](2,2)(3,.5)(4,-1)
\qbezier[20](4,-1)(5,-1.5)(6,-2)
\qbezier[20](6,-2)(7,-1)(8,0)
\put(-6,0){\circle*{0.25}}
\put(-4,2){\circle*{0.25}}
\put(-2,-3){\circle*{0.25}}
\put(0,0){\circle*{0.25}}
\put(2,2){\circle*{0.25}}
\put(4,-1){\circle*{0.25}}
\put(6,-2){\circle*{0.25}}
\put(8,0){\circle*{0.25}}
\put(10.5,-.1){$\N \times \{0\}$}
\put(-6.8,-1){$(0,0)$}
\end{picture}
\end{center}
\vspace{1cm}
\caption{\small A $7$-step two-sided path that makes $2$ visits 
to the interface.}
\label{fig-hompolexam}
\end{figure}

Let $S=(S_i)_{i\in\N_0}$ denote the random walk with law $P$ 
starting from $S_0=0$. Let
\[
R(n) = P(S_i \neq 0\,\,\forall\, 1 \leq i < n,\,S_n=0), \qquad n\in\N.
\]
denote the \emph{return time distribution} to the interface. Throughout the 
sequel it is assumed that $\sum_{n\in\N} R(n)=1$ and
\[
R(n) = n^{-1-a}\,\ell(n), \qquad n\in\N,
\]
for some $a \in (0,\infty)$ and some $\ell(\cdot)$ slowly varying at infinity 
(i.e., $\lim_{x\to\infty} \ell(cx)/\ell(x)$ $=1$ for all $c\in (0,\infty)$). Note 
that this assumption implies that $R(n)>0$ for $n$ large enough, i.e., $R(\cdot)$ is 
aperiodic. It is trivial, however, to extend the analysis below to include the 
periodic case. $\SRW$ corresponds to $a=\frac12$ and period $2$.


\subsection{Free energy}
\label{S3.2}

The free energy can be computed explicitly. Let $\phi(x) = \sum_{n\in\N} 
x^n\,R(n)$, $x \in [0,\infty)$. 

\begin{theorem}
\label{thm:hompolfe} 
{\rm [Fisher~\cite{Fi84}, Giacomin~\cite{Gi07}, Chapter 2]}
The free energy
\[
f(\zeta) = \lim_{n\to\infty} \frac{1}{n} \log Z_n^\zeta
\] 
exists for all $\zeta\in\R$ and is given by
\[
f(\zeta) = \left\{\begin{array}{ll}
0, &\mbox{ if } \zeta \leq 0,\\
r(\zeta), &\mbox{ if } \zeta>0,
\end{array}
\right.
\]
where $r(\zeta)$ is the unique solution of the equation
\[
\phi(e^{-r}) = e^{-\zeta}, \qquad \zeta>0.
\]
\end{theorem}
 
\begin{proof}
For $\zeta \leq 0$, estimate
\[
\sum_{m>n} R(m) = P(S_i \neq 0\,\,\forall\,1\leq i\leq n) \leq Z_n^\zeta \leq 1,
\]
which implies $f(\zeta)=0$ because the left-hand side decays polynomially in $n$. 

For $\zeta>0$, let
\[
R^\zeta(n) = \erom^{\zeta-r(\zeta)n}\,R(n), \qquad n\in\mathbb{N}.
\]
By the definition of $r(\zeta)$, this is a probability distribution on $\mathbb{N}$, 
with a finite mean $M^\zeta=\sum_{n\in\N} nR^\zeta(n)$ because $r(\zeta)>0$. The 
partition sum when the polymer is constrained to end at $0$ can be written as
\[
Z_n^{*,\zeta} 
= \sum_{ {w\in\cW_n} \atop {w_n=0} } \erom^{\zeta L_n(w)}\,P_n(w)
= \erom^{r(\zeta)n}\,Q^\zeta(n \in T)
\]
with
\[
Q^\zeta(n \in T) = \sum_{m=1}^n 
\sum_{ {j_1,\dots,j_m \in \mathbb{N}} \atop {j_1+\cdots+j_m=n} }
\prod_{k=1}^m R^\zeta(j_k),
\]
where $T$ is the renewal process whose law $Q^\zeta$ is such that the i.i.d.\ 
renewals have law $R^\zeta$. Therefore, by the renewal theorem,
\[
\lim_{n\to\infty} Q^\zeta(n \in T) = 1/M^\zeta, 
\]
which yields
\[
\lim_{n\to\infty} \frac{1}{n}\,\log Z_n^{*,\zeta} = r(\zeta). 
\] 

By splitting the partition sum $Z_n^\zeta$ according to the last hitting time 
of $0$ (see the end of {\bf Tutorial 1 in Appendix~\ref{appA}}), it is straightforward 
to show that there exists a $C<\infty$ such that
\[
Z_n^{*,\zeta} \leq Z_n^\zeta \leq (1+Cn)Z_n^{*,\zeta} \qquad \forall\,n\in\N_0.
\]
It therefore follows that
\[
f(\zeta) = \lim_{n\to\infty} \frac{1}{n} \log Z_n^\zeta = r(\zeta).
\]
\end{proof}

For $\SRW$ (see Spitzer~\cite{Sp76}, Section 1)
\[
\phi(x)=1-\sqrt{1-x^2}, \qquad x \in [0,1].
\]
By Theorem~\ref{thm:hompolfe}, this gives
\[
f(\zeta) = r(\zeta) = \tfrac12\left[\zeta-\log(2-e^{-\zeta})\right], 
\quad 
f'(\zeta) = \tfrac12\left[1-\tfrac{\erom^{-\zeta}}{2-\erom^{-\zeta}}\right], 
\qquad \zeta > 0, 
\]
which is plotted in Fig.~\ref{fig-hompolfeSRW}.

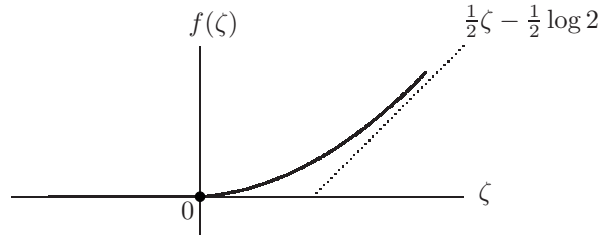
\begin{figure}[htbp]
\vspace{1cm}
\begin{center}
\setlength{\unitlength}{0.5cm}
\begin{picture}(7,6)(-2,-2)
\put(-6,0){\line(12,0){12}}
\put(-1,-1){\line(0,5){5}}
{\thicklines
\qbezier(-5,0)(-3,0)(-1,0)
\qbezier(-1,0)(2,0.2)(5,3.3)
}
\qbezier[40](2,0)(4,2)(6,4)
\put(-1.5,-.6){$0$}
\put(6.4,-.1){$\zeta$}
\put(-1.3,4.4){$f(\zeta)$}
\put(6,4.5){$\tfrac12\zeta-\tfrac12\log 2$}
\put(-1,0){\circle*{0.25}}
\end{picture}
\end{center}
\caption{\small Plot of the free energy for pinned $\SRW$.}
\label{fig-hompolfeSRW}
\end{figure}

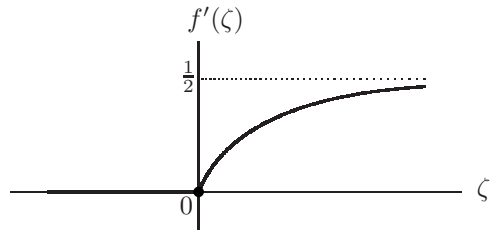
\begin{figure}[htbp]
\begin{center}
\setlength{\unitlength}{0.5cm}
\begin{picture}(7,6)(-3,-1.5)
\put(-6,0){\line(12,0){12}}
\put(-1,-1){\line(0,5){5}}
{\thicklines
\qbezier(-1,0)(0,2.5)(5,2.8)
\qbezier(-5,0)(-3,0)(-1,0)
}
\qbezier[40](-1,3)(1,3)(5,3) 
\put(-1.5,2.9){$\tfrac12$}
\put(-1.5,-.6){$0$}
\put(6.4,-.1){$\zeta$}
\put(-1.3,4.4){$f'(\zeta)$}
\put(-1,0){\circle*{0.25}}
\end{picture}
\end{center}
\caption{\small Plot of the average fraction of adsorbed monomers for 
pinned $\SRW$. The phase transition is second order.}
\label{fig-hompolabfr}
\end{figure}


\subsection{Path properties and order of the phase transition}
\label{S3.3}

\begin{theorem}
\label{thm:hompolpath}
{\rm [Deuschel, Giacomin and Zambotti~\cite{DeGiZa05}, Caravenna, Giacomin
and Zambotti~\cite{CaGiZa06}, Giacomin~\cite{Gi07}, Chapter 2]}
Under the law $P_n^\zeta$ as $n\to\infty$:\\
(a) If $\zeta>0$, then the path hits the interface with a strictly positive 
density, while the length and the height of the largest excursion away from 
the interface up to time $n$ are of order $\log n$.\\
(b) If $\zeta<0$, then the path hits the interface finitely often.\\
(c) If $\zeta=0$, then the number of hits grows like a power of $n$. 
\end{theorem}

\noindent
A detailed description of the path measure near the critical value is given in 
Sohier~\cite{So09}.

\begin{theorem}
\label{thm:hompolphtrord}
{\rm [Fisher~\cite{Fi84}, Giacomin~\cite{Gi07}, Chapter 2]}
There exists an $\ell^*(\cdot)$ slowly varying at infinity such that
\[
f(\zeta) = \zeta^{1/(1 \wedge a)}\,\ell^*(1/\zeta)\,[1+o(1)], 
\qquad \zeta \downarrow 0.
\]
\end{theorem}

\noindent
Theorem~\ref{thm:hompolphtrord} shows that, for all $m\in\N$, the order of the 
phase transition is $m$ when $a \in [\frac{1}{m},\frac{1}{m-1})$. For $\SRW$, 
$a=\frac12$ and the phase transition is second order (see Fig.~\ref{fig-hompolabfr}). 

The proof of Theorem~\ref{thm:hompolpath} depends on fine estimates of 
the partition sum, beyond the exponential asymptotics found in 
Theorem~\ref{thm:hompolfe}. The proof of Theorem \ref{thm:hompolphtrord} 
is given in {\bf Tutorial 3 in Appendix ~\ref{appC}}.


\subsection{Wetting}
\label{S3.4}

What happens when the interface is impenetrable? Then the set of paths is 
replaced by (see Fig.~\ref{fig-hompolexwet})
\[
\cW_n^+ = \big\{w=(i,w_i)_{i=0}^n\colon w_0=0,\,
w_i \in\N_0\,\,\forall\,0 \leq i \leq n\big\}.
\]
Accordingly, write $P_n^{\zeta,+}(w)$, $Z_n^{\zeta,+}$ and $f^+(\zeta)$ for the 
path measure, the partition sum and the free energy. One-sided pinning at an 
interface is called \emph{wetting}. 

\begin{figure}[htbp]
\vspace{1.5cm}
\begin{center}
\setlength{\unitlength}{0.5cm}
\begin{picture}(0,0)(3,0)
{\thicklines
\qbezier(-4,0)(0,0)(10,0)
}
\qbezier[20](-6,0)(-5,1)(-4,2)
\qbezier[20](-4,2)(-3,1.5)(-2,1)
\qbezier[20](-2,1)(-1,.5)(0,0)
\qbezier[30](0,0)(1,1.5)(2,3)
\qbezier[30](2,3)(3,1.5)(4,-0)
\qbezier[25](4,0)(5,1)(6,2)
\qbezier[20](6,2)(7,1.5)(8,1)
\put(-6,0){\circle*{0.25}}
\put(-4,2){\circle*{0.25}}
\put(-2,1){\circle*{0.25}}
\put(0,0){\circle*{0.25}}
\put(2,3){\circle*{0.25}}
\put(4,0){\circle*{0.25}}
\put(6,2){\circle*{0.25}}
\put(8,1){\circle*{0.25}}
\put(10.5,-.1){$\N \times \{0\}$}
\put(-6.8,-1){$(0,0)$}
\end{picture}
\end{center}
\vspace{0.5cm}
\caption{\small A $7$-step one-sided path that makes $2$ visits 
to the interface.}
\label{fig-hompolexwet}
\end{figure}
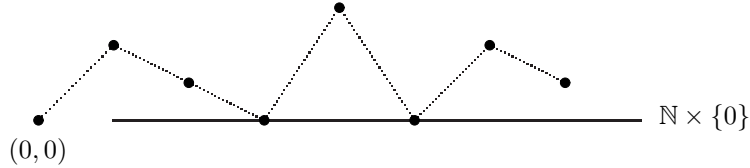

Let
\[
\label{b+def}
R^+(n) = P(S_i>0\,\,\forall\,1 \leq i < n,\,S_n=0), \qquad n\in\N.
\] 
This is a defective probability distribution. Define
\[
\phi^+(x) = \sum_{n\in\N} x^n\,R^+(n), \qquad  x \in [0,\infty),
\]
and put
\[
\widetilde \phi(x) = \frac{\phi^+(x)}{\phi^+(1)}, \qquad 
\zeta_c^+ = \log\Big[\frac{1}{\phi^+(1)}\Big] > 0.
\]

\begin{theorem}
\label{thm:hompolfewet}
{\rm [Fisher~\cite{Fi84}, Giacomin~\cite{Gi07}, Chapter 2]}
The free energy is given by
\[
f^+(\zeta) = \left\{\begin{array}{ll}
0, &\mbox{ if } \zeta \leq \zeta_c^+,\\
r^+(\zeta), &\mbox{ if } \zeta>\zeta_c^+,
\end{array}
\right.
\]
where $r^+(\zeta)$ is the unique solution of the equation
\[
\widetilde \phi(e^{-r}) = e^{-(\zeta-\zeta_c^+)}, \qquad \zeta>\zeta_c^+.
\]
\end{theorem}

The proof is similar to that of the pinned polymer. Localization on an 
impenetrable interface is harder than on a penetrable interface, because 
the polymer suffers a larger loss of entropy. This is the reason why 
$\zeta_c^+>0$. For $\SRW$, symmetry gives 
\[
R^+(n)=\tfrac12\,R(n), \qquad n\in\N.
\] 
Consequently, 
\[
\zeta_c^+=\log 2, \quad \widetilde \phi(\cdot) = \phi(\cdot),
\] 
implying that
\[
f^+(\zeta)=f(\zeta-\zeta_c^+), \qquad \zeta\in\R.
\]
Thus, the free energy suffers a shift (i.e., the curves in 
Figs.~\ref{fig-hompolfeSRW}--\ref{fig-hompolabfr} move to the right
by $\log 2$) and the qualitative behavior is similar to 
that of pinning.


\subsection{Pulling at an adsorbed polymer}
\label{S3.5}

A polymer can be pulled off an interface by a force. Replace the pinning Hamiltonian by
\[
H_n^{\zeta,\phi}(w) = -\zeta L_n(w) - \phi w_n,
\]
where $\phi\in (0,\infty)$ is a force in the upward direction acting on the 
endpoint of the polymer. Note that $\phi w_n$ is the work exerted by the force 
to move the endpoint a distance $w_n$ away from the interface. Write $Z_n^{\zeta,
\phi}$ to denote the partition sum and
\[
f(\zeta,\phi) = \lim_{n\to\infty} \frac{1}{n}\log Z_n^{\zeta,\phi}
\]
to denote the free energy. Consider the case where the reference random walk can 
only make steps of size 
$\leq 1$, i.e., pick $p \in [0,1]$ and put 
\[
P(S_1=-1)=P(S_1=+1)=\tfrac12 p, \quad P(S_1=0)= 1-p.
\]

\begin{theorem}
\label{thm:pinforcefe}
{\rm [Giacomin and Toninelli~\cite{GiTo07}]}
For every $\zeta\in\R$ and $\phi>0$, the free energy exists and is given by
\[
f(\zeta,\phi) = f(\zeta) \vee g(\phi), 
\]
with $f(\zeta)$ the free energy of the pinned polymer without force and 
\[
g(\phi) = \log\big[\,p\cosh(\phi)+(1-p)\,\big].
\]
\end{theorem}

\begin{proof}
Write
\[
Z_n^{\zeta,\phi} = Z_n^{*,\zeta} + \sum_{m=1}^n Z_{n-m}^{*,\zeta}\,\bar{Z}_m^\phi,
\]
where $Z_n^{*,\zeta}$ is the constrained partition sum without force encountered 
in Sections~\ref{S3.1}--\ref{S3.3}, and 
\[
\bar{Z}_m^\phi = \sum_{x\in\Z\backslash\{0\}} \erom^{\phi x}\,R(m;x), \qquad m\in\N,
\] 
with
\[
R(m;x) = P\big(S_i \neq 0\,\,\forall\,1 \leq i < m,\,S_m=x\big).
\]
It suffices to show that
\[
g(\phi) = \lim_{m\to\infty} \frac{1}{m}\,\log \bar{Z}_m^\phi, 
\]
which will yield the claim because
\[
f(\zeta) = \lim_{n\to\infty} \frac{1}{n}\log Z_n^{*,\zeta}.
\]

The contribution to $\bar{Z}_m^\phi$ coming from $x\in\Z\backslash\N_0$ is 
bounded from above by $1/(1-\erom^{-\phi})<\infty$ and therefore is 
negligible. (The polymer does not care to stay below the interface because 
the force is pulling it upwards.) For $x\in\N$ the \emph{reflection 
principle} gives
\[
\begin{aligned}
R(m;x) 
&= \tfrac12 p\,P\big(S_i>0\,\,\forall\,2 \leq i < m,\,S_m=x \mid S_1=1\big)\\
&= \tfrac12 p\,\big[P(S_m = x \mid S_1=1)-P(S_m = x \mid S_1=-1)\big]\\  
&= \tfrac12 p\,\big[P(S_{m-1}=x-1)-P(S_{m-1}=x+1)\big]
\qquad \forall\,m\in\N.
\end{aligned}
\]  
The first equality holds because the path cannot jump over the interface. 
The second inequality holds because, for any path from $1$ to $x$ that 
hits the interface, the piece of the path until the first hit of the 
interface can be reflected in the interface to yield a path from $-1$ 
to $x$. Substitution of the above relation into the sum defining 
$\bar{Z}_m^\phi$ gives
\[
\begin{aligned}
\bar{Z}_m^\phi &= O(1) + p \sinh(\phi) \sum_{x\in\mathbb{N}} 
\erom^{\phi x}\,P(S_{m-1}=x)\\ 
&= O(1) + O(1) + p \sinh(\phi)\,E\big(\erom^{\phi S_{m-1}}\big).
\end{aligned}
\]
But
\[
E\big(\erom^{\phi S_{m-1}}\big) = [p\cosh(\phi)+(1-p)]^{m-1},
\]
and so the above claim follows.
\end{proof}

The force either leaves most of the polymer adsorbed, when 
\[
f(\zeta,\phi)=f(\zeta)>g(\phi),
\]
or pulls most of the polymer off, when 
\[
f(\zeta,\phi)=g(\phi)>f(\zeta).
\]
A first-order phase transition occurs at those values of $\zeta$ and 
$\phi$ where $f(\zeta)=g(\phi)$, i.e., the critical value 
of the force is given by
\[
\phi_c(\zeta) = g^{-1}\big(f(\zeta)\big), \qquad \zeta \in \R,
\]
with $g^{-1}$ the inverse of $g$. Think of $g(\phi)$ as the free energy
of the polymer with force $\phi$ not interacting with the interface.


\subsection{Re-entrant force-temperature diagram}
\label{S3.6}

In order to analyze $\zeta\mapsto\phi_c(\zeta)$, we plot it as a function 
of temperature, putting
\[
\zeta = 1/T, \quad \phi = F/T, \quad F_c(T) = T\phi_c(1/T).
\]
It turns out that the curve $T \mapsto F_c(T)$ is increasing when $p\in 
(0,\tfrac23]$, but has a minimum when $p\in (\tfrac23,1)$. The latter 
behavior is remarkable, since it says that there is a force $F$ such 
that the polymer is adsorbed both for small $T$ and for large $T$, but 
is desorbed for moderate $T$ .

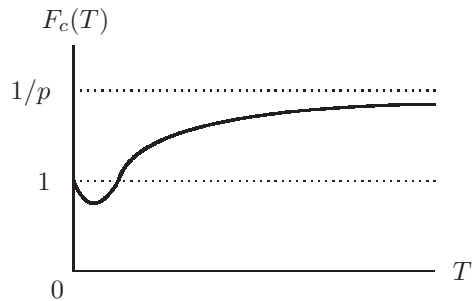
\begin{figure}[htbp]
\vspace{2cm}
\begin{center}
\setlength{\unitlength}{0.6cm}
\begin{picture}(0,0)(4,3.5)
\put(0,0){\line(8,0){8}}
\put(0,0){\line(0,5){5}}
{\thicklines
\qbezier(0,2)(0.4,1)(1,2)
\qbezier(1,2)(1.4,3.6)(8,3.7)
}
\qbezier[50](0,2)(4,2)(8,2)
\qbezier[50](0,4)(4,4)(8,4)
\put(-0.5,-.6){$0$}
\put(8.4,-.2){$T$}
\put(-.7,5.4){$F_c(T)$}
\put(-.8,1.8){$1$}
\put(-1.4,3.8){$1/p$}
\end{picture}
\end{center}
\vspace{2.5cm}
\caption{\small Re-entrant force-temperature diagram for $p \in (\tfrac23,1)$.}
\label{fig-reentrphd}
\end{figure}

For $p=\tfrac23$ all paths are equally likely, while for $p \in (\tfrac23,1)$ paths 
that move up and down are more likely than paths that stay flat. This leads to
the following heuristic explanation of the re-entrant behavior. For every $T$, the 
adsorbed polymer makes excursions away from the interface and therefore has a 
strictly positive entropy. Some of this entropy is lost when a force is applied 
to the endpoint of the polymer, so that the part of the polymer near the endpoint 
is pulled away from the interface and is caused to move upwards steeply. There are 
two cases: 

\medskip\noindent
$p=\tfrac23$:
As $T$ increases the effect of this entropy loss on the free energy increases, 
because ``$\mathrm{free\,energy} = \mathrm{energy}-\mathrm{temperature}\times
\mathrm{entropy}$''. This effect must be counterbalanced by a larger force to 
achieve desorption.

\medskip\noindent
$p\in(\tfrac23,1)$:
Steps in the east direction are favored over steps in the north-east and 
south-east directions, and this tends to place the adsorbed polymer farther 
away from the interface. Hence the force decreases for small $T$ (i.e., 
$F_c(T)<F_c(0)$ for small $T$, because at $T=0$ the polymer is fully adsorbed).


\subsection{Open problems}
\label{S3.7}

Some key challenges are:
\begin{itemize}
\item
Investigate pinning and wetting of $\SAW$ by a linear interface, i.e., study 
the undirected version of the model in Sections~\ref{S3.1}--\ref{S3.4}. Partial 
results have been obtained in the works of A.J.\ Guttmann, J.\ Hammersley, 
E.J.\ Janse van Rensburg, E.\ Orlandini, A.\ Owczarek, A.\ Rechnitzer, 
C.\ Soteros, C.\ Tesi, S.G.\ Whittington, and others. For references, see
den Hollander~\cite{dHo09}, Chapter 7. 
\item
Look at polymers living inside wedges or slabs, with interaction at the boundary. 
This leads to combinatorial problems of the type described in the lectures by 
Di Francesco during the summer school, many of which are hard. There is a large 
literature, with contributions coming from M.\ Bousquet-Melou, R.\ Brak, A.J.\ 
Guttmann, E.J.\ Janse van Rensburg, A.\ Owczarek, A.\ Rechnitzer, S.G.\ Whittington, 
and others. For references, see Guttmann~\cite{Gu09}. 
\item
Caravenna and P\'etr\'elis~\cite{CaPe09a,CaPe09b} study a directed polymer pinned 
by a periodic array of interfaces. They identify the rate at which the polymer 
hops between the interfaces as a function of their mutual distance and determine
the scaling limit of the endpoint of the polymer. There are several regimes
depending on the sign of the adsorption strength and on how the distance between 
the interfaces scales with the length of the polymer. Investigate what happens when 
the interfaces are placed at random distances.  
\item
What happens when the shape of the interface itself is random? Pinning of a polymer 
by a polymer, both performing directed random walks, can be modelled by the Hamiltonian
$H_n^\zeta(w,w')=-\zeta L_n(w,w')$, $\zeta\in\R$, with $L_n(w,w')=\sum_{i=1}^n 
1_{\{w_i=w_i'\}}$ the collision local time of $w,w'\in\cW_n$, the set of directed 
paths introduced in Section~\ref{S3.1}. This model was studied by Birkner, Greven 
and den Hollander~\cite{BiGrdHo11}, Birkner and Sun~\cite{BiSu10,BiSupr}, Berger and 
Toninelli~\cite{BeTo10}. A variational formula for the critical adsorption strength 
is derived in \cite{BiGrdHo11}. This variational formula turns out to be hard to analyze. 
\end{itemize}

\vspace{0.5cm}
\begin{center}
\begin{boxedminipage}{12cm}
In Sections~\ref{S1}--\ref{S3} we considered several models of a polymer 
chain interacting with itself and/or with an interface. In 
Sections~\ref{S4}--\ref{S6} we move to models with \emph{disorder}, i.e.,
there is a random environment with which the polymer chain is interacting. 
Models with disorder are much harder than models without disorder. In order
to advance mathematically, we will restrict ourselves to \emph{directed} paths. 
\end{boxedminipage}
\end{center}


\section{A polymer near a random interface}
\label{S4}

In this section we consider a directed polymer near a linear interface 
carrying ``\emph{random charges}''. As in Section~\ref{S3}, the polymer receives 
an energetic reward or penalty when it hits the interface, but this time the 
size of the reward or penalty is determined by disorder attached to the interface
(see Fig.~\ref{fig-polpin}). The goal is to determine under what conditions 
the disorder is able to pin the polymer to the interface.

In Sections~\ref{S4.1}--\ref{S4.2} we define the model. In 
Sections~\ref{S4.3}--\ref{S4.4} we use large deviation theory to derive 
a \emph{variational formula} for the critical curve separating a 
\emph{localized phase} from a \emph{delocalized phase}, both for the 
quenched and the annealed version of the model (recall part III of Section~\ref{S1.5}). 
In Section~\ref{S4.5} we use the two variational formulas to analyze under 
what conditions the two critical curves are different (= the disorder is 
relevant) or are the same (= the disorder is irrelevant). In Section~\ref{S4.6} 
we explain why denaturation of DNA is described by this model. In 
Section~\ref{S4.7} we close by formulating some open problems.

\begin{figure}[htbp]
\includegraphics[width=.60\hsize]{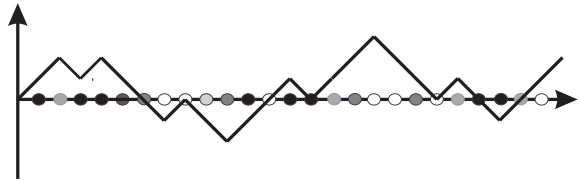}
\caption{\small Different shades represent different disorder values.}
\label{fig-polpin}
\end{figure}


\subsection{Model}
\label{S4.1}

Let $S=(S_n)_{n\in\N_0}$ be a recurrent Markov chain on a countable 
state space $\Upsilon$ with a marked point $\ast$. Write $P$ to denote the 
law of $S$ given $S_0=\ast$. Let 
\[
R(n) = P(S_i \neq \ast\,\,\forall\,1\leq i<n,\,S_n=\ast), \qquad n\in\N, 
\]
denote the \emph{return time distribution} to $\ast$, and assume that
\[
\lim_{n\to\infty} \frac{\log R(n)}{\log n} = -(1+a)
\quad \mbox{ for some } a\in [0,\infty).
\]
This is a weak version of the regularity condition assumed in Section~\ref{S3.1}
for the homogeneous pinning model.

Let 
\[
\omega=(\omega_i)_{i\in\N_0}
\] 
be an i.i.d.\ sequence of $\R$-valued random variables with marginal law $\mu_0$, 
playing the role of a random environment. Write $\mP=\mu_0^{\otimes\N_0}$ to 
denote the law of $\omega$. Assume that $\mu_0$ is non-degenerate and satisfies
\[
M(\beta) = \mE(\erom^{\beta\omega_0}) = \int_\R \erom^{\beta x} \mu(\drom x)
< \infty \qquad \forall\,\beta \geq 0.
\]

For fixed $\omega$, define a law on the set of directed paths of length 
$n\in\N_0$ by putting
\[
\frac{\drom P_n^{\beta,h,\omega}}{\drom P_n}\big((i,S_i)_{i=0}^n\big)
= \frac{1}{Z_n^{\beta,h,\omega}}\,\exp\left[\sum_{i=0}^{n-1}
(\beta\omega_i-h)\,1_{\{S_i=\ast\}}\right],
\]
where $\beta\in [0,\infty)$ is the \emph{disorder strength}, $h\in\R$ is the 
\emph{disorder bias}, $P_n$ is the projection of $P$ onto $n$-step paths, and 
$Z_n^{\beta,h,\omega}$ is the normalizing partition sum. Note that the 
homogeneous pinning model in Section~\ref{S3} is recovered by putting 
$\beta=0$ and $h=-\zeta$ (with the minor difference that now the Hamiltonian
includes the term with $i=0$ but not the term with $i=n$). Without loss of 
generality we can choose $\mu_0$ to be such that $\mE(\omega_0)=0$, $\mE(\omega_0^2)=1$ 
(which amounts to a shift of the parameters $\beta,h$). 

In our standard notation, the above model corresponds to the choice
\[
\begin{aligned}
\cW_n &= \Big\{w=(i,w_i)_{i=0}^n\colon\,w_0=\ast,\,
w_i\in\Upsilon\,\,\forall\,0<i\leq n\Big\},\\
H_n^{\beta,h,\omega}(w) &= - \sum_{i=0}^{n-1} (\beta\omega_i-h)\,1_{\{w_i=\ast\}}.
\end{aligned}
\]
(As before, we think of $(S_i)_{i=0}^n$ as the realization of $(w_i)_{i=0}^n$ 
drawn according to $P_n^{\beta,h,\omega}$.) The key example modelling our polymer 
with pinning is 
\[
\Upsilon=\Z^d, \quad \ast=\{0\}, \quad P = \mbox{ law of directed } 
\SRW \mbox{ in } \Z^d, \qquad d=1,2,
\]
for which $a=\tfrac12$ and $a=0$, respectively. We expect that pinning occurs 
for large $\beta$ and/or small $h$: the polymer gets a large enough energetic 
reward when it hits the positive charges and does not lose too much in terms
of entropy when it avoids the negative charges. For the same reason we expect 
that no pinning occurs for small $\beta$ and/or large $h$. In 
Sections~\ref{S4.2}--\ref{S4.6} we identify the phase transition curve and 
investigate its properties.


\subsection{Free energies}
\label{S4.2}

The \emph{quenched free energy} is defined as
\[
f^\mathrm{que}(\beta,h) = \lim_{n\to\infty} \frac{1}{n} \log Z_n^{\beta,h,\omega}
\quad \omega{\rm -a.s.}
\]
Subadditivity arguments show that $\omega$-a.s.\ the limit exists and is non-random 
(see {\bf Tutorial 1 in Appendix~\ref{appA}}). Since 
\[
Z_n^{\beta,h,\omega} = E\left(\exp\left[\sum_{i=0}^{n-1} (\beta\omega_i-h)\,
1_{\{S_i=*\}}\right]\right) \geq \erom^{\beta\omega_0-h} \sum_{m \geq n} R(m), 
\]
which decays polynomially in $n$, it follows that $f^\mathrm{que}(\beta,h) \geq 0$. 
This fact motivates the definition
\[
\begin{aligned}
\cL &= \big\{(\beta,h)\colon\,f^\mathrm{que}(\beta,h)>0\big\},\\
\cD &= \big\{(\beta,h)\colon\,f^\mathrm{que}(\beta,h)=0\big\},
\end{aligned}
\]
which are referred to as the \emph{quenched localized phase}, respectively, the 
\emph{quenched delocalized phase}. The associated \emph{quenched critical curve} is
\[
h_c^\mathrm{que}(\beta) = \inf\{h\in\R\colon\,f^\mathrm{que}(\beta,h)=0\},
\qquad \beta \in [0,\infty).
\]
Because $h \mapsto f^\mathrm{que}(\beta,h)$ is non-increasing, we have $f^\mathrm{que}
(\beta,h)=0$ for $h \geq h_c^\mathrm{que}(\beta)$. Convexity of $(\beta,h) \mapsto 
f^\mathrm{que}(\beta,h)$ implies that $\beta \mapsto h_c^\mathrm{que}(\beta)$ is convex. 
It is easy to check that both are finite (this uses the bound $f^\mathrm{que} \leq 
f^\mathrm{ann}$ with $f^\mathrm{ann}$ the annealed free energy defined below) and
therefore are also continuous. Futhermore, $h_c^\mathrm{que}(0)=0$ (because the 
critical threshold for the homogeneous pinning model is zero), and $h_c^\mathrm{que}
(\beta)> 0$ for $\beta>0$ (see below). Together with convexity the latter 
imply that $\beta \mapsto h_c^\mathrm{que}(\beta)$ is strictly increasing.

Alexander and Sidoravicius~\cite{AlSi06} prove that $h_c^\mathrm{que}
(\beta)>0$ for $\beta>0$ for arbitrary non-degenerate $\mu_0$ (see 
Fig.~\ref{fig-ranpincritcurve}). This result is important, because it 
shows that localization occurs even for a \emph{moderately negative} 
average value of the disorder, contrary to what we found for the homogeneous 
pinning model in Section~\ref{S3}. Indeed, since $\E(\beta\omega_1-h)=-h<0$, 
even a globally repulsive interface can locally pin the polymer provided
the global repulsion is modest: all the polymer has to do is hit the positive 
charges and avoid the negative charges.

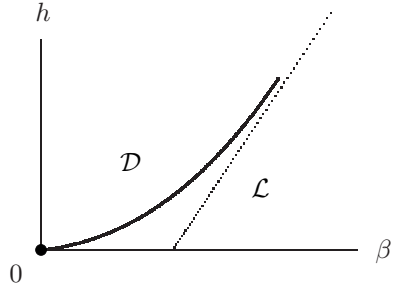
\begin{figure}[htbp]
\begin{center}
\setlength{\unitlength}{0.35cm}
\begin{picture}(12,12)(0,-2)
\put(0,0){\line(12,0){12}}
\put(0,0){\line(0,8){8}}
{\thicklines
\qbezier(0,0)(5,0.5)(9,6.5)
}
\qbezier[60](5,0)(7,3)(11,9)
\put(-1.2,-1.2){$0$}
\put(12.7,-0.3){$\beta$}
\put(-0.2,8.7){$h$}
\put(0,0){\circle*{.4}}
\put(8,2){$\cL$}
\put(3,3){$\cD$}
\end{picture}
\end{center}
\caption{\small Qualitative picture of $\beta \mapsto h_c^\mathrm{que}(\beta)$
(the asymptote has finite slope if and only if the support of $\mu_0$ is bounded 
from above). The details of the curve are known only partially (see below).}
\label{fig-ranpincritcurve}
\end{figure}

The \emph{annealed free energy} is defined by (recall Section~\ref{S1.5})
\[
f^\mathrm{ann}(\beta,h) =  \lim_{n\to\infty} \frac{1}{n}
\log \mE\big(Z_n^{\beta,h,\omega}\big).
\]
This is the free energy of a homopolymer. Indeed, $\E(Z_n^{\beta,h,\omega})
= Z_n^{h-\log M(\beta)}$, the partition function of the homogeneous pinning 
model with parameter $h-\log M(\beta)$. The associated \emph{annealed critical curve}
\[
h_c^\mathrm{ann}(\beta) = \inf\{h\in\R\colon\,f^\mathrm{ann}(\beta,h)=0\},
\qquad \beta \in [0,\infty),
\]
can therefore be computed explicitly:
\[
h_c^\mathrm{ann}(\beta) = \log \mE(\erom^{\beta\omega_0}) = \log M(\beta).
\]

By Jensen's inequality, we have
\[
f^\mathrm{que} \leq f^\mathrm{ann}
\quad \longrightarrow \quad
h_c^\mathrm{que} \leq h_c^\mathrm{ann}.
\]
In Fig.~\ref{fig-uncrittemp} below we will see how the two critical curves are 
related.

\begin{definition}
\label{def:disorderrel}
For a given choice of $R$, $\mu_0$ and $\beta$, the disorder is said to be
\emph{relevant} when $h_c^\mathrm{que}(\beta)<h_c^\mathrm{ann}(\beta)$ and 
\emph{irrelevant} when $h_c^\mathrm{que}(\beta)=h_c^\mathrm{ann}(\beta)$.
\end{definition}

\noindent
\emph{Note}: In the physics literature, the notion of relevant disorder is 
reserved for the situation where the disorder not only changes the critical 
value but also changes the behavior of the free energy near the critical 
value. In what follows we adopt the more narrow definition given above. 
It turns out, however, that for the pinning model considered here a change 
of critical value entails a change of critical behavior as well.

Some 15 papers have appeared in the past 5 years, containing sufficient conditions 
for relevant, irrelevant and marginal disorder, based on various 
types of estimates. Key references are: 
\begin{itemize}
\item
Relevant disorder: Derrida, Giacomin, Lacoin and Toninelli~\cite{DeGiLaTo09}, 
Alexander and Zygouras~\cite{AlZy08}.
\item
Irrelevant disorder: Alexander~\cite{Al08}, Toninelli~\cite{To08},
Lacoin~\cite{La10b}.
\item
Marginal disorder: Giacomin, Lacoin and Toninelli~\cite{GiLaTo10}. 
\end{itemize}
See also Giacomin and Toninelli~\cite{GiTo09}, Alexander and Zygouras~\cite{AlZy10},
Giacomin, Lacoin and Toninelli~\cite{GiLaTo11}. (The word ``marginal'' stands for 
``at the border between relevant and irrelevant'', and can be either relevant or
irrelevant.)  

In Sections~\ref{S4.4}--\ref{S4.6} we derive \emph{variational formulas} 
for $h_c^\mathrm{que}$ and $h_c^\mathrm{ann}$ and provide necessary and 
sufficient conditions on $R$, $\mu_0$ and $\beta$ for relevant disorder. 
The results are based on Cheliotis and den Hollander~\cite{ChdHo10}.
In Section~\ref{S4.3} we give a quick overview of the necessary tools
from large deviation theory developed in Birkner, Greven and den 
Hollander~\cite{BiGrdHo10}.


\subsection{Preparations}
\label{S4.3}

In order to prepare for the large deviation analysis in Section~\ref{S4.5},
we need to place the random pinning problem in a different context.

Think of $\omega=(\omega_i)_{i\in\N_0}$ as a random sequence of \emph{letters} 
drawn from the alphabet $\R$. Write $\cP^\mathrm{inv}(\R^{\N_0})$ to denote 
the set of probability measures on infinite letter sequences that are 
shift-invariant. The law $\mu_0^{\otimes\N_0}$ of $\omega$ is an element 
of $\cP^\mathrm{inv}(\R^{\N_0})$. A typical element of $\cP^\mathrm{inv}
(\R^{\N_0})$ is denoted by $\Psi$.

Let $\widetilde{\R} = \cup_{k\in\N}\,\R^k$. Think of $\widetilde{\R}$ 
as the set of \emph{finite words}, and of $\widetilde{\R}^\N$ as the set of 
\emph{infinite sentences}. Write $\cP^\mathrm{inv}(\widetilde{\R}^\N)$ to denote 
the set of probability measures on infinite sentences that are shift-invariant. 
A typical element of $\cP^\mathrm{inv}(\widetilde{\R}^\N)$ is denoted by $Q$.

The excursions of $S$ away from the interface cut out successive words from 
the random environment $\omega$, forming an infinite sentence (see 
Fig.~\ref{fig-wordcut}). Under the joint law of $S$ and $\omega$, this 
sentence has law $q_0^{\otimes\N}$ with
\[
q_0(\drom x_0,\ldots,\drom x_{k-1}) = R(k)\,\mu_0(\drom x_0) \times \dots 
\times \mu_0(\drom x_{k-1}), \qquad k\in\N,\,x_0,\ldots,x_{k-1} \in \R.
\]

\begin{figure}[htbp]
\includegraphics[width=.65\hsize]{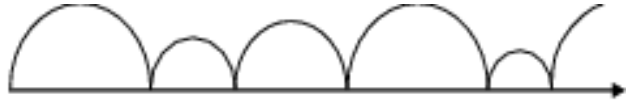}
\caption{\small Infinite sentence generated by $S$ on $\omega$.} 
\label{fig-wordcut}
\end{figure}

For $Q\in\cP^\mathrm{inv}(\widetilde{\R}^\N)$, let 
\[
\begin{aligned}
I^\mathrm{que}(Q) &= H\big(Q\,|\,q_0^{\otimes\N}\big) 
+ a\,m_Q\,H\big(\Psi_Q\,|\,\mu_0^{\otimes\N_0}\big),\\[0.2cm]
I^\mathrm{ann}(Q) &= H\big(Q\,|\,q_0^{\otimes\N}\big),
\end{aligned}
\]
where
\begin{itemize}
\item
$\Psi_Q\in\cP(\R^{\N_0})$ is the projection of $Q$ via concatenation 
of words; 
\item
$m_Q$ is the average word length under $Q$;
\item 
$H(\cdot|\cdot)$ denotes specific relative entropy.  
\end{itemize}
It is shown in Birkner, Greven and den Hollander~\cite{BiGrdHo10} that 
$I^\mathrm{que}$ and $I^\mathrm{ann}$ are the quenched and the annealed 
rate function in the \emph{large deviation principle} (LDP) for the 
\emph{empirical process of words}. More precisely,  
\[
\exp[-NI^\mathrm{que}(Q)+o(N)] \quad \mbox{ and } \quad 
\exp[-NI^\mathrm{ann}(Q)+o(N)]
\] 
are the respective probabilities that the first $N$ words generated by 
$S$ on $\omega$, periodically extended to form an infinite sentence, 
have an empirical distribution that is close to $Q \in \cP^\mathrm{inv}
(\widetilde{\R}^\N)$ in the weak topology. {\bf Tutorial 4 in 
Appendix~\ref{appD}} provides the background of this LDP.

The main message of the formulas for $I^\mathrm{que}(Q)$ and $I^\mathrm{ann}(Q)$
is that
\[
I^\mathrm{que}(Q) = I^\mathrm{ann}(Q) + \mbox{ an explicit extra term}. 
\]
We will see in Section~\ref{S4.4} that the extra term is crucial for
the distinction between relevant and irrelevant disorder.


\subsection{Application of the LDP}
\label{S4.4}

For $Q\in\cP^\mathrm{inv}(\widetilde{\R}^\N)$, let $\pi_{1,1} Q\in\cP(\R)$ 
denote the projection of $Q$ onto the first letter of the first word. Define 
$\Phi(Q)$ to be the average value of the first letter under $Q$,
\[
\Phi(Q)= \int_\R x\,(\pi_{1,1} Q)(\drom x),
\qquad Q \in \cP^\mathrm{inv}(\widetilde{\R}^\N),
\]
and $\cC$ to be the set
\[
\cC = \Big\{Q \in\cP^\mathrm{inv}(\widetilde{\R}^\N)\colon\,
\int_\R |x|\,(\pi_{1,1} Q)(\drom x) <\infty\Big\}.
\]
The following theorem provides variational formulas for the critical curves.

\begin{theorem}
\label{thm:varcharpin}
{\rm [Cheliotis and den Hollander~\cite{ChdHo10}]}
Fix $\mu_0$ and $R$. For all $\beta\in [0,\infty)$,
\[
\begin{aligned}
h_c^\mathrm{que}(\beta) &= \sup_{Q\in\cC}
[\beta\Phi(Q)-I^\mathrm{que}(Q)],\\[0.4cm]
h_c^\mathrm{ann}(\beta) &= \sup_{Q\in\cC}
[\beta\Phi(Q)-I^\mathrm{ann}(Q)].
\end{aligned}
\]
\end{theorem}

For $\beta \in [0,\infty)$, let
\[
\mu_\beta(\drom x) = \frac{1}{M(\beta)}\,
\erom^{\beta x}\,\mu_0(\drom x), \qquad x \in \R,
\]
and let  $Q_\beta = q_\beta^{\otimes\N} \in \cP^\mathrm{inv}(\widetilde{\R}^\N)$ 
be the law of the infinite sentence generated by $S$ on $\omega$ when the first 
letter of each word is drawn from the tilted law $\mu_\beta$ rather than $\mu_0$,
i.e.,
\[
q_\beta(\drom x_0,\ldots,\drom x_{n-1}) = R(n)\,\mu_\beta(\drom x_0) \times \dots 
\times \mu_0(\drom x_{n-1}), \qquad n\in\N,\,x_0,\ldots,x_{n-1} \in \R.
\]
It turns out that $Q_\beta$ is the \emph{unique maximizer} of the annealed variational 
formula. This leads to the following two theorems.
 
\begin{theorem}
\label{thm:disrelcrit}
{\rm [Cheliotis and den Hollander~\cite{ChdHo10}]}
Fix $\mu_0$ and $R$. For all $\beta\in [0,\infty)$,
\[
h_c^\mathrm{que}(\beta) < h_c^\mathrm{ann}(\beta) \quad \Longleftrightarrow \quad
I^\mathrm{que}(Q_\beta) > I^\mathrm{ann}(Q_\beta).
\]
\end{theorem}

\begin{theorem}
\label{thm:uncrittemp}
{\rm [Cheliotis and den Hollander~\cite{ChdHo10}]}
For all $\mu_0$ and $R$ there exists a $\beta_c=\beta_c(\mu_0,R) 
\in [0,\infty]$ such that
\[
h_c^\mathrm{que}(\beta) \left\{\begin{array}{ll}
= h_c^\mathrm{ann}(\beta) &\quad\mbox{ if } \beta \in [0,\beta_c],\\[0.4cm]
< h_c^\mathrm{ann}(\beta) &\quad\mbox{ if } \beta \in (\beta_c,\infty).\\
\end{array}
\right.
\]
\end{theorem}

\noindent
Theorem~\ref{thm:disrelcrit} gives a necessary and sufficient condition for 
relevant disorder, while Theorem~\ref{thm:uncrittemp} shows that relevant and 
irrelevant disorder are separated by a single critical temperature
(see Fig.~\ref{fig-uncrittemp}).

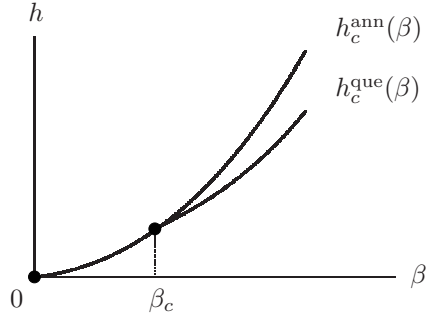
\begin{figure}[htbp]
\begin{center}
\setlength{\unitlength}{0.4cm}
\begin{picture}(12,12)(0,-1.5)
\put(0,0){\line(12,0){12}}
\put(0,0){\line(0,8){8}}
{\thicklines
\qbezier(0,0)(5,0.5)(9,7.5)
\qbezier(4,1.6)(7,3)(9,5.5)
}
\qbezier[20](4,0)(4,1)(4,1.6)
\put(-.8,-1){$0$}
\put(12.5,-0.2){$\beta$}
\put(-0.2,8.5){$h$}
\put(10,6){$h_c^\mathrm{que}(\beta)$}
\put(10,8){$h_c^\mathrm{ann}(\beta)$}
\put(3.8,-1){$\beta_c$}
\put(0,0){\circle*{.4}}
\put(4,1.6){\circle*{.4}}
\end{picture}
\end{center}
\caption{\small Uniqueness of the critical temperature $\beta_c$.}
\label{fig-uncrittemp}
\end{figure}


\subsection{Consequences of the variational characterization}
\label{S4.5}

Corollaries~\ref{cor:doscrit}--\ref{cor:nomasstop} give us control over $\beta_c$. 
Abbreviate $\chi=\sum_{n\in\N} [P(S_n=\ast)]^2$, i.e., the average number of
times two independent copies of our Markov chain $S$ meet at $\ast$.

\begin{coro}
\label{cor:doscrit}
{\rm [Cheliotis and den Hollander~\cite{ChdHo10}]}
(a) If $a=0$, then $\beta_c=\infty$ for all $\mu_0$.\\
(b) If $a\in (0,\infty)$, then, for all $\mu_0$, $\chi<\infty$ implies that
$\beta_c \in (0,\infty]$.
\end{coro}

\begin{coro}
\label{cor:bdscrittemp}
{\rm [Cheliotis and den Hollander~\cite{ChdHo10}]}
(a) $\beta_c \geq \beta_c^*$ with
\[
\beta_c^* = \sup\big\{\beta\in [0,\infty)\colon\,
M(2\beta)/M(\beta)^2 < 1+\chi^{-1}\big\}.
\]
(b) $\beta_c \leq \beta_c^{**}$ with
\[
\beta_c^{**} = \inf\big\{\beta \in [0,\infty)\colon\,h(\mu_\beta\mid\mu_0)>h(R)\big\},
\]
where $h(\cdot\mid\cdot)$ is relative entropy and $h(\cdot)$ is entropy.
\end{coro}

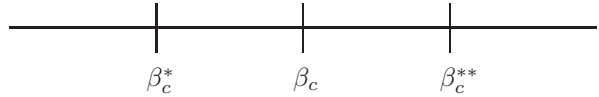
\begin{figure}[htbp]
\vspace{-.5cm}
\begin{center}
\setlength{\unitlength}{0.65cm}
\begin{picture}(12,3)(0,-1.3)
\put(0,0){\line(12,0){12}}
\put(6,-.5){\line(0,1){1}}
\put(3,-.5){\line(0,1){1}}
\put(9,-.5){\line(0,1){1}}
\put(5.8,-1.2){$\beta_c$}
\put(2.8,-1.2){$\beta_c^*$}
\put(8.8,-1.2){$\beta_c^{**}$}
\end{picture}
\end{center}
\caption{\small Bounds on $\beta_c$.}
\label{fig-bdscrittemp}
\end{figure}

\begin{coro}
\label{cor:nomasstop}
{\rm [Cheliotis and den Hollander~\cite{ChdHo10}]}
If $a \in (0,\infty)$, then $\beta_c \in [0,\infty)$ for all $\mu_0$ with
\[
\mu_0(\{w\})=0,
\] 
where $w = \mathrm{sup}[\mathrm{supp}(\mu_0)]$. 
\end{coro}

For the case where $R$ is regularly varying at infinity, i.e.,
\[
R(n) = n^{-(1+a)}\ell(n), \qquad n\in\N,
\] 
with $\ell(\cdot)$ slowly varying at infinity (which means that $\lim_{x\to\infty} 
\ell(cx)/\ell(x)=1$ for all $c \in (0,\infty)$), renewal theory gives
\[
P(S_n=\ast) \sim \left\{
\begin{array}{ll}
\frac{C}{n^{1-a}\ell(n)}, &a \in (0,1),\\
C, &a \in (1,\infty),\\
\ell^*(n), &a=1,
\end{array}
\right. \qquad n\to\infty,
\]
for some $C \in (0,\infty)$ and $\ell^*(\cdot)$ slowly varying at 
infinity. It therefore follows that $\chi<\infty$ if and only if
$a \in (0,\frac12)$ or $a=\frac12$, $\sum_{n\in\N} n^{-1}[\ell(n)]^{-2}
<\infty$.

A challenging open problem is the following conjecture, which has been proved under
more restrictive assumptions on $R$ (see Section~\ref{S4.7}). 

\begin{conjecture}
\label{conj:chiinfty}
{\rm [Cheliotis and den Hollander~\cite{ChdHo10}]}
If $a\in (0,\infty)$, then, for all $\mu_0$, $\chi=\infty$ implies that $\beta_c=0$.
\end{conjecture}

\medskip\noindent
\emph{Note}: The results in Theorem~\ref{thm:uncrittemp} and 
Corollaries~\ref{cor:doscrit}, \ref{cor:bdscrittemp} and \ref{cor:nomasstop} 
have all been derived in the literature by other means (see the references
cited at the end of Section~\ref{S4.2} and references therein). The point of 
the above exposition is to show that these results also follow in a natural
manner from a \emph{variational analysis} of the random pinning model, based 
on Theorems~\ref{thm:varcharpin} and \ref{thm:disrelcrit}.

\medskip
The following heuristic criterion, known as the \emph{Harris criterion}, applies
to the random pinning model. 
\begin{itemize}
\item[$\blacktriangleright$]
``Arbitrary weak disorder modifies the nature of a phase transition when 
the order of the phase transition in the non-disordered system is $<2$.''
\end{itemize}
Since, when $R$ is regularly varying at infinity, the order of the phase 
transition for the homopolymer is $<2$ when $a>\tfrac12$ and $\geq 2$ 
when $a\leq\tfrac12$ (see {\bf Tutorial 3 in Appendix~\ref{appC}}), the 
above results fit with this criterion. It is shown in Giacomin and
Toninelli~\cite{GiTo06c} that the disorder makes the phase transition smoother:
in the random pinning model the order of the phase transition is \emph{at
least two}, irrespective of the value of $a$.

At the critical value $a=\tfrac12$ the disorder can be \emph{marginally relevant} 
or \emph{marginally irrelevant}, depending on the choice of $\ell(\cdot)$.
See Alexander~\cite{Al08}, Giacomin, Lacoin and Toninelli~\cite{GiLaTo10}.


\subsection{Denaturation of DNA}
\label{S4.6}

DNA is a string of AT and CG base pairs forming a double helix: A and T share 
two hydrogen bonds, C and G share three. Think of the two strands as performing 
random walks in three-dimensional space subject to the restriction that they 
do not cross each other. Then the distance between the two strands is a random 
walk conditioned not to return to the origin. Since three-dimensional random 
walks are transient, this condition has an effect similar to that of a hard 
wall. 

This view of DNA is called the Poland-Sheraga model (see Fig.~\ref{fig-PolandSheraga}). 
The localized phase $\cL$ corresponds to the bounded phase of DNA, where the two 
strands are attached. The delocalized phase $\cD$ corresponds to the denaturated 
phase of DNA, where the two strands are detached. 

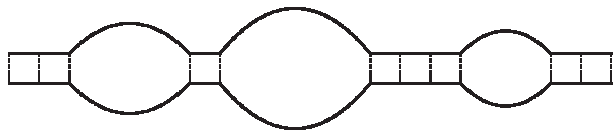
\begin{figure}[htbp]
\vspace{-.5cm}
\begin{center}
\setlength{\unitlength}{0.4cm}
\begin{picture}(18,5)(0,-1.5)
\qbezier[10](0,0)(0,0.5)(0,1)
\qbezier[10](1,0)(1,0.5)(1,1)
\qbezier[10](2,0)(2,0.5)(2,1)
{\thicklines
\qbezier(0,0)(1,0)(2,0)
\qbezier(0,1)(1,1)(2,1)
}
{\thicklines
\qbezier(2,0)(4,-2)(6,0)
\qbezier(2,1)(4,3)(6,1)
}
\qbezier[10](6,0)(6,0.5)(6,1)
\qbezier[10](7,0)(7,0.5)(7,1)
{\thicklines
\qbezier(6,0)(6.5,,0)(7,0)
\qbezier(6,1)(6.5,1)(7,1)
}
{\thicklines
\qbezier(7,0)(9.5,-3)(12,0)
\qbezier(7,1)(9.5,4)(12,1)
}
\qbezier[10](12,0)(12,0.5)(12,1)
\qbezier[10](13,0)(13,0.5)(13,1)
\qbezier[10](14,0)(14,0.5)(14,1)
\qbezier[10](15,0)(15,0.5)(15,1)
{\thicklines
\qbezier(12,0)(13.5,0)(15,0)
\qbezier(12,1)(13.5,1)(15,1)
}
{\thicklines
\qbezier(15,0)(16.5,-1.5)(18,0)
\qbezier(15,1)(16.5,2.5)(18,1)
}
\qbezier[10](18,0)(18,0.5)(18,1)
\qbezier[10](19,0)(19,0.5)(19,1)
\qbezier[10](20,0)(20,0.5)(20,1)
{\thicklines
\qbezier(18,0)(19,0)(20,0)
\qbezier(18,1)(19,1)(20,1)
}
\end{picture}
\end{center}
\caption{\small Schematic representation of the two strands of DNA in the 
Poland-Sheraga model. The dotted lines are the interacting base pairs, the 
loops are the denaturated segments without interaction.}
\label{fig-PolandSheraga}
\end{figure}

Since the order of the base pairs in DNA is irregular and their binding energies 
are different, DNA can be thought of as a polymer near an interface with 
\emph{binary disorder}. Of course, the order of the base pairs will not be 
i.i.d., but the random pinning model is reasonable at least for a qualitative 
description. Upon heating, the hydrogen bonds that keep the base pairs together 
can break and the two strands can separate, either partially or completely. This 
is called \emph{denaturation}. See Cule and Hwa~\cite{CuHw97}, Kafri, Mukamel 
and Peliti~\cite{KaMuPe00} for background.


\subsection{Open problems}
\label{S4.7}

Some key challenges are:
\begin{itemize}
\item
Provide the proof of Conjecture~\ref{conj:chiinfty}. The papers cited at the end of 
Section~\ref{S4.2} show that if $R$ is regularly varying at infinity (the condition 
mentioned below Corollary~\ref{cor:nomasstop}), then $\beta_c=0$ for $a\in(\tfrac12,
\infty)$, and also for $a=\tfrac12$ when $\ell(\cdot)$ does not decay too fast.
\item
Determine whether the phase transition is second order or higher order.
\item
Find sharp bounds for $\beta_c$, in particular, find a necessary and 
sufficient condition on $\mu_0$ and $R$ under which $\beta_c=\infty$ 
(i.e., the disorder is irrelevant for all temperatures). 
\item
Bolthausen, Caravenna and de Tili\`ere~\cite{BoCadTi09} apply a renormalization 
approach to random pinning. Develop this approach to study the critical curve.
\end{itemize}

P\'etr\'elis~\cite{Pe06} studies pinning at an interface with an internal structure. 
Information on the critical curve is hard to come by.


\section{A copolymer interacting with two immiscible fluids}
\label{S5}

A copolymer is a polymer consisting of different types 
of monomers. The order of the monomers is determined by the 
polymerization process through which the copolymer is grown.
This section looks at a $(1+1)$-dimensional directed copolymer, 
consisting of a random concatenation of hydrophobic and hydrophilic 
monomers, near a linear interface separating two immiscible solvents, 
oil and water, as depicted in Fig.~\ref{fig-copolex}.

\begin{figure}[htbp]
\includegraphics[width=.50\hsize]{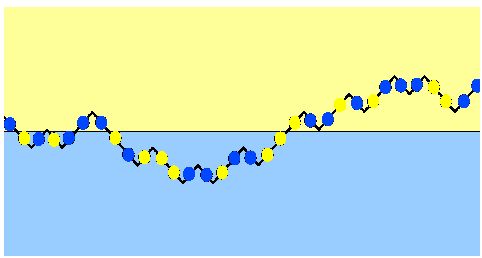}
\caption{\small A directed copolymer near a linear interface. Oil and
hydrophobic monomers are light-shaded, water and hydrophilic monomers
are dark-shaded.}
\label{fig-copolex}
\end{figure}

The copolymer has a tendency to stay close to the oil-water interface, in 
order to be able to place as many of its monomers in their preferred fluid. 
In doing so it lowers energy but loses entropy. A phase transition may be 
expected between a \emph{localized phase}, where the copolymer stays close 
to the interface, and a \emph{delocalized phase}, where it wanders away.
Which of the two phases actually occurs depends on the strengths of the 
chemical affinities. 

Copolymers near liquid-liquid interfaces are of interest due to their 
extensive application as surfactants, emulsifiers, and foaming or antifoaming 
agents. Many fats contain stretches of hydrophobic and hydrophilic monomers, 
arranged in some sort of erratic manner, and therefore are examples of random 
copolymers. (For the description of such systems, the undirected version of 
the model depicted in Fig.~\ref{fig-copolint} is of course more appropriate,
but we restrict ourselves to the directed version because this is mathematically 
much more tractable.) The transition between a localized and a delocalized 
phase has been observed experimentally, e.g.\ in neutron reflection studies of 
copolymers consisting of blocks of ethylene oxide and propylene oxide near a 
hexane-water interface. Here, a thin layer of hexane, approximately $10^{-5}\,
\mbox{m}$ thick, is spread on water. In the localized phase, the copolymer is 
found to stretch itself along the interface in a band of width approximately 
$20\,\mbox{\AA}$.

\begin{figure}[htbp]
\includegraphics[width=.40\hsize]{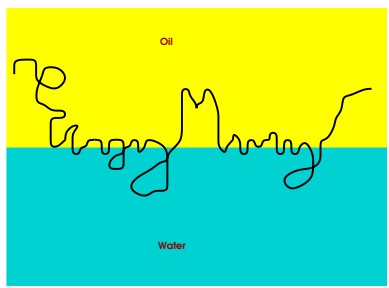}
\caption{\small An undirected copolymer near a linear interface. The 
disorder along the copolymer is not indicated.}
\label{fig-copolint}
\end{figure}

In Sections~\ref{S5.1}--\ref{S5.4} we define and study the copolymer model.
In Section~\ref{S5.5} we look at a version of the copolymer model where 
the linear interface is replaced by a \emph{random interface}, modelling 
a  micro-emulsion. Section~\ref{S5.6} lists some open problems. 


\subsection{Model}
\label{S5.1}

Let
\[
\cW_n = \big\{w=(i,w_i)_{i=0}^n\colon\,w_0=0,\,
w_{i+1}-w_i = \pm 1 \,\,\forall\,0\leq i<n\big\} 
\]
denote the set of all $n$-step directed paths that start from the origin 
and at each step move either north-east or south-east. Let
\[
\omega = (\omega_i)_{i\in\N} \mbox{ be i.i.d.\ with }
\mP(\omega_1=+1) = \mP(\omega_1=-1)=\tfrac12
\]
label the order of the monomers along the copolymer. Write $\mP$ 
to denote the law of $\omega$. The Hamiltonian, for fixed $\omega$, is
\[
H_n^{\beta,h,\omega}(w) = -\beta \sum_{i=1}^n (\omega_i+h)\,{\rm sign}(w_{i-1},w_i),
\qquad w \in \cW_n,
\]
with $\beta,h \in [0,\infty)$ the \emph{disorder strength}, respectively, the 
\emph{disorder bias} (the meaning of ${\rm sign}(w_{i-1},w_i)$ is explained below). 
The path measure, for fixed $\omega$, is
\[
P_n^{\beta,h,\omega}(w) = \frac{1}{Z_n^{\beta,h,\omega}}\,
\erom^{-H_n^{\beta,h,\omega}(w)}\,P_n(w), \qquad w \in \cW_n,
\]
where $P_n$ is the law of the $n$-step directed random walk, which is the uniform 
distribution on $\cW_n$. Note that $P_n$ is the projection on $\cW_n$ of the law 
$P$ of the infinite directed walk whose vertical steps are $\SRW$.

The interpretation of the above definitions is as follows: $\omega_i=+1$ or $-1$ stands 
for monomer $i$ being hydrophobic or hydrophilic; ${\rm sign}(w_{i-1},w_i)=+1$ or $-1$
stands for monomer $i$ lying in oil or water; $-\beta(\omega_i+h){\rm sign}(w_{i-1},
w_i)$ is the energy of monomer $i$. For $h=0$ both monomer types interact equally 
strongly, while for $h=1$ the hydrophilic monomers do not interact at all. Thus, 
only the regime $h \in [0,1]$ is relevant, and for $h>0$ the copolymer prefers the 
oil over the water.  

Note that the energy of a path is a sum of contributions coming from its 
\emph{successive excursions away from the interface} (this viewpoint was already
exploited in Section~\ref{S4} for the random pinning model). All that is relevant 
for the energy of the excursions is what stretch of $\omega$ they sample, and whether 
they are above or below the interface. The copolymer model is harder than the random
pinning model, because the energy of an excursion depends on the sum of the values 
of $\omega$ in the stretch that is sampled, not just on the first value. We expect 
the localized phase to occur for large $\beta$ and/or small $h$ and the delocalized 
phase for small $\beta$ and/or large $h$. Our goal is to identify the critical 
curve separating the two phases.


\subsection{Free energies}
\label{S5.2}

The quenched free energy is defined as
\[
f^\mathrm{que}(\beta,h) = \lim_{n\to\infty} \frac{1}{n} \log Z_n^{\beta,h,\omega}
\qquad \omega{\rm -a.s.}
\]
Subadditivity arguments show that $\omega$-a.s.\ the limit exists and is non-random 
for all $\beta,h \in [0,\infty)$ (see {\bf Tutorial 1 in Appendix~\ref{appA}}). 
The following lower bound holds:
\[
f^\mathrm{que}(\beta,h) \geq \beta h \qquad \forall\,\beta,h \in [0,\infty).
\]

\begin{proof}
Abbreviate
 \[
\Delta_i={\rm sign}(S_{i-1},S_i)
\]
and write
\[
\begin{aligned}
Z_n^{\beta,h,\omega} 
&= E\left(\exp\left[\beta\sum_{i=1}^n (\omega_i+h)\Delta_i\right]\right)\\
&\geq E\left(\exp\left[\beta\sum_{i=1}^n (\omega_i+h)\Delta_i\right]
\,1_{\{\Delta_i=+1\,\,\forall\,1 \leq i \leq n\}}\right)\\[0.2cm]
&= \exp\left[\beta\sum_{i=1}^n (\omega_i+h)\right]\,
P(\Delta_i=+1\,\,\forall\,1 \leq i \leq n)\\[0.3cm]
&=\exp[\beta h n + o(n) + O(\log n)] \qquad \omega{\rm -a.s.},
\end{aligned}
\]
where the last line uses the strong law of large numbers for $\omega$
and the fact that $P(\Delta_i=+1\,\,\forall\,1 \leq i \leq n) \geq C/n^{1/2}$
for some $C>0$.
\end{proof}

Put
\[
g^\mathrm{que}(\beta,h) = f^\mathrm{que}(\beta,h)-\beta h.
\]
The above proof shows that $g^\mathrm{que}(\beta,h)=0$ corresponds to the 
strategy where the copolymer wanders away from the interface in the 
upward direction. This fact motivates the definition
\[
\begin{aligned}
\cL &= \{(\beta,h)\colon\,g^\mathrm{que}(\beta,h) > 0\},\\
\cD &= \{(\beta,h)\colon\,g^\mathrm{que}(\beta,h) = 0\},
\end{aligned}
\]
referred to as the \emph{localized phase}, respectively, the \emph{delocalized 
phase}. The associated quenched critical curve is
\[
h_c^\mathrm{que}(\beta) = \inf\{h\in[0,\infty)\colon\,g^\mathrm{que}(\beta,h) = 0\}, 
\quad \beta \in [0,\infty).
\]
Convexity of $(\beta,t)\mapsto g^\mathrm{que}(\beta,t/\beta)$ implies that
$\beta\mapsto\beta h_c^\mathrm{que}(\beta)$ is convex. It is easy to
check that both are finite and therefore also continuous. Furthermore,
$h_c^\mathrm{que}(0)=0$ and $h_c^\mathrm{que}(\beta)>0$ for $\beta>0$
(see below). For fixed $h$, $\beta \mapsto g^\mathrm{que}(\beta,h)$ is convex 
and non-negative, with $g^\mathrm{que}(0,h)=0$, and hence is non-decreasing.
Therefore $\beta\mapsto h_c^\mathrm{que}(\beta)$ is non-decreasing as well. 
With the help of the convexity of $\beta\mapsto\beta h_c^\mathrm{que}(\beta)$, 
it is easy to show that $\beta\mapsto\beta h_c^\mathrm{que}(\beta)$ is strictly 
increasing (see Giacomin~\cite{Gi07}, Theorem 6.1). Moreover, $\lim_{\beta\to\infty} 
h_c^\mathrm{que}(\beta)=1$ (see below). A plot is given in Fig.~\ref{fig-hccritqual}.

\begin{figure}[htbp]
\vspace{-.5cm}
\begin{center}
\setlength{\unitlength}{0.4cm}
\begin{picture}(12,12)(0,-1)
\put(0,0){\line(12,0){12}}
\put(0,0){\line(0,8){8}}
{\thicklines
\qbezier(0,0)(1,4.5)(10,6)
}
\qbezier[60](0,7)(4,7)(10,7)
\put(-.8,-.8){$0$}  
\put(12.5,-0.2){$\beta$}
\put(-0.1,8.5){$h$}
\put(-0.8,6.7){$1$}
\put(0,0){\circle*{.4}}
\end{picture}
\end{center}
\caption{\small Qualitative picture of $\beta \mapsto h_c^\mathrm{que}(\beta)$.
The details of the curve are known only partially (see below).}
\label{fig-hccritqual}
\end{figure}
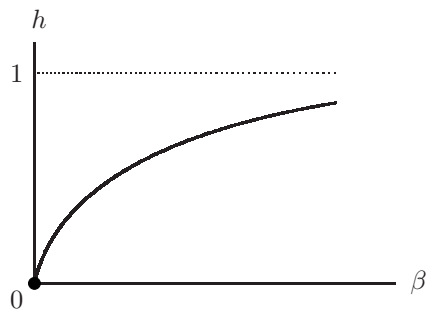

The following upper bound on the critical curve comes from an annealed 
estimate.

\begin{theorem}
\label{thm:hcub}
{\rm [Bolthausen and den Hollander~\cite{BodHo97}]}
$h_c^\mathrm{que}(\beta) \leq \frac{1}{2\beta}\,\log\cosh(2\beta)$ for all
$\beta \in (0,\infty)$.
\end{theorem}

\begin{proof}
Estimate
\[
\begin{aligned}
g^\mathrm{que}(\beta,h) 
&= \lim_{n\to\infty} \frac{1}{n}\,\mE\left(\log\left[\erom^{-\beta h n}
Z_n^{\beta,h,\omega}\right]\right)\\ 
&= \lim_{n\to\infty} \frac{1}{n}\, 
\mE\left(\log E\left(\exp\left[\beta \sum_{i=1}^n 
(\omega_i+h) (\Delta_i-1)\right]\right)\right)\\
&\leq \lim_{n\to\infty} \frac{1}{n} \log E\left(\mE\left(\exp\left[\beta \sum_{i=1}^n 
(\omega_i+h) (\Delta_i-1)\right]\right)\right)\\
&= \lim_{n\to\infty} \frac{1}{n} \log E \left(\prod_{i=1}^n
\left[\tfrac12 e^{-2\beta(1+h)} + 
\tfrac12 e^{-2\beta(-1+h)}\right]^{1_{\{\Delta_i=-1\}}}\right).
\end{aligned}
\]
The right-hand side is $\leq 0$ as soon as the term between square brackets 
is $\leq 1$. Consequently, 
\[
(2\beta)^{-1}\log\cosh(2\beta) < h \quad \longrightarrow 
\quad g^\mathrm{que}(\beta,h) = 0.
\]
\end{proof}
 
The following lower bound comes from strategies where the copolymer dips 
below the interface during rare long stretches in $\omega$ where the empirical 
mean is sufficiently biased downwards.    

\begin{theorem}
\label{thm:hclb}
{\rm [Bodineau and Giacomin\cite{BoGi04}]}
$h_c^\mathrm{que}(\beta) \geq (\frac43\beta)^{-1} \log\cosh(\frac43\beta)$ 
for all $\beta\in(0,\infty)$.
\end{theorem}

\begin{proof}
See {\bf Tutorial 5 in Appendix~\ref{appE}}.
\end{proof}

\noindent
Theorems~\ref{thm:hcub}--\ref{thm:hclb} are summarized in Fig.~\ref{fig-hcublb}.

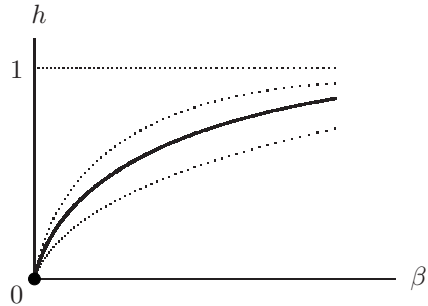
\begin{figure}[htbp]
\vspace{-1cm}
\begin{center}
\setlength{\unitlength}{0.4cm}
\begin{picture}(12,12)(0,-1)
\put(0,0){\line(12,0){12}}
\put(0,0){\line(0,8){8}}
{\thicklines
\qbezier(0,0)(1,4.5)(10,6)
}
\qbezier[60](0,0)(1,3)(10,5)
\qbezier[60](0,0)(1,6)(10,6.5)
\qbezier[60](0,7)(4,7)(10,7)
\put(-.8,-.8){$0$}  
\put(12.5,-0.2){$\beta$}
\put(-0.1,8.5){$h$}
\put(-0.8,6.7){$1$}
\put(0,0){\circle*{.4}}
\end{picture}
\end{center}
\caption{\small Upper and lower bounds on $\beta \mapsto h_c^\mathrm{que}(\beta)$.}
\label{fig-hcublb}
\end{figure}

Toninelli~\cite{To08a}, Toninelli~\cite{To08b}, Bodineau, Giacomin, Lacoin and 
Toninelli~\cite{BoGiLaTo08} show that the upper and lower bounds on $h_c^\mathrm{que}
(\beta)$ are strict. In fact, the strict inequalities can be extended to the setting 
considered in Section~\ref{S4}: arbitrary disorder with a finite moment-generating 
function and excursion length distributions that are regularly varying at infinity). 
Bolthausen, den Hollander and Opoku~\cite{BodHoOppr} derive a variational expression 
for $h_c^\mathrm{que}(\beta)$, similar in spirit to what was done in Section~\ref{S4.4}, 
and extend the strict inequalities to excursion length distributions that are 
logarithmically equivalent to a power law. 


\subsection{Weak interaction limit}
\label{S5.3}

\begin{theorem}
\label{thm:copolweak}
{\rm [Boltausen and den Hollander~\cite{BodHo97}]}
There exists a $K_c \in (0,\infty)$ such that
\[
\lim_{\beta \downarrow 0} \frac{1}{\beta}\,h_c^\mathrm{que}(\beta) = K_c.
\]
\end{theorem}

\noindent
The idea behind this result is that, as $\beta,h\downarrow 0$, the excursions 
away from the interface become longer and longer (entropy gradually takes over 
from energy). As a result, both $w$ and $\omega$ can be approximated by Brownian 
motions. In essence, the weak interaction result follows from the scaling property
\[
\lim_{\epsilon \downarrow 0} \epsilon^{-2}\,f^\mathrm{que}(\epsilon \beta,\epsilon h) 
= \widetilde f^\mathrm{que}(\beta,h), \qquad \beta,h \geq 0, 
\]
where $\widetilde f^\mathrm{que}(\beta,h)$ is the quenched free energy of a 
space-time continuous version of the copolymer model, with Hamiltonian
\[
H_t^{\beta,h,b}(B) = - \beta \int_0^t (\drom b_s+h\,\drom s)\,\,{\rm sign}(B_s) 
\]
and with path measure given by
\[
\frac{\drom P_t^{\beta,h,b}}{\drom P}(B) = \frac{1}{Z_t^{\beta,h,b}}\,
\erom^{-H_t^{\beta,h,b}(B)},
\]
where $B=(B_s)_{s\geq 0}$ is the polymer path, $P$ is the Wiener measure, and 
$b=(b_s)_{s\geq 0}$ is a Brownian motion that plays the role of the quenched 
disorder. The proof is based on a \emph{coarse-graining argument}. Due to the presence 
of exponential weight factors, the above scaling property is much more delicate 
than the standard invariance principle relating $\SRW$ and Brownian motion.

For the continuum model, a standard scaling argument shows that the quenched critical 
curve is linear. Its slope $K_c$ is not known and has been the subject of heated debate. 
The bounds in Theorems~\ref{thm:hcub}--\ref{thm:hclb} imply that $K_c \in [\frac23,1]$. 
Toninelli~\cite{To08b} proved that $K_c<1$. Caravenna, Giacomin and Gubinelli~\cite{CaGiGu06} 
did simulations and found that $K_c \in [0.82,0.84]$. Moreover, Caravenna, Giacomin and 
Gubinelli~\cite{CaGiGu06} and Sohier (private communication) found that
\[
h_c^\mathrm{que}(\beta) \approx \frac{1}{2K_c\beta}\,\log\cosh(2K_c\beta)
\]
is a good approximation for small and moderate values of $\beta$.

The Brownian model describes a continuum copolymer where each infinitesimal
element has a random degree of ``hydrophobicity'' or ``hydrophilicity''. It 
turns out that the continuum model is the scaling limit of a whole class of 
discrete models (see Caravenna and Giacomin~\cite{CaGi10}, Caravenna, Giacomin 
and Toninelli~\cite{CaGiTopr}), i.e., there is \emph{universality}. This property 
actually holds for a one-parameter family of continuum models indexed by a tail 
exponent $a \in (0,1)$, of which the Brownian copolymer is the special case 
corresponding to $a=\tfrac12$. It is known that the above approximation of the 
critical curve is not an equality in general. Bolthausen, den Hollander and 
Opoku~\cite{BodHoOppr} obtain sharp upper and lower bounds on $K_c$.

A related coarse-graining result is proved in P\'etr\'elis~\cite{Pe09} for a copolymer 
model with additional random pinning in a finite layer around the interface (of the 
type considered in Section~\ref{S4}). It is shown that the effect of the disorder 
in the layer vanishes in the weak interaction limit, i.e., only the disorder 
along the copolymer is felt in the weak interaction limit.


\subsection{Qualitative properties of the phases}
\label{S5.4}

We proceed by stating a few path properties in the two phases. 

\begin{theorem}
\label{thm:pathcopol}
{\rm [Biskup and den Hollander~\cite{BidHo99}, Giacomin and Toninelli~\cite{GiTo05,GiTo06d}]}
(a) If $(\beta,h) \in \cL$, then the path intersects the interface with a 
strictly positive density, while the length and the height of the largest 
excursion away from the interface up to time $n$ is order $\log n$.\\
(b) If $(\beta,h) \in {\rm int}(\cD)$, then the path intersects the interface 
with zero density. The number of intersections is $O(\log n)$. 
\end{theorem}

\noindent
For $(\beta,h) \in {\rm int}(\cD)$, the number of intersections is expected to be $O(1)$ 
under the average quenched path measure (see Part III of Section~\ref{S1.5}). So far this
has only been proved for $(\beta,h)$ above the annealed upper bound.

\begin{theorem}
\label{thm:ordphtr}
{\rm [Giacomin and Toninelli~\cite{GiTo06b,GiTo06c}]}
For every $\beta\in (0,\infty)$, 
\[
0 \leq g^\mathrm{que}(\beta,h) = O\left( [h_c^\mathrm{que}(\beta)-h]^2 \right) 
\qquad \mbox{ as } h \uparrow h_c^\mathrm{que}(\beta).
\]
\end{theorem}

\begin{theorem}
\label{thm:infdiff}
{\rm [Giacomin and Toninelli~\cite{GiTo06d}]}
$(\beta,h) \mapsto f^\mathrm{que}(\beta,h)$ is infinitely differentiable 
on $\cL$. 
\end{theorem}

\noindent
Theorem~\ref{thm:ordphtr} says that the phase transition is \emph{at least 
second order}, while Theorem~\ref{thm:infdiff} says that the critical curve 
is the only location where a phase transition of \emph{finite order} occurs.
Theorem~\ref{thm:ordphtr} is proved in {\bf Tutorial 5 in Appendix~\ref{appE}}. 

\medskip
All of the results in Sections~\ref{S5.2}--\ref{S5.4} extend to $\omega_i\in\R$ 
rather than $\omega_i\in\{-1,+1\}$, provided the law of $\omega_i$ has a finite 
moment-generating function, and to more general excursion length distributions, 
of the type considered in Section~\ref{S4.1}. For an overview, see Caravenna,
Giacomin and Toninelli~\cite{CaGiTopr}.


\subsection{A copolymer in a micro-emulsion}
\label{S5.5}

What happens when the linear interface is replaced by a \emph{random interface}? 
In particular, what happens when the oil forms droplets that float around in the 
water, as in Fig.~\ref{fig-copolemul}? An example is milk, which is a micro-emulsion
consisting (among others) of water and tiny fat-droplets. Milk is stabilized by 
a protein called casein, a copolymer that wraps itself around the droplets and 
prevents them to coagulate.  

\begin{figure}[htbp]
\includegraphics[width=.40\hsize]{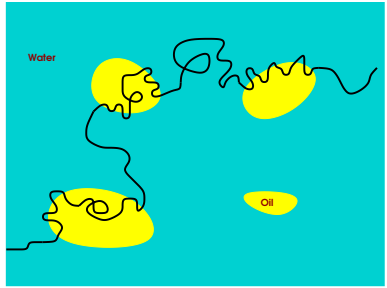}
\caption{\small An undirected copolymer in an emulsion. The disorder along
the copolymer is not indicated.}
\label{fig-copolemul}
\end{figure}

A phase transition may be expected between a \emph{localized phase}, where the 
copolymer spends most of its time near the boundary of the droplets and makes 
rapid hops from one droplet to the other, and a \emph{delocalized phase}, where 
it spends most of its time inside and outside of droplets. We will see that the 
actual behavior is rather more complicated. This is due to the fact that there 
are three (!) types of randomness in the model: a random polymer path, a random 
ordering of monomer types, and a random arrangement of droplets in the emulsion.

Here is a quick definition of a model. Split $\Z^2$ into square blocks of size 
$L_n$. The copolymer follows a directed self-avoiding path that is allowed to 
make steps $\uparrow,\downarrow,\rightarrow$ and to enter and exit blocks at 
diagonally opposite corners (see Fig.~\ref{fig-copolemulcorner}). Each monomer 
has probability $\frac12$ to be hydrophobic and probability $\tfrac12$ to be 
hydrophilic, labeled by $\omega$. Each block has probability $p$ to be filled 
with oil and probability $1-p$ to be filled with water, labeled by $\Omega$. 
Assign energies $-\alpha$ and $-\beta$ to the matches hydrophobic/oil, respectively, 
hydrophilic/water and energy $0$ to the mismatches.

\begin{figure}[htbp]
\includegraphics[width=.40\hsize]{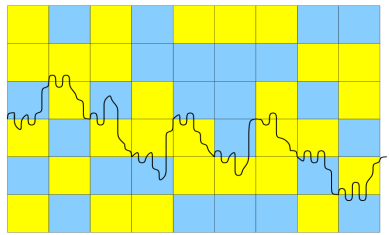}
\caption{\small A directed self-avoiding path crossing blocks of oil and water, 
entering and exiting blocks at diagonally opposite corners. The disorder along
the copolymer is not indicated.}
\label{fig-copolemulcorner}
\end{figure}

The above model was studied in den Hollander and Whittington~\cite{dHoWh06}, den 
Hollander and P\'etr\'elis~\cite{dHoPe09a,dHoPe09b,dHoPe10}. The key parameter 
ranges are $p\in (0,1)$, $\alpha,\beta\in (0,\infty)$, $|\beta| \leq \alpha$. The 
model is studied in the limit
\[
\lim_{n\to\infty} L_n = \infty, \qquad \lim_{n\to\infty} \tfrac{1}{n}L_n = 0.
\]
This is a coarse-graining limit in which \emph{the polymer scale and the emulsion 
scale separate}. In this limit both scales exhibit self-averaging. 

Theorems~\ref{thm:fecopolemul}--\ref{thm:phases} below summarize the main 
results (in qualitative language), and are illustrated by 
Figs.~\ref{fig-phdiagsup}--\ref{fig-pathsub}.

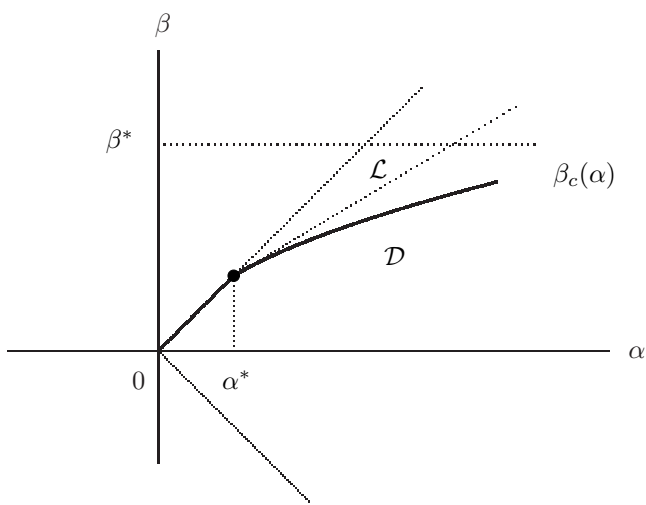
\begin{figure}
\begin{center}
\setlength{\unitlength}{0.5cm}
\begin{picture}(12,12)(0,-2)
\put(0,0){\line(12,0){12}}
\put(0,0){\line(0,8){8}}
\put(0,0){\line(0,-3){3}}
\put(0,0){\line(-4,0){4}}
{\thicklines
\qbezier(2,2)(4,3.2)(9,4.5)
\thicklines
\qbezier(0,0)(1,1)(2,2)
}
\qbezier[60](2,2)(4.5,3.5)(9.5,6.5)
\qbezier[60](0,5.5)(5,5.5)(10,5.5)
\qbezier[60](2,2)(4.5,4.5)(7,7)
\qbezier[15](2,0)(2,1)(2,2)
\qbezier[70](4,-4)(2,-2)(0,0)
\put(-.7,-1){$0$}
\put(12.5,-0.2){$\alpha$}
\put(-0.1,8.5){$\beta$}
\put(2,2){\circle*{.3}}
\put(1.7,-1){$\alpha^*$}
\put(-1.4,5.4){$\beta^*$}
\put(10.5,4.5){$\beta_c(\alpha)$}
\put(5.6,4.6){$\mathcal{L}$}
\put(6,2.3){$\mathcal{D}$}
\end{picture}
\end{center}
\vspace{1.5cm}
\caption{\small Phase diagram in the supercritical regime.}
\label{fig-phdiagsup}
\end{figure}

\begin{figure}[htbp]
\includegraphics[width=.50\hsize]{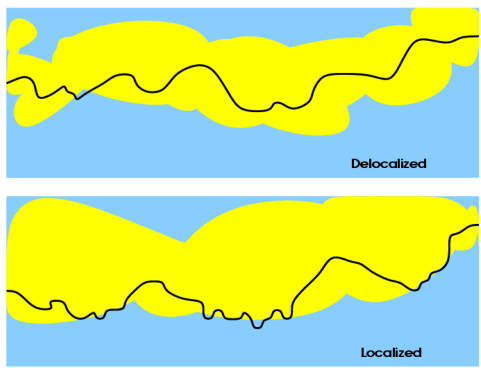}
\caption{\small Path behavior in the two phases in the supercritical regime.}
\label{fig-pathsup}
\end{figure}

\begin{figure}[htbp]
\vspace{1cm}
\begin{center}
\setlength{\unitlength}{0.5cm}
\begin{picture}(12,12)(0,-3.5)
\put(0,0){\line(12,0){12}}
\put(0,0){\line(0,8){8}}
\put(0,0){\line(0,-3){3}}
\put(0,0){\line(-3,0){3}}
{\thicklines
\qbezier(0,0)(2,2)(4,4)
\qbezier(4,4)(3.5,2)(3,1)
\qbezier(3,1)(2,0)(1,-1)
\qbezier(3,1)(5,2)(9,3)
\qbezier(4,4)(6,5)(10,6)
}
\qbezier[20](4,0)(4,2)(4,4)
\qbezier[60](4,4)(6,6)(8,8) 
\qbezier[40](0,0)(1.6,-1.5)(3.2,-3.0)
\qbezier[80](0,6.5)(4.5,6.5)(9,6.5)
\put(-.8,-.8){$0$}  
\put(12.5,-0.2){$\alpha$}
\put(-0.1,8.5){$\beta$}
\put(3.7,-.8){$\alpha^*$}
\put(4,4){\circle*{.25}}
\put(3,1){\circle*{.25}}
\put(1.6,.8){$\cD_1$}
\put(6.5,1){$\cD_2$}
\put(6.5,3.5){$\cL_1$}
\put(8.5,7){$\cL_2$}
\end{picture}
\end{center}
\caption{\small Phase diagram in the subcritical regime.}
\label{fig-phdiagsub}
\end{figure}
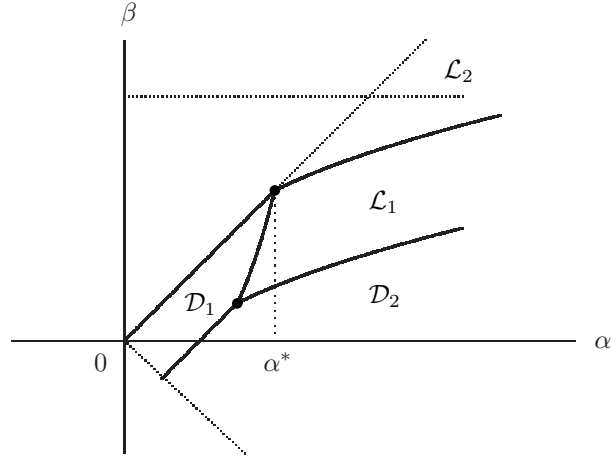

\begin{figure}[htbp]
\includegraphics[width=.40\hsize]{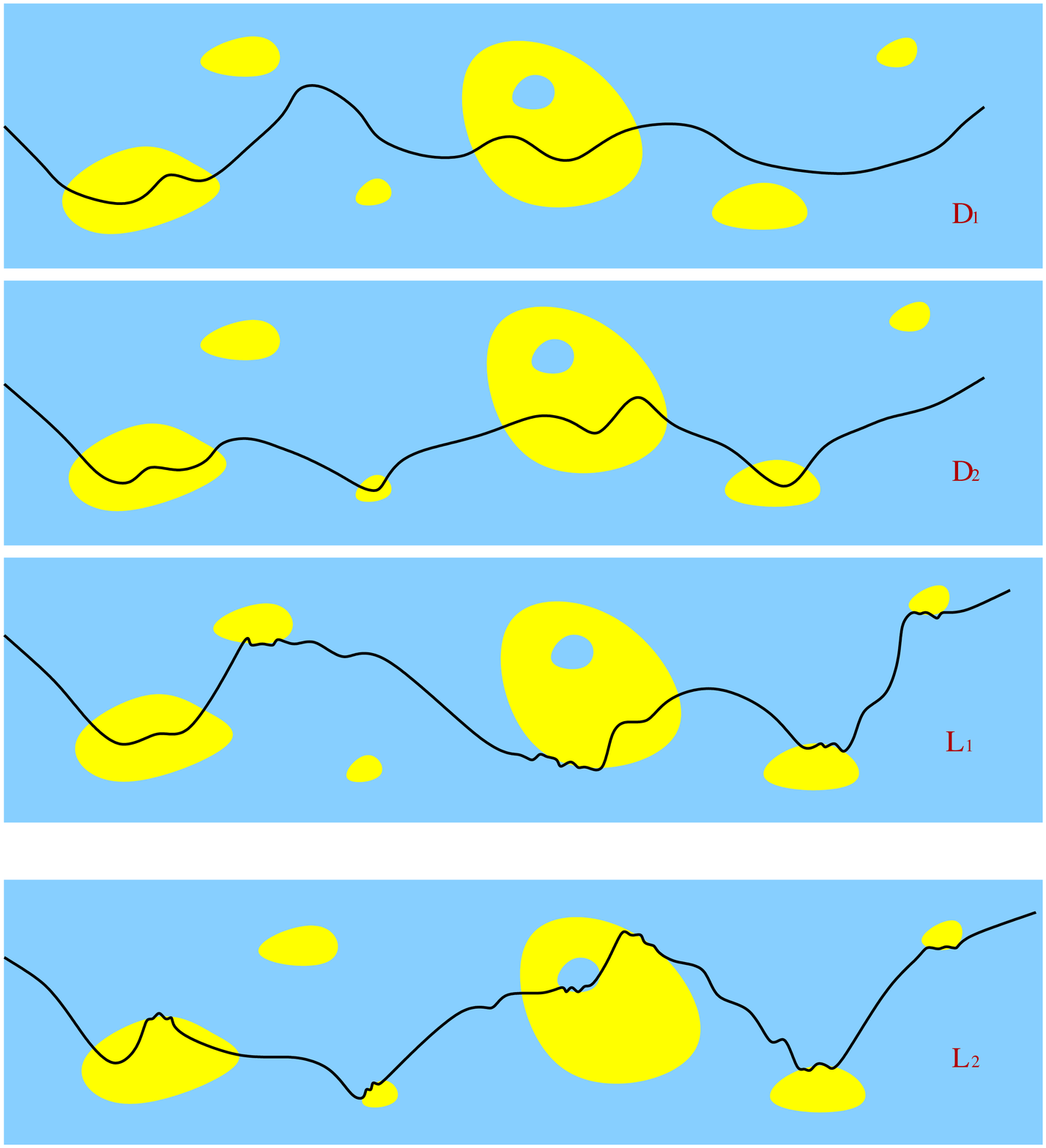}
\caption{\small Path behavior in the four phases in the subcritical regime.}
\label{fig-pathsub}
\end{figure}

\begin{theorem}
\label{thm:fecopolemul}
{\rm [den Hollander and Whittington~\cite{dHoWh06}]}
The free energy exists and is non-random $\omega,\Omega$-a.s., and is given by a 
variational formula involving the free energies of the copolymer in each of the 
four possible pairs of adjacent blocks, the frequencies at which the copolymer 
visits these pairs on the emulsion scale, and the fractions of time the copolymer 
spends in these pairs on the polymer scale.
\end{theorem}

\begin{theorem}
\label{thm:phases}
{\rm [den Hollander and Whittington~\cite{dHoWh06}, den Hollander and 
P\'etr\'e\-lis~\cite{dHoPe09a,dHoPe09b,dHoPe10}]}
The analysis of the variational formula reveals that there are two regimes:\\ 
(I) Supercritical: the oil blocks percolate. There are two phases separated by 
one critical curve.\\
(II) Subcritical: the oil blocks do not percolate. There are four phases separated 
by three critical curves meeting in two tricritical points.  
\end{theorem}

As shown in Figs.~\ref{fig-phdiagsup}--\ref{fig-pathsub}, the copolymer-emulsion 
model shows a remarkably rich phase behavior and associated path behavior. In
the supercritical regime the phase diagram shows one critical curve separating 
two phases. There is a \emph{delocalized} phase $\cD$ where the copolymer lies 
entirely inside the infinite oil cluster, and a \emph{localized phase} $\cL$ 
where part of the copolymer lies near the boundary of the infinite oil cluster. 
In the subcritical regime the phase diagram shows three critical curves separating 
four phases meeting at two tricritical points. There are two delocalized phases
$\cD_1$, $\cD_2$ and two localized phases $\cL_1$, $\cL_2$. For each pair, the 
distinction comes from the way in which the copolymer behaves near the boundary 
of the finite oil clusters.

The corner restriction is unphysical, but makes the model mathematically tractable. 
In den Hollander and P\'etr\'elis~\cite{dHoPepr} this restriction is removed, but
the resulting variational formula for the free energy is more complex. The 
coarse-graining limit is an important simplification: mesoscopic disorder is easier 
to deal with than microscopic disorder. An example of a model with microscopic 
disorder in space-time will be the topic of Section~\ref{S6}.


\subsection{Open problems}
\label{S5.6}

Here are some challenges:
\begin{itemize}
\item
For the copolymer model in Sections~\ref{S5.1}--\ref{S5.4}, prove that throughout 
the interior of the delocalized phase the path intersects the interface only finitely 
often under the average quenched path measure. 
\item
Determine whether the phase transition is second order or higher order.
\item
Compute the critical slope $K_c$ of the Brownian copolymer.
\item
For the copolymer/emulsion model in Section~\ref{S5.5}, determine the fine details of 
the phase diagrams in Figs.~\ref{fig-phdiagsup} and \ref{fig-phdiagsub}, and of
the path properties in Figs.~\ref{fig-pathsup} and \ref{fig-pathsub}.
\end{itemize}


\section{A polymer in a random potential}
\label{S6}

This section takes a look at a $(1+d)$-dimensional directed polymer in a random 
potential: the polymer and the potential live on $\N\times\Z^d$, where $\N$ is 
time and $\Z^d$, $d \geq 1$, is space (see Fig.~\ref{fig-polpot}). In Section~\ref{S6.1} 
we define the model. In Sections~\ref{S6.2}--\ref{S6.4} we study the two phases that 
occur: the \emph{weak disorder phase}, in which the polymer largely ignores the disorder 
and behaves diffusively, and the \emph{strong disorder phase}, in which the polymer 
hunts for favorable spots in the disorder and behaves superdiffusively. In 
Section~\ref{S6.5} we derive bounds on the critical temperature separating the 
two phases. Section~\ref{S6.6} lists a few open problems.

\begin{figure}[htbp]
\includegraphics[width=.40\hsize]{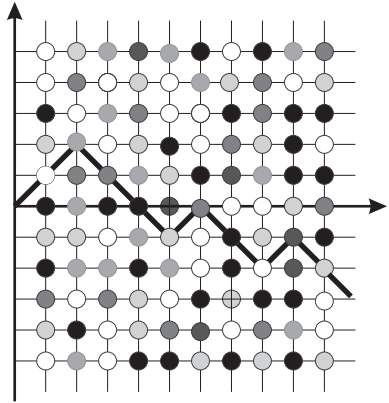}
\caption{\small A directed polymer in a random potential. Different shades of white, 
grey and black represent different values of the potential.}
\label{fig-polpot}
\end{figure}


\subsection{Model}
\label{S6.1}

The set of paths is
\[
\begin{aligned}
\cW_n &= \big\{w=(i,w_i)_{i=0}^n\colon\\
&\qquad w_0=0,\,\|w_{i+1}-w_i\| = 1\,\,\forall\,0 \leq i < n\big\}.
\end{aligned}
\]
The random environment
\[
\omega = \{\omega(i,x)\colon\,i\in\N,\,x\in\Z^d\}
\]
consists of an i.i.d.\ field of $\R$-valued non-degenerate random variables with 
moment generating function
\[
M(\beta) = \mE\big(\erom^{\beta\omega(1,0)}\big)<\infty 
\qquad \forall\,\beta \in [0,\infty),
\]
where $\mP$ denotes the law of $\omega$. The Hamiltonian is
\[
H_n^{\beta,\omega}(w) = -\beta \sum_{i=1}^n \omega(i,w_i), \qquad w\in\cW_n,
\]
where $\beta$ plays the role of the \emph{disorder strength}. The 
associated quenched path measure is
\[
P_n^{\beta,\omega}(w) = \frac{1}{Z_n^{\beta,\omega}}\,
\erom^{-H_n^{\beta,\omega}(w)}\,P_n(w),
\qquad w \in\cW_n,
\]
where $P_n$ is the projection onto $\cW_n$ of the law $P$ of directed $\SRW$ on 
$\Z^d$.

We may think of the model as a version of the ``copolymer in emulsion'' described 
in Section~\ref{S5.5} where the disorder is \emph{microscopic} rather than 
\emph{mesoscopic}. There are deep relations with several other models in probability 
theory and statistical physics, including growth and wave-front-propagation models 
and first-passage percolation. Indeed, for $\beta=\infty$ the polymer follows the 
path along which the sum of the disorder is largest. This case corresponds to 
\emph{oriented first-passage percolation}, of which some aspects are discussed in 
the lectures by Garban and Steif~\cite{GaSt11}. For $\beta<\infty$ the model is 
sometimes referred to as \emph{oriented first-passage percolation at positive 
temperature}.

The key object in the analysis of the model is the following quantity:
\[
Y_n^{\beta,\omega} = \frac{Z_n^{\beta,\omega}}{\mE(Z_n^{\beta,\omega})}, 
\qquad n\in\N_0.
\]
This is the \emph{ratio of the quenched and the annealed partition sum}.
The point is that
\[
(Y_n^{\beta,\omega})_{n\in\N_0}
\] 
is a \emph{martingale} w.r.t.\ the natural filtration generated by $\omega$,
i.e., $\cF=(\cF_n)_{n\in\N_0}$ with $\cF_n=\sigma(\omega(i,x)\colon\,0\leq i
\leq n,\,x\in\Z^d)$. Indeed, this is seen by writing
\[
Y_n^{\beta,\omega} = E\left(\prod_{i=1}^n 
\left[\frac{\erom^{\beta\omega(i,S_i)}}{M(\beta)}\right]\right),
\qquad Y_0^{\beta,\omega}=1, 
\]
from which it is easily deduced that $\mE(Y_n^{\beta,\omega}|\cF_{n-1})
= Y_{n-1}^{\beta,\omega}$. Note that $\mE(Y_n^{\beta,\omega})=1$ and 
$Y_n^{\beta,\omega}>0$ for all $n\in\N_0$.


\subsection{A dichotomy: weak and strong disorder}
\label{S6.2}

Since $Y_n^{\beta,\omega}\geq 0$, it follows from the martingale convergence theorem
that
\[
Y^{\beta,\omega} = \lim_{n\to\infty} Y_n^{\beta,\omega} 
\quad \mbox{exists } \omega\text{-a.s.} 
\]
Moreover, since the event $\{\omega\colon\,Y^{\beta,\omega}>0\}$ is 
measurable w.r.t.\ the tail sigma-algebra of $\omega$, it follows from the 
Kolmogorov zero-one law that the following \emph{dichotomy} holds:
\[
\begin{array}{lll}
&\mathrm{(\WD)}\colon\, &\quad\mP(Y^{\beta,\omega}>0)=1,\\
&\mathrm{(\SD)}\colon\, &\quad\mP(Y^{\beta,\omega}=0)=1.
\end{array}
\]
In what follows it will turn out that $\mathrm{(WD)}$ characterizes 
\emph{weak disorder}, for which the behavior of the polymer is diffusive
in the $\Z^d$-direction, while $\mathrm{(SD)}$ characterizes \emph{strong 
disorder}, for which the behavior is (expected to be) superdiffusive 
(see Fig.~\ref{fig-WDSD}). Note that the nomenclature is appropriate: 
in phase $\mathrm{(WD)}$ the quenched and the annealed partition 
sum remain comparable in the limit as $n\to\infty$, indicating a 
weak role for the disorder, while in phase $\mathrm{(SD)}$ the annealed 
partition sum grows faster than the quenched partition sum, indicating 
a strong role for the disorder.

\begin{figure}[htbp]
\vspace{.5cm}
\includegraphics[scale = 0.3]{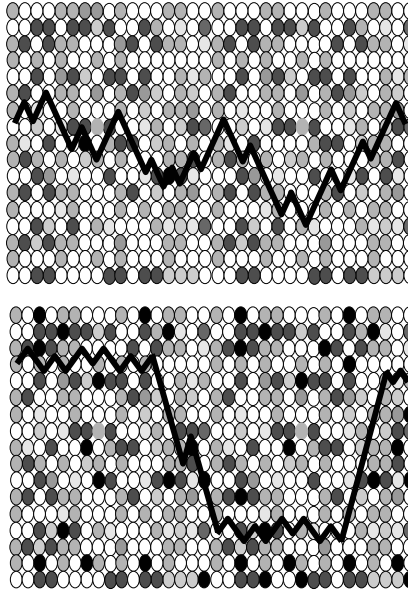}
\caption{\small Typical path behavior in the two phases.}
\label{fig-WDSD}
\end{figure}


\subsection{Separation of the two phases}
\label{S6.3}

\begin{theorem}
\label{thm:phsep}
{\rm [Comets and Yoshida~\cite{CoYo06}]}
For any choice of the disorder distribution, $\beta\mapsto\mE
(\sqrt{Y^{\beta,\omega}})$ is non-increasing on $[0,\infty)$. 
Consequently, there exists a $\beta_c \in [0,\infty]$ such 
that (see Fig.~{\rm \ref{fig-phsep}})
\[
\begin{aligned}
\beta \in [0,\beta_c) \quad &\longrightarrow \quad \mathrm{(WD)},\\ 
\beta \in (\beta_c,\infty) \quad &\longrightarrow \quad \mathrm{(SD)}.
\end{aligned}
\]
\end{theorem}

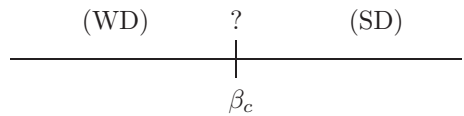
\begin{figure}[htbp]
\begin{center}
\setlength{\unitlength}{0.5cm}
\begin{picture}(12,3)(0,-1)
\put(0,0){\line(12,0){12}}
\put(6,-.5){\line(0,1){1}}
\put(5.8,-1.3){$\beta_c$}
\put(5.85,.8){?}  
\put(1.9,.8){$\mathrm{(WD)}$}
\put(9,.8){$\mathrm{(SD)}$}
\end{picture}
\end{center}
\caption{\small Separation of the two phases. It is not known which of
the two phases includes $\beta_c$.}
\label{fig-phsep}
\end{figure}

Since
\[
\begin{aligned}
f^\mathrm{que}(\beta) &= \lim_{n\to\infty} \frac{1}{n} \log Z_n^{\beta,\omega}
\qquad \omega\text{-a.s.},\\
f^\mathrm{ann}(\beta) &= \lim_{n\to\infty} \frac{1}{n} \log \mE(Z_n^{\beta,\omega}),
\end{aligned}
\]
it follows from the above theorem that
\[
f^\mathrm{que}(\beta) = f^\mathrm{ann}(\beta) \qquad \forall\,\beta \in [0,\beta_c],
\]
where the critical value $\beta=\beta_c$ can be added because free energies are continuous. 
It is expected that (see Fig.~\ref{fig-fepolpot})
\[
f^\mathrm{que}(\beta) < f^\mathrm{ann}(\beta) \qquad \forall\,
\beta \in (\beta_c,\infty), 
\]
so that for $\beta \in (\beta_c,\infty)$ the quenched and the annealed partition sum
have different exponential growth rates, but this remains open. Partial results have 
been obtained in Comets and Vargas~\cite{CoVa06}, Lacoin~\cite{La10}. 

\begin{figure}[htbp]
\begin{center}
\setlength{\unitlength}{0.35cm}
\begin{picture}(12,12)(0,-1.5)
\put(0,0){\line(12,0){12}}
\put(0,0){\line(0,8){8}}
{\thicklines
\qbezier(0,0)(5,0.5)(9,7.5)
\qbezier(4,1.6)(7,3)(9,5.5)
}
\qbezier[20](4,0)(4,1)(4,1.6)
\put(-.8,-1){$0$}
\put(12.5,-0.2){$\beta$}
\put(-0.2,8.5){$f$}
\put(10,6){$f^\mathrm{que}(\beta)$}
\put(10,8){$f^\mathrm{ann}(\beta)$}
\put(3.8,-1){$\beta_c$}
\put(0,0){\circle*{.4}}
\put(4,1.6){\circle*{.4}}
\end{picture}
\end{center}
\caption{\small Conjectured behavior of the quenched and the annealed free energy.}
\label{fig-fepolpot}
\end{figure}
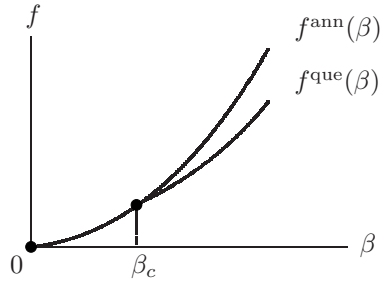


\subsection{Characterization of the two phases}
\label{S6.4}

Let 
\[
\pi_d = (P \otimes P')(\exists\,n\in\N\colon\,S_n=S'_n)
\]
denote the collision probability of two independent copies of $\SRW$. Note that 
$\pi_d=1$ in $d=1,2$ and $\pi_d<1$ in $d \geq 3$. For $\beta\in [0,\infty)$, 
define
\[
\begin{aligned}
\Delta_1(\beta) &= \log[M(2\beta)/M(\beta)^2],\\ 
\Delta_2(\beta) &= \beta[\log M(\beta)]'-\log M(\beta).
\end{aligned}
\]
Both $\beta\mapsto\Delta_1(\beta)$ and $\beta\mapsto\Delta_2(\beta)$ are strictly 
increasing on $[0,\infty)$, with $\Delta_1(0)=\Delta_2(0)=0$ and $\Delta_1(\beta)
>\Delta_2(\beta)$ for $\beta\in(0,\infty)$. 

Define
\[
\mathrm{max}_n^{\beta,\omega} = \max_{x\in\mZ^d} P_n^{\beta,\omega}(S_n=x), 
\qquad n\in\N.
\]
This quantity measures how localized the endpoint $S_n$ of the polymer is in the 
given potential $\omega$: if $\lim_{n\to\infty} \mathrm{max}_n^{\beta,\omega}=0$, 
then the path spreads out, while if $\limsup_{n\to\infty} \mathrm{max}_n^{\beta,\omega} 
>0$, then the path localizes (at least partially).
 
\begin{theorem}
\label{thm:polpotchar1}
{\rm [Imbrie and Spencer~\cite{ImSp88}, Bolthausen~\cite{Bo89}, Sinai~\cite{Si95}, 
Carmona and Hu~\cite{CaHu02}, Comets, Shiga and Yoshida~\cite{CoShYo03}]}
Suppose that 
\begin{itemize}
\item[$\mathrm{(I)}$]
\qquad $d \geq 3$, $\Delta_1(\beta) < \log(1/\pi_d)$.
\end{itemize} 
Then
\[
\lim_{n\to\infty} \frac{1}{n}\,E_n^{\beta,\omega}(\|S_n\|^2)=1 \qquad \omega\text{-a.s.}
\]
 and
\[\lim_{n\to\infty}\,\mathrm{max}_n^{\beta,\omega} = 0 \qquad \omega\text{-a.s.}
\]
\end{theorem}

\begin{theorem}
\label{thm:polpotchar2}
{\rm [Carmona and Hu~\cite{CaHu02}, Comets, Shiga and Yoshida~\cite{CoShYo03}]}
Suppose that 
\begin{itemize}
\item[$\mathrm{(II)}$]
\qquad $d=1,2$, $\beta>0$ \quad or \quad $d\geq 3$, $\Delta_2(\beta) > \log(2d)$.
\end{itemize} 
Then there exists a $c=c(d,\beta)>0$ such that 
\[
\limsup_{n\to\infty}\,\mathrm{max}_n^{\beta,\omega} \geq c \qquad \omega\text{-a.s.}
\]
\end{theorem}

Theorems~\ref{thm:polpotchar1}--\ref{thm:polpotchar2} show that the polymer has qualitatively 
different behavior in the two regimes. In (I), the scaling is \emph{diffusive}, with the 
diffusion constant \emph{not renormalized} by the disorder. The reason why the diffusion 
constant is not renormalized is the directedness of the path: this causes the annealed model 
to be directed $\SRW$. In (II), there is certainly no scaling to Brownian motion, due to the 
presence of atoms: the endpoint of the polymer concentrates around one or more most favorable 
sites whose locations depend on $\omega$. These locations are expected to be at a distance 
much larger than $\sqrt{n}$, i.e., the scaling is predicted to be \emph{superdiffusive}. 
This has, however, only been proved in some special cases, in particular, for a one-dimensional 
model of a directed polymer in a Gaussian random environment (Petermann~\cite{Pe00}). 
Further results, also for related models, have been obtained in Piza~\cite{Pi97}, 
M\'ejane~\cite{Me04}, Carmona and Hu~\cite{CaHu04}, Bezerra, Tindel and Viens~\cite{BeTiVi08} 
and Lacoin~\cite{La11}. The latter reference contains a discussion of the physical conjectures 
and the mathematical results on this topic.

The proofs of Theorems~\ref{thm:polpotchar1}--\ref{thm:polpotchar2} are based on 
a series of technical estimates for the martingale $(Y_n^{\beta,\omega})_{n\in\N_0}$. 
These estimates also show that
\[
\mathrm{(I)} \longrightarrow \mathrm{(\WD)},
\qquad 
\mathrm{(II)} \longrightarrow \mathrm{(\SD)}.
\]
It has been conjectured that, throughout phase (SD), 
\[
E_n^{\beta,\omega}(\|S_n\|^2) \asymp n^{2\nu} \quad n\to\infty,\,\omega-a.s.
\]
($\asymp$ means modulo logarithmic factors), where the exponent $\nu$ is predicted 
not to depend on $\beta$ and to satisfy
\[
\nu=\tfrac23 \mbox{ for } d=1, \qquad \nu \in (\tfrac12,\tfrac23) \mbox{ for } d=2,
\]
signalling \emph{superdiffusive behavior}.


\subsection{Bounds on the critical temperature}
\label{S6.5}

Theorems~\ref{thm:polpotchar1}--\ref{thm:polpotchar2} show that $\beta_c=0$ for 
$d=1,2$ and $\beta_c\in(0,\infty]$ for $d\geq 3$ (because $\Delta_1(0)=0$ and 
$\pi_d<1$). However, there is a gap between regimes (I) and (II) in $d \geq 3$ 
(because $\pi_d>1/2d$ and $\Delta_1(\beta)>\Delta_2(\beta)$ for all $\beta>0$). 
Thus, the results do not cover the full parameter regime. In fact, all we know 
is that
\[
\beta_c \in [\beta_c^1,\beta_c^2].
\]
with (see Fig.~\ref{fig-polpotcritvals})
\[
\begin{aligned}
\beta_c^1 &= \sup\big\{\beta\in 
[0,\infty)\colon\,\Delta_1(\beta)<\log(1/\pi_d)\big\},\\
\beta_c^2 &= \inf\big\{\beta\in 
[0,\infty)\colon\,\Delta_2(\beta)>\log(2d)\big\}.
\end{aligned}
\]

\begin{figure}[htbp]
\vspace{-.5cm}
\begin{center}
\setlength{\unitlength}{0.6cm}
\begin{picture}(12,3)(0,-1.3)
\put(0,0){\line(12,0){12}}
\put(6,-.5){\line(0,1){1}}
\put(3,-.5){\line(0,1){1}}
\put(9,-.5){\line(0,1){1}}
\put(5.6,-1.3){$\beta_c$}
\put(2.6,-1.3){$\beta_c^1$}
\put(8.6,-1.3){$\beta_c^2$}
\end{picture}
\end{center}
\caption{\small For $d \geq 3$ three cases are possible depending on the law
$\mP$ of the disorder: (1) $0<\beta_c^1<\beta_c^2<\infty$; (2) $0<\beta_c^1<\beta_c^2
=\infty$; (3) $\beta_c^1=\beta_c^2=\infty$.}
\label{fig-polpotcritvals}
\end{figure}
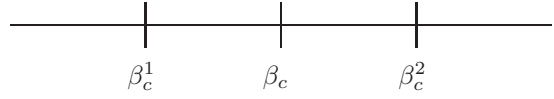

Various attempts have been made to sharpen the estimates on $\beta_c$:
fractional moment estimates on the martingale (Evans and Derrida~\cite{EvDe92}, 
Coyle~\cite{Co98}, Camanes and Carmona~\cite{CaCa09}); size-biasing of the 
martingale (Birkner~\cite{Bi04}). We describe the latter estimate, which 
involves a critical threshold $z^*$ associated with the collision local time 
of two independent $\SRW$s.

\begin{theorem}
{\rm [Birkner~\cite{Bi04}]} 
Let
\[
z^* = \sup\big\{z\geq 1\colon\,E\big(z^{V(S,S')}\big)<\infty\,\,S'-a.s.\big\},
\]
where
\[
V(S,S')=\sum_{n\in\N} 1_{\{S_n=S'_n\}}
\] 
is the collision local time of two independent $\SRW${\rm s}, 
and $E$ denotes expectation over $S$. Define
\[
\beta_c^* = \sup\big\{\beta\in[0,\infty)\colon\,M(2\beta)/M(\beta)^2<z^*\big\}.
\]
 Then
\[
\beta<\beta_c^* \quad \longrightarrow \quad {\rm (WD)}
\]
and, consequently, $\beta_c\geq\beta_c^*$.
\end{theorem}

\begin{proof}
Abreviate
\[
e = \{e(i,x)\}_{i\in\N,x\in\Z^d}
\]
with 
\[
e(i,x) = \erom^{\beta\omega(i,x)}/M(\beta). 
\]
Consider a size-biased version of $e$, written 
\[
\hat{e} = \{\hat{e}(i,x)\}_{i\in\N,x\in\Z^d},
\]
that is i.i.d., is independent of $e$ and has law $\hat{\mP}$ given by
\[
\hat{\mP}(\hat{e}(1,0)\in\,\cdot\,) = \mE\big(e(1,0)\,1_{\{e(1,0)\in\,\cdot\,\}}\big).
\]
No normalization is needed because $\mE(e(1,0))=1$. 

Given $S'$, put
\[
\hat{e}_{S'} = \{\hat{e}_{S'}(i,x)\}_{i\in\N,x\in\Z^d},
\]
with
\[
\hat{e}_{S'}(i,x) = 1_{\{S'_i\neq x\}}\,e(i,x) + 1_{\{S'_i = x\}}\,\hat{e}(i,x),
\]
i.e., size-bias $e$ to $\hat{e}$ everywhere along $S'$, and define
\[
\hat{Y}_n^{e,\hat{e},S'} = E\left(\prod_{i=1}^n \hat{e}_{S'}(i,S_i)\right). 
\]
This is a \emph{size-biased version of the basic martingale}, which in the present 
notation reads
\[
Y_n^e = E\left(\prod_{i=1}^n e(i,S_i)\right). 
\]

The point of the size-biasing carried out above is that for any bounded function 
$f\colon\,[0,\infty)\to\R$, 
\[
\mE\big(Y_n^e\,f(Y_n^e)\big) 
= (\mE\otimes\hat{\mE}\otimes E')\Big(f\big(\hat{Y}_n^{e,\hat{e},S'}\big)\Big),
\]
where $\mE,\hat{\mE},E'$ denote expectation w.r.t.\ $e,\hat{e},S'$, respectively. 
Indeed, the latter follows from the computation
\[
\begin{aligned}
\mE\big(Y_n^e\,f(Y_n^e)\big) 
&= \mE\left[E'\left(\prod_{i=1}^n e(i,S'_i)\right)\,
f\left(E\left(\prod_{i=1}^n e(i,S_i)\right)\right)\right]\\
&= E'\left(\mE\left[\left(\prod_{i=1}^n e(i,S'_i)\right)\,
f\left(E\left(\prod_{i=1}^n e(i,S_i)\right)\right)\right]\right)\\
& =^{!} E'\left((\mE\otimes\hat{\mE})\left[
f\left(E\left(\prod_{i=1}^n \hat{e}_{S'}(i,S_i)\right)\right)\right]\right)\\
&= (\mE\otimes\hat{\mE}\otimes E')\Big(f\big(\hat{Y}_n^{e,\hat{e},S'}\big)\Big),
\end{aligned}
\]
where the third equality uses the definition of $\hat{e}_{S'}$. 

The above identity relates the two martingales, and implies that
\[
\begin{aligned}
&(Y_n^e)_{n\in\N_0} \mbox{ is uniformly integrable }\\ 
&\qquad \quad \Longleftrightarrow \quad (\hat{Y}_n^{e,\hat{e},S'})_{n\in\N_0}
\mbox{ is tight} \qquad (\ast) 
\end{aligned}
\]
(as can be seen by picking $f$ such that $\lim_{u\to\infty} f(u)=\infty$). However, 
an easy computation gives
\[
(\mE\otimes\hat{\mE})\Big(\hat{Y}_n^{e,\hat{e},S'}\Big) 
= E\left(z^{\sum_{i=1}^n 1_{\{S_i=S'_i\}}}\right)
= E\big(z^{V(S,S')}\big) 
\]
with $z = M(2\beta)/M(\beta)^2$, where the factor $2\beta$ arises because after
the size-biasing the intersection sites of $S$ and $S'$ are visited by both 
paths. Hence
\[
E(z^{V(S,S')})<\infty \quad S'\text{-a.s.}
\] 
is enough to ensure that the r.h.s.\ of $(\ast)$ holds. This completes the proof 
because the l.h.s.\ of $(\ast)$ is equivalent to (WD). Indeed, a.s.\ convergence 
plus uniform integrability imply convergence in mean, so that $\mE(Y_n^e)=1$
for all $n\in\N_0$ yields $\mE(Y^e)=1$.
\end{proof}

In Birkner, Greven and den Hollander~\cite{BiGrdHo11} it was proved that 
$z^*>\pi_d$ in $d \geq 5$, implying that $\beta_c^*>\beta_c^1$. It was 
conjectured that the same is true in $d=3,4$. Part of this conjecture was 
settled in Birkner and Sun~\cite{BiSu10,BiSupr} and Berger and 
Toninelli~\cite{BeTo10} (see Fig.~\ref{fig-bdscrittemppolpot}).

\begin{figure}[htbp]
\vspace{-0.5cm}
\begin{center}
\setlength{\unitlength}{0.6cm}
\begin{picture}(12,3)(0,-1.3)
\put(0,0){\line(12,0){12}}
\put(3,-.5){\line(0,1){1}}
\put(4.5,-.5){\line(0,1){1}}
\put(6,-.5){\line(0,1){1}}
\put(9,-.5){\line(0,1){1}}
\put(2.6,-1.3){$\beta_c^1$}
\put(4.1,-1.3){$\beta_c^*$}
\put(5.6,-1.3){$\beta_c$}
\put(8.6,-1.3){$\beta_c^2$}
\end{picture}
\end{center}
\caption{\small Bounds on the critical temperature.}
\label{fig-bdscrittemppolpot}
\end{figure}
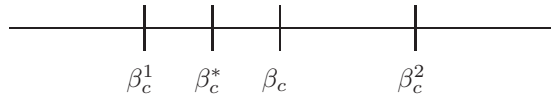


\subsection{Open problems}
\label{S6.6}

\begin{itemize}
\item
Show that in phase $\mathrm{(SD)}$ the polymer is concentrated inside a 
most favorable corridor and identify how this corridor depends on $\omega$. 
\item
Determine whether $\beta_c$ is part of $\mathrm{(WD)}$ or $\mathrm{(SD)}$.
\item
Derive a variational expression for $\beta_c$.
\item
Extend the analysis to undirected random walk. Important progress has been 
made in Ioffe and Velenik~\cite{IoVe10,IoVe11,IoVepra,IoVeprb}, 
Zygouras~\cite{Zy09,Zypr}, and references cited therein. See 
also Section~\ref{S2.7}.  
\end{itemize}


\appendix

\section{Tutorial 1}
\label{appA}

In this tutorial we describe two methods that can be used to prove the 
\emph{existence} of the quenched free energy associated with the random 
pinning model described in Section~\ref{S4}. Section~\ref{pinmod} recalls
the model, Sections~\ref{pinconv}--\ref{concentr} prove existence of
the quenched free energy when the endpoint of the polymer is constrained 
to lie in the interface, while Section~\ref{rem} shows how to remove 
this constraint afterwards. The method of proof is widely applicable, and 
is not specific to the random pinning model.  


\subsection{Random pinning of a polymer at an interface}
\label{pinmod}

\textit{Configurations of the polymer.} 
Let $n \in \N$ and consider a polymer made of $n$ monomers. The allowed configurations 
of this polymer are modeled by the $n$-step trajectories of a 1-dimensional random 
walk $S=(S_i)_{i\in\N_0}$. We focus on the case where $S_0=0$ and $(S_i-S_{i-1})_{i
\in\N}$ is an i.i.d.\ sequence of random variables satisfying
\[
P(S_1=1)=P(S_1=-1)=P(S_1=0)=\tfrac13,
\]
although the argument given below applies more generally. We denote by $\cW_n$ the set 
of all $n$-step trajectories of $S$.

\medskip\noindent
\textit{Disorder at interface.}
Let $\omega=(\omega_i)_{i\in\N}$ be an i.i.d.\ sequence of $\R$-valued random 
variables (which we take bounded for ease of exposition). For $i\in \N$ the interaction 
intensity between the $i$-th monomer and the interface takes the value $\omega_i$. Note 
that $\omega$ and $S$ are independent, and write $\P$ for the law of $\omega$. Pick $M>0$ 
such that $|\omega_1| \leq M$ $\P$-a.s.

\medskip\noindent 
\textit{Interaction polymer-interface.}
The flat interface that interacts with the polymer is located at height $0$, so 
that the polymer hits this interface every time $S$ comes back to $0$. Thus, with 
every $S\in \cW_n$ we associate the energy  
\[
H_n^{\beta,\omega}(S) = - \beta \sum_{i=0}^{n-1} \omega_i \,\ind_{\{S_i=0\}},
\]  
where $\beta\in (0,\infty)$ stands for the inverse temperature (and for ease of exposition
we take zero bias, i.e., we set $h=0$ in the Hamiltonian in Section~\ref{S4.1}). We think 
of $S$ as a random realization of the path of the polymer.

\medskip\noindent 
\textit{Partition function and free energy.} 
For fixed $n$, the quenched (= frozen disorder) partition function and free energy are 
defined as 
\[
Z_n^{\beta,\omega} = E\big(\erom^{-H_N^{\beta,\omega}(S)}\big)
\quad \text{and} \quad
f_n^{\beta,\omega}=\tfrac1n \log  Z_n^{\omega,\beta}.
\]


\subsection{Convergence of the free energy}
\label{pinconv}

Our goal is to prove the following theorem. 

\begin{theorem}
\label{theo}
For every $\beta\in \R$ there exists an $f(\beta)\in [0,\beta M]$ such that 
\[
\lim_{n\to \infty}  \E(f_n^{\beta,\omega}) = f(\beta)
\]
and 
\[
\lim_{n\to \infty} f_n^{\beta,\omega} = f(\beta) \quad  \P-a.e.\ \omega.
\]
\end{theorem}

As indicated above, we will prove Theorem~\ref{theo} via two different methods. In 
Section~\ref{king} we will state Kingman's \emph{Subadditive Ergodic Theorem} and 
see how this can be applied to obtain Theorem~\ref{theo}. In Section~\ref{concentr} 
we will re-prove Theorem~\ref{theo} by using a concentration of measure argument. 
The latter method is more involved, but also more flexible than the former method. 
For technical reasons, we will first prove Theorem~\ref{theo} with the partition 
function restricted to those trajectories that hit the interface at their right 
extremity, i.e., 
\[
Z_n^{*,\beta,\omega} = E\big(\erom^{-H_n^{\beta,\omega}(S)}\  
\ind_{\{S_n=0\}}\big) \quad \text{and} \quad 
f_n^{*,\beta,\omega} = \tfrac1n \log Z_n^{*,\beta,\omega}.
\]
In Section~\ref{rem} we will see that the restriction on the endpoint has no effect 
on the value of the limiting free energy.


\subsection{Method 1: Kingman's theorem}
\label{king}

\begin{theorem}
\label{theoking}
{\rm [Kingman's Subadditive Ergodic Theorem; see Steele~\cite{St89}]}
Let $(\Omega,A,\mu)$ be a probability space, let $T$ be an ergodic measure-preserving 
transformation acting on $\Omega$, and let $(g_n)_{n\in\N}$ be a sequence of random 
variables in $L_1(\mu)$ that satisfy the subadditivity relation
\[
g_{m+n} \geq g_m+g_n (T^m), \qquad m,n\in\N.
\]
Then 
\[
\lim_{n\to \infty} \frac{g_n}{n} = \sup_{k\in\N} E_{\mu}\left(\frac{g_k}{k}\right)
\qquad \mu\text{-a.s.} 
\]
\end{theorem}

\medskip\noindent
(1) Let $T$ be the left-shift on $\R^\N$. Prove that, for $m,n\in\N$ and $\omega\in \R^\N$,
\[
\label{subbad}
\log Z_{m+n}^{*,\beta,\omega} \geq \log Z_{m}^{*,\beta,\omega}
+ \log Z_{n}^{*,\beta,T^m(\omega)}.
\]

\medskip\noindent
(2) Apply Theorem~\ref{theoking} with $(\Omega,A,\mu)=(\R^\N,\text{Bor}(\R^{\N}),\P)$ 
and prove Theorem~\ref{theo} with the endpoint restriction.


\subsection{Method 2: Concentration of measure}
\label{concentr}

This method consists of first proving the first line in Theorem~\ref{theo}, i.e., 
the convergence of the \emph{average quenched} free energy, and then using a 
concentration of measure inequality to show that, with large probability, the 
quenched free energy is almost equal to its expectation, so that the second line 
in Theorem~\ref{theo} follows. See Giacomin and Toninelli~\cite{GiTo05} for
fine details.

\medskip\noindent
(1) Use \eqref{subbad} and prove that $(\E(\log Z_n^{*,\beta,\omega}))_{n\in\N}$ 
is a superadditive sequence, i.e., for $m,n\in\N$,
\[
\E(\log Z_{m+n}^{*,\beta,\omega}) \geq \E(\log Z_m^{*,\beta,\omega})
+ \E(\log Z_n^{*,\beta,\omega}).
\]

\medskip\noindent 
(2) Deduce that (see also the tutorial in Appendix A.1 of Bauerschmidt, Duminil-Copin, 
Goodman and Slade~\cite{BaDuGoSl11}) 
\[
\label{result}
\lim_{n\to \infty} \E(f_n^{*,\beta,\omega}) 
=\sup_{k\in\N}\,\E(f_k^{*,\beta,\omega}) 
= f(\beta) \in [0,\beta M].
\]

\medskip\noindent 
To proceed, we need the following inequality. 

\begin{theorem}
\label{thm:com}
{\rm [Concentration of measure; see Ledoux~\cite{Le01}]}
There exist $C_1,C_2$ $>0$ such that for all $n\in\N$, $K>0$, $\gep>0$ and $G_n\colon\,
\R^n \mapsto \R$ a $K$-Lipschitz (w.r.t.\ the Euclidean norm) convex function,
\[
\P\Big(\big|G_n(\omega_0,\dots,\omega_{n-1})
-\E\big(G_n(\omega_0,\dots,\omega_{n-1})\big)\big|>\gep\Big)
\leq C_1 \erom^{-\tfrac{C_2\gep^2}{K^2}}.
\]
\end{theorem}

\medskip\noindent 
(3) By H\"older's inequality, the function $\omega \in \R^n \mapsto f_n^{*,\beta,\omega} 
\in \R$ is convex. To prove that it is $(\beta/\sqrt{n})$-Lipschitz, pick $\omega,
\omega'\in\R^n$ and compute 
\[
\begin{aligned}
\left|\frac{\partial}{\partial t}\,f_n^{*,\beta,t\omega+(1-t)\omega'}\right| 
&= \frac{\beta}{n}\,\left|\sum_{i=0}^{n-1} P_n^{*,\beta,\omega}(S_i=0)\,(\omega_i-\omega'_i)\right|\\
&\leq \frac{\beta}{n}\,
\sqrt{\sum_{i=0}^{n-1} \big[P_n^{*,\beta,\omega}(S_i=0)\big]^2}\,
\sqrt{\sum_{i=0}^{n-1} (\omega_i-\omega'_i)^2}\\
&\leq \frac{\beta}{\sqrt{n}}\,\sqrt{\sum_{i=0}^{n-1} (\omega_i-\omega'_i)^2},
\end{aligned}
\]
where $P_n^{*,\beta,\omega}$ is the path measure with the endpoint restriction.

\medskip\noindent
(4) Apply Theorem~\ref{thm:com} to prove that, for $\gep>0$, 
\[
\sum_{n\in\N} \P\big(|f_n^{*,\beta,\omega}-\E(f_n^{*,\beta,\omega})|>\gep\big) < \infty.
\]

\medskip\noindent
(5) Combine (2) and (4) to show that, for $\P$-a.e.\ $\omega$, $f_n^{*,\beta,\omega}$ 
tends to $f(\beta)$ as $n\to \infty$, which proves Theorem~\ref{theo} with the endpoint 
restriction.


\subsection{Removal of the path restriction}
\label{rem}

The proof of Theorem~\ref{theo} will be completed once we show that restricting the 
partition function to $\{S_n=0\}$ does not alter the results. To that end, we denote 
by $\tau$ the first time at which the random walk $S$ hits the interface. 

\medskip\noindent
(6) Note that there exists a $C_3>0$ such that (see Spitzer~\cite{Sp76}, Section 1)
\[
P(\tau=n)=\frac{C_3}{n^{3/2}}\,[1+o(1)]\quad \text{and} \quad 
P(\tau>n) =\frac{2 C_3}{n^{1/2}}\,[1+o(1)].
\]

\medskip\noindent 
(7) Consider the last hit of the interface and show that 
\[
Z_n^{\beta,\omega} = \sum_{j=0}^n Z_j^{*,\beta,\omega}\,P( \tau>n-j).
\]

\medskip\noindent 
(8) Prove Theorem~\ref{theo} by combining (5), (6) and (7).


\section{Tutorial 2}
\label{appB}

The goal of this tutorial is to provide the combinatorial computation of the 
free energy for the directed polymer with self-attraction described in 
Sections~\ref{S2.4}--\ref{S2.5} leading to Theorem~\ref{thm:dircolpoltr}. 
This computation is taken from Brak, Guttmann and Whittington~\cite{BrGuWh92}. 
Section~\ref{moddirpolatt} recalls the model, Section~\ref{fedirpolatt} proves 
the existence of the free energy, while Section~\ref{compfedirpolatt} derives 
a formula for the free energy with the help of generating functions.


\subsection{Model of a directed polymer in a poor solvent}
\label{moddirpolatt}

We begin by recalling some of the notation used in Sections~\ref{S2.4}--\ref{S2.5}.

\medskip\noindent
\textit{Configurations of the polymer.}
For $n\in\N$, the configurations of the polymer are modelled by $n$-step 
$(1+1)$-dimensional directed self-avoiding paths $w=(w_i)_{i=0}^n$ 
that are allowed to move up, down and to the right, i.e., 
\begin{align*}
\cW_n=\{(w_i)_{i=0}^n \in (\N_0\times \Z)^{n+1}\colon & w_0=0,
w_1-w_0=\rightarrow,\\
& w_i-w_{i-1}=\{\uparrow, \downarrow, \rightarrow\}\  \forall
1 \leq i \leq n,\\
& w_i\neq w_j\   \forall\, 0 \leq i<j \leq n\}.
\end{align*}

\medskip\noindent
\textit{Self-touchings.}
The monomers constituting the polymer have an attractive interaction: an energetic 
reward is given for each \emph{self-touching}, i.e., for each pair $(w_i,w_j)$ with 
$i<j-1$ and $|w_i-w_j|=1$. Accordingly, with each $w\in \cW_n$ we associate the number 
of self-touchings
\[
J_n(w)= \sum_{0\leq i<j-1\leq n-1}  \,\ind_{\{|w_i-w_j|=1\}},
\]  
and the energy 
\[
H_n^{\gamma}(w)=-\gamma J_n(w),
\]  
where $\gamma\in\R$ is the interaction parameter.

\medskip\noindent 
\textit{Partition function, free energy and generating function.} 
For fixed $n$, the partition function and free energy are defined as 
\[
Z_n^\gamma=\sum_{w\in \cW_n} \erom^{-H_n^{\gamma}(w)}, \qquad
f_n(\gamma)=\tfrac1n \log Z_n^\gamma.
\]
For $n\in\N_0$ and $x \in [0,\infty)$, let
\[
Z_n(x)=\sum_{m\in\N_0} c_n(m)\, x^m, \quad 
c_n(m)=|\{w\in \cW_n\colon\, J_n(w)=m\}|. 
\]
Then $Z_n^{\gamma}=Z_n(\erom^{\gamma})$, and the generating function of $Z_n^\gamma$ 
can be written as
\[
\sum_{n\in\N_0} Z_n^\gamma y^n=G(\erom^\gamma,y)
\]
with
\[ 
G(x,y)=\sum_{n\in\N_0} \sum_{m\in\N_0} c_n(m) x^m y^n, \qquad x,y \in [0,\infty).
\]


\subsection{Existence of the free energy}
\label{fedirpolatt}

Existence comes in three steps.

\medskip\noindent
(1) Show that for $m,n\in\N_0$ and $x \in [0,\infty)$,
\[
Z_{m+n+1}(x) \geq Z_m(x)\,Z_n(x) \quad \text{and} 
\quad Z_n(x) \leq [3(1 \vee x)]^n.
\]

\medskip\noindent 
(2) Deduce that 
\[ 
\lim_{n\to\infty} \frac1n \log Z_n(x)= \sup_{k\in\N}\,
\frac1k \log Z_k(x) = \bar{f}(x) \in (0,\log 3+(0 \vee \log x)].
\]
Thus, $f(\gamma)=\bar{f}(\erom^\gamma)$, $\gamma\in\R$.

\medskip\noindent 
(3) For $x\in [0,\infty)$, let $y_c(x)$ be the radius of convergence of the generating function 
$G(x,y)$. Show that 
\begin{equation}
\label{radconv}
\bar{f}(x) = -\log y_c(x).
\end{equation}


\subsection{Computation of the free energy}
\label{compfedirpolatt}

To prove Theorem~\ref{thm:dircolpoltr}, we must compute $y_c(x)$, $x\in[0,\infty)$.
In what follows we derive the formula for $G(x,y)$ given in Lemma~\ref{lem:genfct}. 

\medskip\noindent
(1) For $n,r,s\in \N_0$, let 
\begin{align*}
&\cW_{n,r}\\
&=\{w\in \cW_n\colon\; \text{$w$ makes exactly $r$ vertical steps 
after the first step east}\},\\
&\cW_{n,r,s}\\
&=\{w\in \cW_{n,r}\colon\; \text{$w$ makes exactly $s$ vertical 
steps after the second step east}\},
\end{align*}
and note that $\cW_{n,r}=\emptyset$ if $n<1+r$ and $\cW_{n,r,s}=\emptyset$ if $n< 2+r+s$. 
Furthermore, for $r,s\in \N$, let 
\begin{align*}
&\cW_{n,r,s}^{\,\uparrow\downarrow}\\
&=\{w\in \cW_{n,r,s}\colon\; \text{the $r$ and $s$ vertical
steps are made in opposite directions}\},\\
&\cW_{n,r,s}^{\,\uparrow\uparrow}\\
&=\{w\in \cW_{n,r,s}\colon\; \text{the $r$ and $s$ vertical
steps are made in the same direction}\},
\end{align*}
so that, for $r\in \N$, $\cW_{n,r}$ can be partitioned as
\[
\cW_{n,r} = \bigcup_{s=0}^{n-r-2}\cW_{n,r,s} 
= \cW_{n,r,0} \cup \left[\bigcup_{s=1}^{n-r-2}
\big[\cW_{n,r,s}^{\,\uparrow\downarrow} 
\cup \cW_{n,r,s}^{\,\uparrow\uparrow}\big]\right].
\]
For $n,r,m\in \N_0$, let $c_{n,r}(m)$ be the number of $n$-step paths with $m$ self-touchings 
making exactly $r$ steps north or south immediately after the first step 
east, and put 
\[
g_r(x,y)=\sum_{n\in\N_0} \sum_{m\in\N_0} c_{n,r}(m)\, x^m\, y^n.
\]
Clearly, $G(x,y)=\sum_{r\in\N_0} g_r(x,y)$.

\medskip\noindent 
(2) Pick $r\in \N$ and use the first equality in the partitioning of $\cW_{n,r}$, together
with the fact that $c_{n,r}(m)=0$ when $n<r+1$, to prove that 
\[
g_r(x,y)=2 y^{r+1}+\sum_{s\in\N_0} \sum_{n=r+2+s}^\infty\ \sum_{w\in \cW_{n,r,s}}
x^{J_n(w)}\, y^n.
\]
For $w\in \cW_n$ and $0 \leq l < s \leq n$, let
\[
J_{l,s}(w)= \sum_{l\leq i<j-1\leq s-1}  \,\ind_{\{|w_i-w_j|=1\}},
\] 
which stands for the number of self-touchings made by $w$ between its $l$-th 
and $s$-th step. Clearly, $J_n(w)=J_{0,n}(w)$.

\medskip\noindent 
(3) Pick $r,s\in \N$ and $n\geq r+s+2$. Prove that  
\[
\begin{aligned}
&\quad w\in \cW_{n,r,s}^{\,\uparrow\uparrow}
\quad \longrightarrow \quad J_n(w)=J_{r+1,n}(w),\\
&\quad w\in \cW_{n,r,s}^{\,\uparrow\downarrow}
\quad \longrightarrow \quad J_n(w)=J_{r+1,n}(w)+\min\{r,s\}.
\end{aligned}
\]

\medskip\noindent 
(4) Use (2) and (3) to show that 
\begin{align}
\label{rec*}
g_r(x,y) &=y^{r+1}\left[2+\sum_{s=0}^r (1+x^s)\, 
g_s(x,y)+\sum_{s=r+1}^\infty (1+x^r)\, g_s(x,y)\right], \qquad  r\in \N.
\end{align}
In the same spirit show that
\begin{equation}
\label{rec2p}
g_0(x,y)=y+y\, G(x,y).
\end{equation}

\medskip\noindent 
(5) Abbreviate $g_r=g_r(x,y)$. Prove that 
\begin{equation}
\label{reca}
g_{r+1}-(1+x) y g_r-(1-x) x^r y^{r+2} g_r + x y^2 g_{r-1}=0 \qquad r\in\N.
\end{equation}
To do so, substitute the expressions obtained for $g_{r-1}$, $g_r$ and $g_{r+1}$ from 
\eqref{rec*} into \eqref{reca}, and isolate the terms containing $y^{2r+3}$. The latter 
leads to a rewrite of the left-hand side of \eqref{reca} as 
\begin{equation}
\label{recal}
x^r y^{2r+3} (x-1) \left[2+\sum_{s=0}^{r} (1+x^s) g_s+\sum_{s=r+1}^\infty (1+x^r) g_s\right]
+x^r y^{r+2} (1-x) g_r.
\end{equation}
Use \eqref{rec*} once more to conclude that \eqref{recal} equals zero.

\medskip\noindent 
(6) From \eqref{reca} we see that $(g_r)_{r\in \N_0}$ is determined by $g_0$ and $g_1$, 
while \eqref{rec2p} constitutes a consistency relation that must be met by the solution 
of \eqref{reca}. Thus, $(g_r)_{r\in \N_0}$ belongs to a two-dimensional vector space generated 
by any two linearly independent solutions. For this reason, we look for two particular solutions 
of \eqref{reca} by making an Ansatz. Set $q=xy$, and write $g_r$ in the form 
\begin{equation}
\label{rec1}
g_r=\lambda^r \sum_{l\in\N_0} p_l\,q^{lr}, 
\qquad r\in\N,\,p_0=1,
\end{equation}
where $\lambda=\lambda(y,q)$ and $p_l=p_l(\lambda,y,q)$, $l\in\N$, are to be determined. 
Substitute \eqref{rec1} into \eqref{reca} to obtain
\begin{equation}
\label{rec23}
\begin{aligned}
&\lambda^2 - \lambda (y+q) + yq\\
&+\sum_{l\in\N} q^{l(r-1)} 
\Big[\big(\lambda^2 q^{2l} - \lambda (y+q)q^l + yq\big)p_l
+ \big(\lambda (q-y)yq^l\big)\,p_{l-1}\Big] = 0.
\end{aligned}
\end{equation}
Conclude that \eqref{rec23} is satisfied when
\begin{equation}
\label{rec2}
p_l = \frac{\lambda (y-q) y q^l}{(\lambda q^l-y) (\lambda q^l-q)}\, p_{l-1},    
\qquad l\in\N,
\end{equation}
provided $\lambda$ solves the equation $\lambda^2-\lambda(y+q)+yq=0$, i.e., $\lambda\in
\{\lambda_1,\lambda_2\}=\{y,q\}$.

\medskip\noindent 
(7) Use (6) to show that $g_r=C_1 g_{r,1}+C_2 g_{r,2}$, $r\in \N$, where $C_1$ and 
$C_2$ are functions of $y,q$ and
\begin{equation}
\label{solutions}
g_{r,i} = g_{r,i}(x,y) = 
(\lambda_i)^r \bigg(1+ \sum_{k\in\N} \frac{(\lambda_i)^k 
(y-q)^k y^k q^{\tfrac12 k (k+1)}}{\prod_{l=1}^k (\lambda_i q^l-y) 
(\lambda_i q^l-q)}\ q^{kr}\bigg) \qquad i=1,2,\,r\in\N_0.
\end{equation}
Pick $x>1$ and $0<y<1$ such that $q=xy<1$, and let $r\to\infty$ in \eqref{solutions}. 
This gives 
\[
\lim_{r\to \infty} q^{-r} g_{r,1}(x,y)=0
\quad \text{and} \quad 
\lim_{r\to \infty} q^{-r} g_{r,2}(x,y)=1.
\] 
Next, an easy computation shows that $\lim_{r\to\infty} \frac{1}{r} \log |\cW_r| = 1+\sqrt{2}$,
with $|\cW_r|=g_r(1,1)$. Pick $x>1$ and $0<y<1$ such that $q=xy<1/(1+\sqrt{2})$, and let $r\to\infty$ 
in \eqref{rec*}. This gives 
\[
\lim_{r\to \infty} q^{-r} g_{r}(x,y)=0,
\]
from which it follows that $C_2=0$.

\medskip\noindent 
(8) It remains to determine $C_1$. To that end, note that, by construction, $(g_{r,1})_{r\in \N_0}$ 
satisfies \eqref{reca} for $r=0$ as well. Use  \eqref{rec2p} and \eqref{reca} to show that
\begin{equation}
\begin{aligned}
\label{rec}
&\tfrac12 C_1 g_{0,1} = g_0 =y+yG,\\
&C_1 g_{1,1} = g_1 = a+bG,
\end{aligned}
\end{equation}
with
\[
a=y^2(2+y-xy),\quad b=y^2 (1+x+y-xy).
\]
Eliminate $C_1$ and express $G$ in terms of $g_{0,1}$ and $g_{1,1}$, to obtain 
\[
G(x,y)=\frac{a H(x,y)-y^2}{b H(x,y)-y^2},
\]
where 
\[
H(x,y)=y\,\frac{g_{0,1}(x,y)}{g_{1,1}(x,y)}.
\]
This completes the proof of Lemma~\ref{lem:genfct} with $\bar{g}_0=g_{0,1}$ and
$\bar{g}_1=g_{1,1}$.

\medskip\noindent
(9) Brak, Guttmann and Whittington~\cite{BrGuWh92} show that the function $H(x,y)$ 
can be represented as a continued fraction. This representation allows for an analysis 
of the singularity structure of $G(x,y)$, in particular, for a computation of $y_c(x)$ 
(the radius of convergence of the power series $y \mapsto G(x,y)$) for fixed $x$. 
For instance, from \eqref{reca} it is easily deduced that
\[
G(1/y,y) = \sum_{r\in\N_0} g_r(1/y,y)
= -1 + \sqrt{ \frac{1-y}{1-3y-y^2-y^3} } \qquad (q=1)
\] 
and this has a singularity at $y_c$ solving the cubic equation $1-3y-y^2-y^3=0$.
Fig.~\ref{fig-dircolpolcr} gives the plot of $x \mapsto y_c(x)$ that comes out 
of the singularity analysis. As explained in Section~\ref{compfedirpolatt}, the 
free energy is $f(\gamma)=-\log y_c(\erom^\gamma)$.


\section{Tutorial 3}
\label{appC}

The purpose of this tutorial is to take a closer look at the free energy of
the homogeneous pinning model described in Section~\ref{S3}. Section~\ref{hompinmod} 
recalls the model, Section~\ref{hompinfe} computes the free energy, while 
Section~\ref{hompinregfe} identifies the order of the phase transition.  


\subsection{The model}
\label{hompinmod}

Let $(S_n)_{n\in\N_0}$ be a random walk on $\Z$, i.e., $S_0 = 0$ and 
$S_i-S_{i-1}$, $i\in\N$, are i.i.d. Let $P$ denote the law of $S$. Introducing 
the first return time to zero $\tau = \inf\{n\in\N\colon\, S_n = 0\}$, we 
denote by $R(\cdot)$ its distribution:
\[
R(n) = P(\tau = n) = P\big(S_i \neq 0\,\,\forall\,1 \leq i \leq n-1,\,S_n = 0\big),
\qquad n \in \N.
\]
We require that $\sum_{n \in \N} R(n) = 1$, i.e., the random walk is \emph{recurrent}, 
and we assume the following \emph{tail asymptotics} for $R(\cdot)$ as $n\to\infty$:
\[
R(n) = \frac{c}{n^{1+a}}\,[1+o(1)], \qquad c>0,\,a \in (0,1) \cup (1,\infty).
\]
The exclusion of $a=1$ is for simplicity (to avoid logarithmic corrections in later 
statements). The constant $c$ could be replaced by a \emph{slowly varying function} 
at the expense of more technicalities, which however we avoid. We recall that, for a 
nearest-neighbor symmetric random walk, i.e., when $P(S_1 = 1) = P(S_1 = -1) 
= p$ and $P(S_1 = 0) = 1 - 2p$ with $p \in (0,\frac12)$, the above tail asymptotics 
holds with $\alpha = \frac 12$.

The set of allowed polymer configurations is $\mathcal W_n = \{w=(i,w_i)_{i=0}^n
\colon\,w_0=0,\,w_i\in\Z\,\,\forall\,0\leq i\leq n \}$, on which we define the 
Hamiltonian $H_n^\zeta(w) = - \zeta L_n(w)$, where $\zeta \in \R$ and
\[
L_n(w) = \sum_{i=1}^n \ind_{\{w_i = 0\}}, \qquad w\in\cW_n,
\]
is the so-called \emph{local time} of the polymer at the interface (which has height 
zero). We denote by $P_n$ the projection of $P$ onto $\mathcal W_n$, i.e., $P_n(w) 
= P(S_i = w_i\,\,\forall\,1\leq i\leq n)$ for $w \in \mathcal W_n$. This is the 
\emph{a priori} law for the non-interacting polymer. We define our polymer model 
as the law $P_n^\zeta$ on $\mathcal W_n$ given by
\[
P_n^\zeta(w) = \frac{1}{Z_n^\zeta}\,\erom^{-H_n^\zeta(w)}\,P_n(w), \qquad w\in\cW_n.
\]
The normalizing constant $Z_n^\zeta$, called the \emph{partition function}, is given 
by
\[
Z_n^\zeta = \sum_{w \in \mathcal W_n} \erom^{-H_n^\zeta(w)} \, P_n(w) 
= E_n \left(\erom^{- H_n^\zeta(w)} \right) =
E \left(\erom^{\zeta \sum_{i=1}^n \ind_{\{S_i = 0\}}} \right).
\]
The \emph{free energy} $f(\zeta)$ is defined as the limit
\[
f(\zeta) = \lim_{n\to\infty} \frac 1n \, \log Z_n^\zeta,
\]
which has been shown to exist in {\bf Tutorial 1}. From a technical viewpoint it is 
more convenient to consider the \emph{constrained partition sum} $Z_n^{*,\zeta}$ 
defined by
\[
Z_n^{*,\zeta} = \sum_{ {w \in \mathcal W_n} \atop {w_n = 0} } 
\erom^{-H_n^\zeta(w)} \, P_n(w) 
= E \left(\erom^{\zeta \sum_{i=1}^n \ind_{\{S_i = 0\}}} \, \ind_{\{S_n = 0\}} \right).
\]
As shown in {\bf Tutorial 1}, if we replace $Z_n^\zeta$ by $Z_n^{*,\zeta}$ in the definition 
of $f(\zeta)$, then this does not change the value of the limit. Therefore we may 
focus on $Z_n^{*,\zeta}$.


\subsection{Computation of the free energy}
\label{hompinfe}

We repeat in more detail the derivation of the formula for the free energy $f(\zeta)$
given in Section~\ref{S3}.

\begin{enumerate}
\item 
Prove that $Z_n^{*,\zeta} \geq \erom^\zeta \, P(\tau = n) =  \erom^\zeta \, R(n)$. Deduce 
that $f(\zeta) \geq 0$ for every $\zeta \in \R$.
\item 
Show that $Z_n^{*,\zeta} \leq 1$ for $\zeta \in (-\infty,0]$. Deduce that $f(\zeta) = 0$ 
for every $\zeta \in (-\infty,0]$.
\item 
Henceforth we focus on $\zeta \in [0,\infty)$. Define for $x \in [0,1]$ the generating 
function $\phi(x) = \sum_{n\in\N} R(n) \, x^n$. Observe that $x \mapsto \phi(x)$ is 
strictly increasing with $\phi(0) = 0$ and $\phi(1)=1$. Deduce that for every $\zeta 
\in [0,\infty)$ there is exactly one value $r = r(\zeta)$ that solves the equation 
$\phi(\erom^{-r}) = \erom^{-\zeta}$. Observe that $\tilde R_\zeta(n) = \erom^\zeta\,R(n)\, 
\erom^{-r(\zeta) n}$ defines a probability distribution on $\N$.
\item 
For $n \in \N$ and $1 \leq k \leq n$, denote by $\Theta_{n,k}$ the set consisting of 
$k+1$ points drawn from the interval $\{0,\ldots,n\}$, including $0$ and $n$. More 
explicitly, the elements of $\Theta_{n,k}$ are of the form $j = (j_0,j_1,\ldots,j_k)$ 
with $j_0 = 0$, $j_k = n$ and $j_{i-1} < j_i$ for all $1 \leq i \leq n$. By summing over 
the locations $i$ at which $S_i = 0$, prove that
\[
Z_n^{*,\zeta} = \sum_{k=1}^n \erom^{\zeta k} \sum_{j \in \Theta_{n,k}}
\prod_{i=1}^k R(j_i - j_{i-1}).
\]
Note that this equation can be rewritten as
\[
Z_n^{*,\zeta} = \erom^{r(\zeta)n} \, u_\zeta(n), \qquad
u_\zeta(n) = \sum_{k=1}^n \sum_{j \in \Theta_{n,k}}
\prod_{i=1}^k \tilde R_\zeta(j_i - j_{i-1}).
\]
\item 
For fixed $\zeta\in (0,\infty)$, we introduce a \emph{renewal process} 
$(\tau_n)_{n\in\N_0}$ with law $P_\zeta$, which is a random walk on $\N_0$ 
with positive increments, i.e., $\tau_0 = 0$ and $\tau_n - \tau_{n-1}$, $n\in\N$, 
are i.i.d.\ under $P_\zeta$ with law $P_\zeta(\tau_1 = n) = \tilde R_\zeta(n)$. 
Show that the following representation formula holds:
\[
u_\zeta(n) = \sum_{k=1}^n P_\zeta(\tau_k = n) =
P_\zeta \left( \bigcup_{k\in \N} \{\tau_k = n\} \right).
\]
In particular, $u_\zeta(n) \le 1$. We will use the following important result 
known as the \emph{renewal theorem}: 
\[
\lim_{n\to\infty} u_\zeta(n) = C \in (0,\infty).
\] 
Here $C=C(\zeta)=[\sum_{m\in\N} m \tilde R_\zeta(m)]^{-1} \in (0,\infty)$.
\item 
Conclude that $\lim_{n\to\infty} \frac 1n \log Z_n^{*,\zeta} = r(\zeta)$ for 
every $\zeta\in [0,\infty)$. This means that for $\zeta\in [0,\infty)$ the 
free energy $f(\zeta)$ coincides with $r(\zeta)$ and therefore satisfies the 
equation $\phi(\erom^{-f(\zeta)}) = \erom^{-\zeta}$.
\end{enumerate}

\noindent
Note that (4) and (5) give a sharp asymptotics of the constrained partition sum.
Also note that the argument only uses the renewal structure of the excursions
of the polymer away from the interface, and therefore can be extended to deal
with a priori random processes other than random walks.


\subsection{Order of the phase transition}
\label{hompinregfe}

From the relation $\phi(\erom^{-f(\zeta)}) = \erom^{-\zeta}$ we next derive some
interesting properties of the free energy.

\begin{enumerate}
\item 
Observe that for $x \in (0,1)$ the function $\phi(x) = \sum_{n\in\N} R(n) \, x^n$ 
is strictly increasing, with non-vanishing first derivative, and is real analytic.
Since $\phi(0) = 0$ and $\phi(1) = 1$, its inverse $\phi^{-1}$, defined from $(0,1)$ 
onto $(0,1)$, is real analytic too, by the \emph{Lagrange inversion theorem}. 
Deduce that the free energy $\zeta \mapsto f(\zeta)=-\log\phi^{-1}(\erom^{-\zeta})$ 
restricted to $\zeta \in (0,\infty)$ is real analytic. The same is trivially 
true for $\zeta \in (-\infty, 0)$, since $f(\zeta) = 0$.
\item 
Conclude that the free energy $\zeta \mapsto f(\zeta)$ is not analytic at 
$\zeta = 0$, by the \emph{identity theorem of analytic functions}. Observe 
that nevertheless the free energy is continuous at $\zeta = 0$.
\item 
Introduce the integrated tail probability $\overline R(n)=\sum_{k = n+1}^\infty R(k)$
for $n \in \N_0$. Deduce from our tail assumption on $R(\cdot)$ that $\overline R(n) 
= \frac{c}{a} n^{-a}[1+o(1)]$ as $n\to\infty$.
\item 
Use summation by parts to show that $1 - \phi(x) = (1-x) \sum_{n\in\N_0} \overline 
R(n) x^n$ for $x \in (0,1)$. 
\begin{proof}
\begin{align*}
1 - \phi(x) & = 1 - \sum_{n\in\N} R(n) x^n 
= 1 - \sum_{n\in\N} (\overline R(n-1) - \overline R(n)) x^n\\
& = \left( 1 + \sum_{n\in\N} \overline R(n) x^n \right) - \sum_{n\in\N}
\overline R(n-1) x^n = \sum_{n\in\N_0} 
\overline R(n) x^n - \sum_{n\in\N_0} \overline R(n) x^{n+1} \\
& = (1-x) \sum_{n\in\N_0} \overline R(n) x^n.
\end{align*}
\end{proof}
\item 
Put $\psi(x) = \sum_{n\in\N_0} \overline R(n) x^n$, so that $1-\phi(x) = (1-x) 
\psi(x)$. We first focus on $a \in (1,\infty)$. Show that in that case $\psi(1) 
= E(\tau) = \sum_{n\in\N} n R(n) \in (0,\infty)$. Deduce from $\phi(\erom^{-f(\zeta)}) 
= \erom^{-\zeta}$ that, as $\zeta \downarrow 0$,
\[
f(\zeta) = \frac{1}{E(\tau)}\,\zeta\,[1+o(1)], 
\qquad a \in (1,\infty).
\]
\item 
We next focus on $a \in (0,1)$. Use a Riemann sum approximation to show that, as 
$r \downarrow 0$,
\[
\psi(\erom^{-r}) = \left(\frac{c\Gamma(1-a)}{a}\right)\,r^{a-1}\,[1+o(1)],
\]
where
\[
\Gamma(1-a) = \int_0^\infty \frac{\erom^{-t}}{t^a}\,\drom t 
\in (0,\infty).
\]
\begin{proof}
Note that, for $a \in (0,1)$, $\psi(\erom^{-r}) \uparrow \infty$ as $r \downarrow 0$, 
because $\overline R(n) = \frac{c}{a} n^{-a}\,[1+o(1)]$. Therefore, for any fixed 
$n_0 \in \N$, we can safely neglect the first $n_0$ terms in the sum defining 
$\psi(\cdot)$, because they give a finite contribution as $r \downarrow 0$. 
This gives
\begin{align*}
&\psi(\erom^{-r}) \sim \sum_{n=n_0}^\infty \overline R(n) \erom^{-n r}
\sim \frac{c}{a} \sum_{n=n_0}^\infty \frac{\erom^{-n r}}{n^a}
= \frac{c}{a} \, r^{a - 1}\, \sum_{n=n_0}^\infty 
r\,\frac{\erom^{-n r}}{(nr)^a}\\
&\sim \frac{c}{a} \, r^{a - 1}\, 
\left( \int_0^\infty \frac{\erom^{-t}}{t^a} \, \drom t \right),
\end{align*}
where $\sim$ refers to $n_0\to\infty$. 
\end{proof}
\item 
Deduce from $\phi(\erom^{-f(\zeta)}) = \erom^{-\zeta}$ that, as $\zeta \downarrow 0$,
\[
f(\zeta) = \left(\frac{a}{c\Gamma(1-a)}\right)^{1/a}
\,\zeta^{1/a}\,[1+o(1)], \qquad a \in (0,1).
\]
\end{enumerate}

\noindent
Note that the smaller $a$ is, the more regular is the free energy for 
$\zeta\downarrow 0$, i.e., the higher is the order of the phase 
transition at $\zeta=0$. For $a \in (1,\infty)$ the derivative of 
the free energy is discontinuous at $\zeta=0$, which corresponds to
a first-order phase transition.


\section{Tutorial 4}
\label{appD}

The purpose of this (long) tutorial is to provide further detail on the 
variational approach to the random pinning model described in Section~\ref{S4}.
Section~\ref{ranpinmod} recalls the model, Section~\ref{LDPback} provides
the necessary background on large deviation theory, Section~\ref{wordlet}
explains the large deviation principles for the empirical process of random
words cut out from a random letter sequence according to a renewal process, 
while Section~\ref{ranpinempproc} shows how the latter are applied to the 
random pinning model to derive a variational formula for the critical curve. 


\subsection{The model}
\label{ranpinmod}

Let $S = (S_n)_{n\in\N_0}$, be a Markov chain on a countable space $\Upsilon$ 
that contains a marked point $*$. Let $P$ denote the law of $S$, and assume 
that $S_0 = *$. We introduce the first return time to $*$, namely, $\tau= 
\inf\{n \in\N\colon\,S_n = *\}$, and we denote by $R(\cdot)$ its distribution:
\[
R(n) = P(\tau = n) = P(S_i \neq * \,\,\forall\, 1 \leq i \leq n-1,\ S_n = *),
\qquad n \in \N.
\]
We require that $\sum_{n \in \N} R(n) = 1$, i.e., the Markov chain is \emph{recurrent}, 
and assume the following \emph{logarithmic tail asymptotics} as $n\to\infty$:
\[
\lim_{n\to\infty} \frac{\log R(n)}{\log n} = -(1+a) \,, \qquad 
\text{with $a \in [0,\infty)$} \,.
\]
For a nearest-neighbor and symmetric random walk on $\Z$, i.e., 
\[
P(S_1 = 1) = P(S_1 = -1) = p, \qquad P(S_1 = 0) = 1 - 2p, \qquad p \in (0,\frac 12),
\] 
this asympotics holds with $a = \frac 12$.

The set of allowed polymer configurations is $\cW_n = \{w = (i, w_i)_{i=0}^n\colon\,
w_0 = *,\, w_i \in \Upsilon\,\,\forall\, 0 < i \leq n \}$ on which we define the 
Hamiltonian
\[
H_n^{\beta,h,\omega}(w) = -\sum_{i=0}^n (\beta \omega_i - h) \ind_{\{w_i = *\}},
\]
where $\beta, h \geq 0$ are two parameters that tune the interaction strength
and $\omega = (\omega_i)_{i\in\N_0}$ is the \emph{random environment}, a typical 
realization of a sequence of i.i.d.\ $\R$-valued random variables with marginal 
law $\mu_0$. The law of the full sequence $\omega$ is therefore $\P = 
\mu_0^{\otimes \N_0}$. We assume that $M(\beta) = \E(\erom^{\beta \omega_0})<\infty$ 
for all $\beta \in \R$, and w.l.o.g.\ we assume that $\E(\omega_0)=0$ and 
$E(\omega_0^2)=1$.

We denote by $P_n$ the projection onto $\cW_n$ of the law of $S$, i.e.,
$P_n(w) = P(S_i = w_i\,\,\forall\,0\leq i \le n)$ for $w \in \cW_n$.
This is the \emph{a priori} law for the non-interacting polymer. We 
define our polymer model as the law $P_n^{\beta,h,\omega}$ on $\cW_n$ 
given by
\[
P_n^{\beta,h,\omega}(w) = \frac{1}{Z_n^{\beta,h,\omega}} 
\,\erom^{-H_n^{\beta,h,\omega}(w)}\,P_n(w).
\]
The normalizing constant $Z_n^{\beta,h,\omega}$ is the \emph{partition sum}
and is given by
\[
\begin{aligned}
Z_n^{\beta,h,\omega} &= \sum_{w \in \cW_n}
\erom^{-H_n^{\beta,h,\omega}(w)} \, P_n(w)\\ 
&= E_n \left( \erom^{-H_n^{\beta,h,\omega}(w)} \right) 
= E \left( \erom^{\sum_{i=0}^n (\beta \omega_i - h) \ind_{\{w_i = *\}}} \right).D
\end{aligned}
\]
The \emph{quenched free energy} $f^{\rm{que}}(\beta,h)$ is defined as the 
limit
\[
f^{\rm{que}}(\beta,h) = \lim_{n\to\infty} 
\frac 1n \log Z_n^{\beta,h,\omega} 
\qquad \text{$\P$-a.s.\ and in $L^1(\P)$},
\]
which has been shown in Tutorial 1 to exist and to be non-random. It can be 
easily shown that $f^{\rm{que}}(\beta, h) \geq 0$, which motivates the 
introduction of a \emph{localized phase} $\cL$ and a \emph{delocalized phase}
$\cD$ defined by
\[
\mathcal L = \{(\beta, h)\colon\,f^{\rm{que}}(\beta, h)>0\}, \qquad
\mathcal D = \{(\beta, h)\colon\,f^{\rm{que}}(\beta, h) = 0\}.
\]
It follows from the convexity and the monotonicity of the free energy
that these phases are separated by a \emph{quenched critical curve} 
\[
\beta \mapsto h_c^{\rm{que}}(\beta) = \inf\{h \in \R\colon\, f^{\rm{que}}
(\beta,h) = 0\}.
\] 
In the remainder of this tutorial we develop insight
into the \emph{variational formula} for $h_c^{\rm{que}}$ that was put
forward in Section~\ref{S5}.

Note that $f^{\rm{que}}(\beta, h) = \lim_{n\to\infty} \frac 1n 
\E(\log Z_n(\beta, h,\omega))$. Interchanging the expectation $\E$ 
and the logarithm, we obtain the \emph{annealed free energy}:
\[
f^{\rm{ann}}(\beta,h) = \lim_{n\to\infty} \, \frac 1n \log 
\E\big( Z_n^{\beta,h,\omega}\big) 
= \lim_{n\to\infty} \, \frac 1n \log 
E \big(\erom^{(\log M(\beta) - h)  \sum_{i=0}^n \ind_{\{S_i = *\}}} \big),
\]
which is nothing but the free energy $f(\zeta)$ of a homogeneous pinning
model with $\zeta = \log M(\beta) - h$. Recall from Tutorial 3 that 
$f(\zeta)>0$ for $\zeta>0$ and $f(\zeta) = 0$ for $\zeta \leq 0$. Introducing 
the \emph{annealed critical curve}
\[
h_c^{\rm{ann}}(\beta)= \inf\{h \in \R\colon\, f^{\rm{ann}}(\beta,h) = 0\},
\] 
we find that $h_c^{\rm{ann}}(\beta) = \log M(\beta)$. Jensen's inequality yields 
$f^{\rm{que}}(\beta, h) \leq f^{\rm{ann}}(\beta, h)$, so that $h_c^{\rm{que}}
(\beta) \leq h_c^{\rm{ann}}(\beta)$. The disorder is said to be \emph{irrelevant} 
if $h_c^{\rm{que}}(\beta) = h_c^{\rm{ann}}(\beta)$ and \emph{relevant} if 
$h_c^{\rm{que}}(\beta) < h_c^{\rm{ann}}(\beta)$.


\subsection{Some background on large deviation theory}
\label{LDPback}

Before we proceed with our analysis of the copolymer model we make an intermezzo,
namely, we give a brief summary of some basic large deviation results. For more 
details, see the monographs by Dembo and Zeitouni~\cite{DeZe98} and den 
Hollander~\cite{dHo00}.


\subsubsection{Relative entropy}

Let $\nu,\rho$ be two probabilities on a measurable space $(\Gamma,\mathcal G)$, 
i.e., $\nu,\rho \in \mathcal M_1(\Gamma)$, the space of probability measures on 
$\Gamma$. For $\nu \ll \rho$ (i.e., $\nu$ is absolutely continuous with respect to 
$\rho$), we denote by $\frac{\mathrm d \nu}{\mathrm d \rho}$ the corresponding 
Radon-Nikod\'ym derivative and we define the \emph{relative entropy} $h(\nu|\rho)$ 
of $\nu$ with respect to $\rho$ by the formula
\[
h(\nu | \rho) 
= \int_\Gamma \log \left( \frac{\mathrm d \nu}{\mathrm d \rho} \right) 
\,\mathrm d \nu
= \int_\Gamma \left( \frac{\mathrm d \nu}{\mathrm d \rho} \right)
\log \left( \frac{\mathrm d \nu}{\mathrm d \rho} \right) \, 
\mathrm d \rho.
\]
For $\nu \not\ll \rho$, we simply put $h(\nu | \rho) = \infty$. Note that the 
function $g(x) = x\log x$ with $g(0)= 0$ is convex (hence continuous) and is
bounded from below on $[0,\infty)$, so that the integral defining $h(\nu | \rho)$ 
is well-defined in $\R \cup \{\infty\}$.

\begin{itemize}
\item
Use Jensen's inequality to show that $h(\nu | \rho) \geq 0$ for all $\nu, \rho$, 
with $h(\nu | \rho) = 0$ if and only if $\nu = \rho$.
\end{itemize}

For fixed $\rho$, the function $\nu \mapsto h(\nu|\rho)$ is convex on 
$\mathcal M_1(\Gamma)$. Note that if $\Gamma$ is a finite set, say 
$\Gamma = \{1, \ldots, r\}$ with $r \in \N$, then we can write
\[
h(\nu | \rho) 
= \sum_{i=1}^r \nu_i \, \log \left(\frac{\nu_i}{\rho_i}\right).
\]


\subsubsection{Sanov's Theorem in a finite space}

Let $Y = (Y_n)_{n\in\N}$ be an i.i.d.\ sequence of random variables taking 
values in a finite set, which we identify with $\Gamma = \{1, \ldots, r\}$ 
with $r \in \N$. Let $\rho = \{\rho_i\}_{i=1}^r$ with $\rho_i = P(Y_1 = i)>0$ 
be the marginal law of this random sequence. Note that $\rho \in \mathcal 
M_1(\Gamma)$. For $n \in \N$ we define the \emph{empirical measure}
\[
L_n = \frac 1n \, \sum_{k=1}^n \delta_{Y_k},
\]
where $\delta_x$ denotes the Dirac mass at $x$. Note that $L_n$ is a 
\emph{random element} of $\mathcal M_1(\Gamma)$, i.e., a random variable 
taking values in $\mathcal M_1(\Gamma)$, which describes the relative 
frequency of the ``letters'' appearing in the sequence $Y_1,\ldots,Y_n$.

The space $\mathcal M_1(\Gamma)$ can be identified with the simplex
$\{x \in (\R^+)^r\colon\,\sum_{i=1}^r x_i = 1\} \subset (\R^+)^r$,
and hence $\mathcal M_1(\Gamma)$ can be equipped with the standard 
Euclidean topology, and we can talk about convergence in $\mathcal M_1(\Gamma)$
(which is nothing but convergence of every component). With this 
identification we have $L_n = \{L_n(i)\}_{i=1}^r$, where $L_n(i)$ 
is the relative frequency of the symbol $i$ in the sequence $Y_1, \ldots, Y_n$, 
i.e., $L_n(i) = \frac{1}{n} \sum_{k=1}^n \ind_{\{Y_k = i\}}$.

\begin{itemize}
\item 
Show that the strong law of large numbers yields the a.s.\ convergence
$\lim_{n\to\infty} L_n = \rho$, where the limit is in $\mathcal M_1(\Gamma)$.
\end{itemize}

The purpose of large deviation theory is to quantify the probability that $L_n$
differs from its limit $\rho$: given a $\nu \in \mathcal M_1(\Gamma)$ different
from $\mu$, what is the probability that $L_n$ is close to $\nu$? Take for 
simplicity $\nu = \{\nu_i\}_{i=1}^r$ of the form $\nu_i = \frac{k_i}{n}$ 
with $k_i \in \N$ and $\sum_{i=1}^r k_i = n$. (Note that this is the family of 
laws that can be attained by $L_n$.)
\begin{itemize}
\item 
Prove that $P(L_n = \nu) = n! \prod_{i=1}^r \frac{{\rho_i}^{k_i}}{k_i !}$.
\item
Use Stirling's formula $n! = n^n \erom^{-n + o(n)}$ to deduce that $P(L_n = \nu) 
= \erom^{-n h(\nu | \rho) + o(n)}$, where $h(\nu|\rho)$ is the relative entropy 
defined above.
\end{itemize}

In this sense, the relative entropy $h(\nu | \rho)$ gives the rate of exponential 
decay for the probability that $L_n$ is close to $\nu$ instead of $\rho$. More generally, 
it can be shown that if $O$ and $C$ are, respectively, an open and a closed subset of 
$\mathcal M_1(\Gamma)$, then, with the notation $I(\nu) = h(\nu|\rho)$, the following 
relations hold:
\begin{equation}
\label{eq:ld}
\begin{aligned}
&\liminf_{n\to\infty} \,\frac 1n \, \log P(L_n \in O) \,\ge\, -\inf_{\nu \in O} I(\nu),\\
&\limsup_{n\to\infty} \,\frac 1n \, \log P(L_n \in C) \,\le\, -\inf_{\nu \in C} I(\nu).
\end{aligned}
\end{equation}
Whenever the above inequalities hold, we say that the sequence of random variables 
$(L_n)_{n\in\N}$ satisfies the \emph{large deviation principle} (LDP) with rate $n$
and with rate function $I(\cdot)$.


\subsubsection{Sanov's theorem in a Polish space}

In the previous section we have worked under the assumption that the space 
$\Gamma$ is finite. However, everything can be generalized to the case when 
$\Gamma$ is \emph{Polish} (a complete separable metric space) equipped with 
the Borel $\sigma$-field. Let $Y = (Y_n)_{n\in\N}$ be an i.i.d.\ sequence of 
random variables taking values in $\Gamma$ and denote by $\rho \in \mathcal 
M_1(\Gamma)$ the law of $Y_1$. We equip the space $\mathcal M_1(\Gamma)$ of 
probability measures on $\Gamma$ with the topology of weak convergence (i.e., 
$\nu_n \to \nu$ in $\mathcal M_1(\Gamma)$ if and only if $\int f \mathrm d 
\nu_n \to \int f \mathrm d \nu$ for every bounded and continuous $f\colon\,
\Gamma \to \R$). This topology turns $\mathcal M_1(\Gamma)$ into a Polish space 
too, which we equip with the corresponding Borel $\sigma$-field. We can therefore 
speak of convergence in $\mathcal M_1(\Gamma)$ as well as random elements of 
$\mathcal M_1(\Gamma)$ (random variables taking values in $\mathcal M_1(\Gamma)$).

In particular, the empirical measure $L_n$ introduced above is well defined 
in this generalized setting as a random element of $\mathcal M_1(\Gamma)$.
With the help of the ergodic theorem it is possible to show that, in analogy 
with the case of finite $\Gamma$, $\lim_{n\to\infty} L_n = \rho$ a.s.\ in 
$\mathcal M_1(\Gamma)$. Also, the large deviation inequalities mentioned above 
continue to hold, again with $I(\nu) = h(\nu|\rho)$ as defined earlier. The 
formal tool to prove this is the \emph{projective limit LDP} of Dawson and 
G\"artner~\cite{DaGa87}.


\subsubsection{Process level large deviations}

One can take a step further and consider an extended empirical measure,
keeping track of ``words'' instead of single ``letters''. More precisely,
let again $Y = (Y_n)_{n\in\N}$ be an i.i.d.\ sequence of random variables
taking values in a Polish space $\Gamma$ and denote by $\rho \in \mathcal 
M_1(\Gamma)$ the law of $Y_1$. For $\ell \in \N$ fixed, one can consider 
the empirical distribution of $\ell$ consecutive variables (``words consisting
of $\ell$ letters``) appearing in the sequence $Y_1, \ldots, Y_n$:
\[
L_n^\ell = \frac{1}{n} \sum_{i=1}^n 
\delta_{(Y_i, Y_{i+1}, \ldots, Y_{i+\ell-1})},
\]
where we use for convenience periodic boundary conditions: $Y_{n+i} = Y_i$
for $i=1, \ldots, \ell -1 $. Note that $L_n^\ell$ is a random element
of the space $\mathcal M_1(\Gamma^\ell)$ of probability measures on $\Gamma^\ell$.
One can show that $\lim_{n\to\infty} L_n^\ell = \rho^{\otimes \ell}$ a.s.\
and one can obtain the large deviations of $L_n^\ell$ with an explicit 
rate function (not pursued here).

One can even go beyond and consider the empirical measure associated with 
``words of arbitrary length''. To do so, it is convenient to denote by 
$(Y_1, \ldots, Y_n)^{\rm{per}}$ the \emph{infinite} sequence obtained 
by repeating periodically $(Y_1, \ldots, Y_n)$, i.e., 
\[
((Y_1, \ldots, Y_n)^{\rm{per}})_{mn + j} = Y_j \mbox{ for } m \in \N_0 
\mbox{ and } j \in \{1, \ldots, n\}.
\] 
Note that $(Y_1, \ldots, Y_n)^{\rm{per}}$ takes values in $\Gamma^\N$. Denoting 
by $\theta$ the left shift on $\Gamma^\N$, i.e., $(\theta x)_i = x_{i+1}$ for 
$x = (x_i)_{i\in\N}$, we can therefore introduce the \emph{empirical process}
\[
R_n = \frac{1}{n} \sum_{i=0}^{n-1} 
\delta_{\theta^i (Y_1, \ldots, Y_n)^{\rm{per}}},
\]
which is by definition a random element of the space $\mathcal M_1^{\rm{inv}}
(\Gamma^\N)$ of shift-invariant probability measures on the Polish space 
$\Gamma^\N$, which is equipped with the product topology and the product 
$\sigma$-field.

Again, one can show that $\lim_{n\to\infty} R_n = \rho^{\otimes \N}$ a.s.\
on $\mathcal M_1^{\rm{inv}}(\Gamma^\N)$. Furthermore, $(R_n)_{n\in\N}$ 
satisfies an LDP, namely, for every open set $O$ and closed set $C$ in $\mathcal 
M_1^{\rm{inv}}(\Gamma^\N)$:
\[
\begin{aligned}
&\liminf_{n\to\infty} \, \frac 1n \, 
\log P(R_n \in O) \geq -\inf_{\nu \in O} I(\nu),\\
&\limsup_{n\to\infty} \, \frac 1n \, 
\log P(R_n \in C) \leq -\inf_{\nu \in C} I(\nu),
\end{aligned}
\]
where the rate function $I(\nu) = H(\nu | \rho^{\otimes\N})$ is the so-called 
\emph{specific relative entropy}:
\[
H(\nu | \rho^{\otimes\N}) 
= \lim_{n\to\infty} \, \frac{1}{n} h(\pi_n \nu | \rho^{\otimes n}),
\]
where $h(\,\cdot\,|\,\cdot\,)$ is the relative entropy defined earlier and 
$\pi_n$ denotes the projection from $\Gamma^\N$ to $\Gamma^n$ onto the first 
$n$ components. The limit can be shown to be non-decreasing: in particular, 
$H(\nu | \rho) = 0$ if and only if $\pi_n \nu = \rho^{\otimes n}$ for every 
$n\in\N$, i.e., $\nu = \rho^{\otimes \N}$.


\subsection{Random words cut out from a random letter sequence}
\label{wordlet}

Let us apply the large deviation theory sketched in the previous section to 
study the sequence of random words cut out from a random letter sequence 
according to an independent renewal process. Our ``alphabet'' will be $\R$, 
while $\widetilde\R= \bigcup_{n\in\N} \R^k$ will be the set of finite words 
drawn from $\R$, which can be metrized to become a Polish space.

We recall from \ref{ranpinmod} that $\omega = (\omega_i)_{i\in\N_0}$ with law $\P$ 
is an i.i.d.\ sequence of $\R$-valued random variables with marginal distribution 
$\mu_0$, and $S = (S_n)_{n\in\N_0}$ with law $P$ is a recurrent Markov chain on the 
countable space $\Upsilon$ containing a marked point~$*$. The sequences 
$\omega$ and $S$ are independent. From the sequence of letters $\omega$ we 
cut out a sequence of words $Y = (Y_i)_{i\in\N}$ using the successive excursions 
of $S$ out of $*$. More precisely, we let $T_k$ denote the epoch of the $k$-th 
return of $S$ to $*$:
\[
T_0 = 0, \qquad T_{k+1} = \inf\{m>T_k\colon\,S_m = *\},
\]
and we set $Y_i = (\omega_{T_{i-1}},\omega_{T_{i-1} + 1},\ldots,\omega_{T_{i}-1})$.
Note that $Y = (Y_i)_{i\in\N} \in \widetilde\R^\N$.

We next define the empirical process associated with $Y$:
\[
R_n = \frac{1}{n} \sum_{i=0}^{n-1}
\delta_{\tilde\theta^i (Y_1, \ldots, Y_n)^{\rm{per}}},
\]
where we denote by $\tilde \theta$ the shift acting on $\widetilde \R$. By definition, 
$R_n$ is a random element of the space $\mathcal M_1^{\rm{inv}}(\widetilde \R^\N)$ of 
shift-invariant probabilities on $\widetilde \R^\N$.

We may look at $Y$ and $R_n$ in at least two ways: either under the law $P^*=\P 
\otimes P$ (= annealed) or under the law $P$ (= quenched). We start with the 
\emph{annealed viewpoint}.

\begin{itemize}
\item 
Show that under $P^*$ the sequence $Y$ is i.i.d.\ with marginal law $q_0$ given by
\[
q_0(\mathrm d x_1, \ldots, \mathrm d x_n) \,=\, R(n) \,
\mu_0(\mathrm d x_1) \times \cdots \times \mu_0(\mathrm d x_n).
\]
\item 
Conclude from \ref{LDPback} that under $P^*$ the sequence $(R_n)_{n\in\N}$ 
satisfies an LDP on $\mathcal M_1^{\rm{inv}}(\Gamma^\N)$ with rate function 
$I^{\rm{ann}}(Q) = H(Q | \mu_0^{\otimes\N})$, the specific relative entropy of
$Q$ w.r.t.\ $\P=\mu_0^{\otimes\N}$.
\end{itemize}
In words, the probability under $P^*$ that the first $n$ words cuts out 
of $\omega$ by $S$, periodically extended to an infinite sequence, have an 
empirical distribution that is close to a law $Q \in \mathcal M_1^{\rm{inv}}
(\Gamma^\N)$ decays exponentially in $n$ with rate $I^{\rm{ann}}(Q)$: 
\[
P^*(R_n \approx Q) = \exp[-n I^{\rm{ann}}(Q) + o(n)].
\] 
We note that $I^{\rm{ann}}(Q) \geq 0$ and $I^{\rm{ann}}(Q) = 0$ if and only if 
$Q = \mu_0^{\otimes \N}$.

We next consider the \emph{quenched viewpoint}, i.e., we fix $\omega$ and we write 
$R_n^\omega$ instead of $R_n$. It is intuitively clear that, when the average 
is over $S$ only, it is more difficult to observe a large deviation. Therefore, 
if under $P$ the sequence $(R_n^\omega)_{n\in\N}$ satisfies an LDP on 
$\mathcal M_1^{\rm{inv}}(\Gamma^\N)$ with rate function $I^{\rm{que}}$,
i.e., if $P(R_n^\omega \approx Q) = \exp[-n I^{\rm{que}}(Q) + o(n)]$, then
we should have $I^{\rm{que}}(Q) \geq I^{\rm{ann}}(Q)$. Indeed, this is the 
case: the difference between $I^{\rm{que}}(Q)$ and $I^{\rm{ann}}(Q)$ can 
in fact be explicitly quantified. For details we refer to Birkner, Greven
and den Hollander~\cite{BiGrdHo10}.


\subsection{The empirical process of words and the pinning model}
\label{ranpinempproc}

We are finally ready to explore the link between the process of random words 
$Y$ described in the previous section and our random pinning model. Define
for $z \in [0,1]$ the generating function
\[
G(z) = \sum_{n\in\N} z^n\,Z_n^{*,\beta,h,\omega},
\]
where $Z_n^{*,\beta,h,\omega}$ denotes the \emph{constrained} partition sum
\[
Z_n^{*,\beta, h,\omega} =
E\left(\erom^{\sum_{i=0}^{n-1} (\beta\omega_i - h) \ind_{\{S_i = *\}}}
\,\ind_{\{S_n = *\}} \right).
\]
We recall that $Z_n^{*,\beta,h,\omega}$ yields the same free energy
as the original partition function $Z_n^{\beta,h,\omega}$, i.e.,
\[
f^{\rm{que}}(\beta, h) = \lim_{n\to\infty} \frac 1n\,
\log Z_n^{*,\beta,h, \omega}
\qquad \text{$\P$-a.s.\ and in $L^1(\mathbb{P})$}.
\]

\begin{itemize}
\item 
Prove that the radius of convergence $\overline z$ of $G(z)$ equals 
$\erom^{-f^{\rm{que}}(\beta,h)}$.
\item 
In analogy with Tutorial 3, show that
\[
z^n \, Z_n^{*,\beta,h,\omega} =
\sum_{N \in \N} \sum_{0 = k_0 < k_1 < \cdots < k_N = n} \prod_{i=1}^N
z^{k_i - k_{i-1}}\, R(k_i - k_{i-1})\,\erom^{\beta \omega_{k_{i-1}} - h}.
\]
\item 
Deduce that $G(z) = \sum_{N\in\N} F_N^{\beta,h,\omega}(z)$, where
\begin{align*}
F_N^{\beta,h,\omega}(z) &= \sum_{0 = k_0 < k_1 < \cdots < k_N < \infty}
\, \prod_{i=1}^N z^{k_i - k_{i-1}}\,R(k_i - k_{i-1})\,
\erom^{\beta \omega_{k_{i-1}} - h}\\
&= E \left(\prod_{i=1}^N z^{T_i - T_{i-1}}\,\erom^{\beta \omega_{T_{i-1}} - h}\right)\\
&= \erom^{N[S_N^{\beta,\omega}(z)-h]}
\end{align*}
with
\[
S_N^{\beta,\omega}(z) = \frac1N\,\log E\left(\exp\left[
\sum_{i=1}^N (T_i - T_{i-1})\log z + \beta\omega_{T_{i-1}}\right]\right).
\]
\end{itemize}

Given an infinite ``sentence'' $y = (y_k)_{k\in\N} \in \widetilde\R^\N$, we denote 
by $y_1 \in \widetilde\R$ its first ``word''. For a ``word'' $x \in \widetilde\R$, 
we denote by $\ell(x)$ the length of $x$ and by $c(x)$ the first letter of $x$.

\begin{itemize}
\item 
Recalling that $Y_i = (\omega_{T_{i-1}}, \omega_{T_{i-1} + 1}, \ldots,\omega_{T_{i}-1})$, 
with the $T_i$'s the hitting times of the interface $*$, prove that
\[
\begin{aligned}
m(R_N^\omega) &=
\int_{\widetilde\R^\N} \ell(y_1) \, R_N^\omega(\mathrm d y)
= \frac{1}{N} \sum_{i=1}^N \ell(Y_i) 
= \frac{1}{N} \sum_{i=1}^N (T_i - T_{i-1}),\\
\Phi(R_N^{\omega}) &= 
\int_{\widetilde\R^\N} c(y_1) \, R_N^\omega(\mathrm d y)
= \frac{1}{N} \sum_{i=1}^N c(Y_i) 
= \frac{1}{N} \sum_{i=1}^N \omega_{T_{i-1}}.
\end{aligned}
\]
Hence
\[
S_N^{\beta,\omega}(z) = \frac1N\,\log  
E\left(\exp\Big[N \big[m(R_N^\omega)\log z + \beta \Phi(R_N^\omega)\big] 
\Big]\right).
\]
\end{itemize}
This shows that $S_N^{\beta,\omega}(z)$ is the expectation of an exponential 
function of $R_N^\omega$. It is therefore clear that the properties of the 
generating function $G(z)$, in particular, its radius of convergence $\overline z$ 
(and hence the quenched free energy) can be deduced from the large deviation 
properties of $R_N^\omega$. Let us therefore set
\[
S^{\rm{que}}(\beta,z) = \limsup_{N\to\infty} S_N^{\beta,\omega}(z) 
\]
and $S^{\rm{que}}(\beta, 1-) = \lim_{z \uparrow 1} S^{\rm{que}}(\beta, z)$.

\begin{itemize}
\item 
Prove that if $h > S^{\rm{que}}(\beta, z)$ then $G(z) < \infty$,
while if $h < S^{\rm{que}}(\beta, z)$ then $G(z) = \infty$.
\item 
Deduce that if $S^{\rm{que}}(\beta, 1-) < h$ then $f^{\rm{que}}(\beta, h) = 0$,
while if $S^{\rm{que}}(\beta, 1-) > h$ then $f^{\rm{que}}(\beta, h) > 0$.
Therefore $h_c^{\rm{que}}(\beta) = S^{\rm{que}}(\beta, 1-)$.
\end{itemize}

Finally, with the help of Varadhan's lemma in large deviation theory it can be 
shown that
\[
h_c^{\rm{que}}(\beta) = S^{\rm{que}}(\beta, 1-) 
= \sup_{Q \in \mathcal M_1^{\rm{inv}}(\widetilde\R^\N)}
\big[ \beta \Phi(Q) - I^{\rm{que}}(Q) \big].
\]
This gives an explicit variational characterization of the quenched critical curve.
An analogous characterization holds for the annealed critical curve too. For details
see Cheliotis and den Hollander~\cite{ChdHo10}.


\section{Tutorial 5}
\label{appE}

In this tutorial we return to the copolymer model treated in 
Sections~\ref{S5.1}--\ref{S5.4} and prove Theorem~\ref{thm:hclb} (lower bound 
on the critical curve) and Theorem~\ref{thm:ordphtr} (order of the phase 
transition is at least two). Section~\ref{copolE1} recalls the model, 
Section~\ref{copollb} proves Theorem~\ref{thm:hclb}, while Section~\ref{copolordphtr} 
proves Theorem~\ref{thm:ordphtr}. 


\subsection{The model}
\label{copolE1}

We begin by recalling some of the notation used in Sections~\ref{S5.1}--\ref{S5.4}.

\medskip\noindent
\textit{Configurations of the copolymer.} 
For $n \in \N$ the allowed configurations of the copolymer are modelled by the 
$n$-step paths of a $(1+1)$-dimensional simple random walk $S=(S_i)_{i\in\N_0}$,
i.e., $S_0=0$ and $(S_i-S_{i-1})_{i\in\N}$ is an i.i.d.\ sequence of Bernoulli 
trials with
\[
P(S_1=+1)=P(S_1=-1)=\tfrac12,
\]
where we write $P$ for the law of $S$. The set of $n$-step paths is denoted by 
$\cW_n$.

\medskip\noindent
\textit{Disorder: randomness of the monomer types.}
The monomers in the copolymer are either hydrophilic or hydrophobic. Their 
order of appearance is encoded by an i.i.d.\ sequence $\omega=(\omega_i)_{i\in\N}$ 
of Bernouilli trials with
\[
\P(\omega_1=+1)=\P(\omega_1=-1)=\tfrac12,
\]
where we write $\P$ for the law of $\omega$, and we assume that $\omega$ and 
$S$ are independent.

\medskip\noindent 
\textit{Interaction polymer-interface.} 
The medium is made up of oil and water separated by a flat interface located at height 
$0$, oil being above the interface and water below. The copolymer gets an energetic 
reward for each monomer it puts in its preferred solvent. Thus, $S\in\cW_n$ has 
energy 
\[
H_n^{\beta,h,\omega}(S) = - \beta \sum_{i=1}^n (\omega_i + h) (\Delta_i - 1),
\qquad S \in \cW_n,
\]  
where $\Delta_i = \mathrm{sign}(S_{i-1},S_i)$ and $\beta\in (0,\infty)$ stands for 
the inverse temperature. The presence of the $-1$ in this Hamiltonian is for later 
convenience and has no effect on the polymer measure. Indeed, by the law of large 
numbers for $\omega$, we have $\beta\sum_{i=1}^n (\omega_i+h) = \beta h n + o(n)$. 
The term $\beta h n$ can be moved to the normalizing partition sum, while the 
term $o(n)$ does not affect the free energy in the limit as $n\to\infty$.  

\medskip\noindent 
\textit{Partition function and free energy.} 
For fixed $n$, the quenched (= frozen disorder) partition sum and finite-volume 
free energy are defined as 
\[
Z_n^{\beta,h,\omega} = E\big(\erom^{-H_n^{\beta,h,\omega}(S)}\big),
\qquad 
g_n^\omega(\beta,h)=\tfrac1n \log Z_n^{\beta,h,\omega}.
\]
Recall that the localized phase $\cL$ and the delocalized phase $\cD$ are defined 
by
\[
\cL = \{(\beta,h)\colon\,g^\mathrm{que}(\beta,h)>0\}, 
\qquad \cD = \{(\beta,h)\colon\,g^\mathrm{que}(\beta,h)=0\},
\]
where $g^\mathrm{que}(\beta,h)=\lim_{n\to \infty} g_n^\omega(\beta,h)$ $\omega$-a.s.


\subsection{Lower bound on the critical curve}
\label{copollb}

Fix $l\in 2\N$. For $j \in\{1,\dots,n/l\}$ (for simplicity we pretend that 
$n/l$ is integer), let 
\[
I_j=\{(j-1)l+1,\dots,jl\}, \qquad \Omega_j=\sum_{i\in I_j}\omega_i.
\] 
Fix $\delta \in (0,1]$, and define
\[
i_0^{\omega}=0, \qquad  
i_{j+1}^\omega=\inf\{k\geq i_j^\omega +2\colon\,
\Omega_k\leq -\delta l\}, \quad j\in\N.
\]
These are the stretches of length $l$ where the empirical average of the disorder is 
$\leq -\delta$, trimmed so that no two stretches occur next to each other, which 
guarantees that $\tau_j^\omega=i_{j+1}^\omega-i_j^\omega-1$, $j\in\N$, are $\geq 1$. 
(The copolymer gets a substantial reward when it moves below the interface during 
these stretches.) Let  
\[
t_n^\omega=\sup\{j\in\N_0\colon\,i_j^\omega \leq n/l\}.
\] 
In the estimate below we will need the subset of paths defined by
(see Fig.~\ref{fig-BG1})
\[
\cW_n^\omega=\big\{S\colon\,S_i<0\,\,\forall\,i\in \cup_{j=1}^{t_n^\omega} 
I_{i_j^\omega}\backslash \partial I_{i_j^\omega}\big\}
\cap \big\{S\colon\,S_i>0\,\,\forall\, i\in \{0,\dots,n\}
\setminus \cup_{j=1}^{t_n^\omega} I_{i_j^\omega} \big\}.
\]

\begin{figure}[htbp]
\vspace{1cm}
\begin{center}
\setlength{\unitlength}{0.25cm}
\begin{picture}(15,5)(13,-3)
\put(19,0){\circle*{.35}}
\put(22,0){\circle*{.35}}
{\qbezier(39,0.4)(39,0)(39,-0.4)}
\put(38.5,-2){$n$}
{\qbezier(0,0)(5,0)(15,0)}
{\thicklines
\qbezier[440](0,0)(13,0)(24.5,0)
}
{\thicklines
\qbezier[14](24.5,0)(27,0)(31,0)
}
{\thicklines
\qbezier[340](31,0)(32,0)(41,0)
}
\put(10,0){\circle*{.35}}
\put(0,0){\circle*{.35}}
\put(13,0){\circle*{.35}}
\put(33,0){\circle*{.35}}
\put(36,0){\circle*{.35}}
\put(4,-4.5){$\tau_1^\omega l$}
\put(15.3,-4.5){$\tau_2^\omega l$}
\put(11,4.3){$I_{i_1^\omega}$}
\put(20,4.3){$I_{i_2^\omega}$}
\put(34,4.3){$I_{i_{t_n^\omega}}$}
{\thicklines
{\qbezier[6](0,0)(0,-1.5)(0,-3)}
}
{\thicklines
{\qbezier[6](10,0)(10,-1.5)(10,-3)}
}
{\thicklines
{\qbezier[6](13,0)(13,-1.5)(13,-3)}
}
{\thicklines
{\qbezier[6](19,0)(19,-1.5)(19,-3)}
}
\put(0.2,-3){\vector(1,0){9.7}}
\put(9.7,-3){\vector(-1,0){9.8}}
\put(13.2,-3){\vector(1,0){5.8}}
\put(18.7,-3){\vector(-1,0){5.6}}
{{\thicklines{
{\qbezier[166](0,0)(5,6)(10,0)}}}}
{{\thicklines{
{\qbezier[166](10,0)(11.5,-3)(13,0)}}}}
{{\thicklines{
{\qbezier[166](13,0)(16,4)(19,0)}}}}
{{\thicklines{
{\qbezier[166](10,0)(11.5,-3)(13,0)}}}}
{{\thicklines{
{\qbezier[166](19,0)(20.5,-3)(22,0)}}}}
{{\thicklines{
{\qbezier[166](33,0)(34.5,-3)(36,0)}}}}
{{\thicklines{
{\qbezier[160](22,0)(23,1.5)(24,2)}}}}
{{\thicklines{
{\qbezier[6](24,2)(25,2.5)(25.5,2.7)}}}}
{{\thicklines{
{\qbezier[160](36,0)(37,1.5)(39,2)}}}}
{{\thicklines{
{\qbezier[160](31,3)(32.5,2)(33,0)}}}}
{{\thicklines{
{\qbezier[12](29,3.5)(30.5,3.3)(31,3)}}}}
{\thicklines
{\qbezier[8](10,0)(10,1.5)(10,3)}}
{\thicklines
{\qbezier[8](13,0)(13,1.5)(13,3)}}
{\thicklines
{\qbezier[8](19,0)(19,1.5)(19,3)}}
{\thicklines
{\qbezier[8](22,0)(22,1.5)(22,3)}}
{\thicklines
{\qbezier[8](33,0)(33,1.5)(33,3)}}
{\thicklines
{\qbezier[8](36,0)(36,1.5)(36,3)}}
\put(10,3){\vector(1,0){2.9}}
\put(13,3){\vector(-1,0){2.9}}
\put(19,3){\vector(1,0){2.9}}
\put(22,3){\vector(-1,0){2.9}}
\put(33,3){\vector(1,0){2.9}}
\put(36,3){\vector(-1,0){2.9}}
\end{picture}
\end{center}
\vspace{0.5cm}
\caption{\small A path in the set $\cW_n^\omega$.}
\label{fig-BG1}
\end{figure}
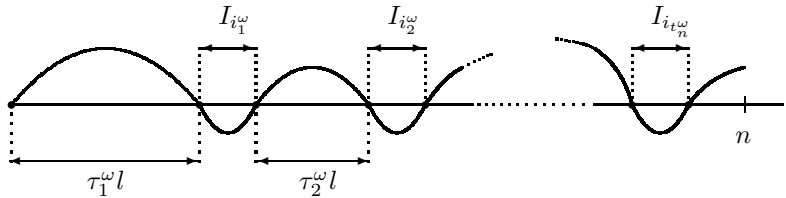

\medskip\noindent 
(1) Let 
\[
\begin{aligned}
R(n) &= P(S_i>0\,\,\forall\,0<i<n,\,S_n=0),\\
\bar{R}(n) &= P(S_i>0\,\,\forall\, 0< i\leq n).
\end{aligned}
\]
Insert the indicator of the set $\cW_n^\omega$ into the definition of
the partition sum, to estimate
\[
\log Z_n^{\beta,h,\omega} \geq \sum_{j=1}^{t_n^\omega} \log R(\tau_j^\omega l)\,
+\,t_n^\omega\,[\log R(l)+2\beta(\delta-h)l] + \log \bar{R}(n-i_{t_n^\omega} l).
\]

\medskip\noindent 
(2) Note that there exists a $C>0$ such that $R(n)\geq C/n^{3/2}$ for 
$n\in\N$. Use this to deduce from (1) that 
\[
\log Z_n^{\beta,h,\omega} 
\geq t_n^\omega\big[\log C-\tfrac32\,\log(\tfrac{n}{t_n^\omega }-l)\big] 
+ t_n^\omega\,\big[\log C - \tfrac32\log l + 2\beta(\delta-h)l\big]
+ O(\log n),
\]
where the first term arises after we apply Jensen's inequality:
\[
\frac{1}{t_n^\omega} \sum_{j=1}^{t_n^\omega} \log \tau_j^\omega
\leq \log\left(\frac{1}{t_n^\omega} \sum_{j=1}^{t_n^\omega} \tau_j^\omega\right).
\]

\medskip\noindent 
(3) Abbreviate 
\[
q_{l,\delta} = \P\Big(\Omega_1 \leq -\delta l\Big).
\]
Use the ergodic theorem to prove that 
\[ 
\lim_{n\to \infty} \frac{t_n^\omega}{n}=\frac1l\, 
\frac{q_{l,\delta}}{1+q_{l,\delta}} = p_{l,\delta} \qquad \omega\text{-a.s.}.
\]
(Note that $k \in \cup_{j\in\N_0} i_j^\omega$ if and only if $\Omega_k \leq -\delta l$ 
and $k-1 \notin \cup_{j\in\N_0} i_j^\omega$.) Since $\sum_{j=1}^{t_n^\omega} \tau_j^\omega l 
\leq n-t_n^\omega l$, it follows that
\[
\limsup_{n\to\infty} \frac{\sum_{j=1}^{t_n^\omega} \tau_i^\omega l}{t_n^\omega} 
\leq \lim_{n\to\infty} \frac{n-t_n^\omega l}{t_n^\omega} = p_{l,\delta}^{-1}-l
\qquad \omega\text{-a.s.}
\]
Conclude from (2) that
\[
\liminf_{n\to \infty} \frac1n \log Z_n^{\beta,h,\omega} \geq p_{l,\delta} 
[-\tfrac32\, \log(p_{l,\delta}^{-1}-l) +  2 \beta (\delta-h) l+O(\log l)]
\qquad \omega\text{-a.s.}
\]
This inequality is valid for all $l\in 2\N$. 

\medskip\noindent
(4) Show, with the help of Cram\'er's theorem of large deviation theory applied 
to $\omega$, that 
\[
\lim_{l \to \infty} \tfrac1l \log q_{l,\delta}
= - \sup_{\lambda>0} \big[\lambda \delta -\log M(-\lambda)\big] 
= - \Sigma(\delta),
\]
where $M(\lambda)=\E(\erom^{\lambda \omega_1})$, the supremum may be trivially
restricted to $\lambda>0$, and the right-hand side is the Legendre transform 
of the cumulant generating function $\lambda \mapsto \log M(-\lambda)$. Use the 
last display and the relation $p_{l,\delta}^{-1}-l=l/q_{l,\delta}$ to show that 
\[
\lim_{l \to \infty} \tfrac1l \log (p_{l,\delta}^{-1}-l) 
= \Sigma(\delta).
\]

\medskip\noindent 
(5) So far $\delta \in (0,1]$ is arbitrary. Now combine (3) and (4), optimize 
over $\delta$, and use that 
\[
\tfrac34 \log M(\tfrac43 \beta) 
= \sup_{\delta \in (0,1]} \big[-\tfrac34 \Sigma(\delta)+\beta \delta\big]
= \tfrac34 \sup_{\delta \in (0,1]} \big[\tfrac43\beta\delta-\Sigma(\delta)\big], 
\]
which is the (inverse) Legendre transform of the rate function in Cram\'er's
theorem, to conclude that $g^\mathrm{que}(\beta,h)>0$ as soon as 
\[
\tfrac34 \log M(\tfrac43 \beta)-\beta h>0.
\]
This completes the proof because $M(\tfrac43\beta) = \cosh(\tfrac43\beta)$.


\subsection{Order of the phase transition}
\label{copolordphtr}

In the proof below we pretend that $\omega$ is an i.i.d.\ sequence of
standard normal random variables, rather than Bernoulli random variables.
At the end of the proof we will see how to adapt the argument.

Define the set of trajectories 
\[
\widetilde{\cW}_n^\omega
=\big\{S\colon\,S_i= 0\,\,\forall j\in \cup_{j=1}^{t_n^\omega} \partial I_{i_j^\omega}\big\}
\cap \big\{S\colon\,S_i>0\,\, \forall i\in \{0,\dots,n\} \setminus 
\cup_{j=1}^{t_n^\omega} I_{i_j^\omega} \big\}.
\]

\begin{figure}[htbp]
\begin{center}
\setlength{\unitlength}{0.25cm}
\begin{picture}(15,5)(12,0)
\put(19,0){\circle*{.35}}
\put(22,0){\circle*{.35}}
{\qbezier(39,0.4)(39,0)(39,-0.4)}
\put(38.5,-2){$n$}
{\qbezier(0,0)(5,0)(15,0)}
{\thicklines
\qbezier[440](0,0)(13,0)(24.5,0)}
{\thicklines
\qbezier[14](24.5,0)(27,0)(31,0)}
{\thicklines
\qbezier[340](31,0)(32,0)(41,0)}
\put(10,0){\circle*{.35}}
\put(0,0){\circle*{.35}}
\put(13,0){\circle*{.35}}
\put(33,0){\circle*{.35}}
\put(36,0){\circle*{.35}}
\put(4,-4.5){$\tau_1^\omega l$}
\put(15.3,-4.5){$\tau_2^\omega l$}
\put(11,4.3){$I_{i_1^\omega}$}
\put(20,4.3){$I_{i_2^\omega}$}
\put(34,4.3){$I_{i_{t_n}^\omega}$}
{\thicklines
{\qbezier[6](0,0)(0,-1.5)(0,-3)}}
{\thicklines
{\qbezier[6](10,0)(10,-1.5)(10,-3)}}
{\thicklines
{\qbezier[6](13,0)(13,-1.5)(13,-3)}
}
{\thicklines
{\qbezier[6](19,0)(19,-1.5)(19,-3)}
}
\put(0.2,-3){\vector(1,0){9.7}}
\put(9.7,-3){\vector(-1,0){9.8}}
\put(13.2,-3){\vector(1,0){5.8}}
\put(18.7,-3){\vector(-1,0){5.6}}
{{\thicklines{
{\qbezier[166](0,0)(5,6)(10,0)}}}}
{{\thicklines{
{\qbezier[166](10,0)(10.25,-1)(10.5,0)}}}}
{{\thicklines{
{\qbezier[166](10.5,0)(10.75,2)(11,0)}}}}
{{\thicklines{
{\qbezier[166](11,0)(11.25,-2)(11.5,0)}}}}
{{\thicklines{
{\qbezier[166](11.5,0)(11.75,1)(12,0)}}}}
{{\thicklines{
{\qbezier[166](13,0)(16,4)(19,0)}}}}
{{\thicklines{
{\qbezier[166](12,0)(12.5,-2)(13,0)}}}}
{{\thicklines{
{\qbezier[166](19,0)(19.25,-1)(19.5,0)}}}}
{{\thicklines{
{\qbezier[166](19.5,0)(19.75,2)(20,0)}}}}
{{\thicklines{
{\qbezier[166](20,0)(20.25,-2)(20.5,0)}}}}
{{\thicklines{
{\qbezier[166](20.5,0)(20.85,2)(21.25,0)}}}}
{{\thicklines{
{\qbezier[166](21.25,0)(21.65,-1.5)(22,0)}}}}
{{\thicklines{
{\qbezier[166](33,0)(33.25,-2)(33.5,0)}}}}
{{\thicklines{
{\qbezier[166](33.5,0)(33.85,1)(34.25,0)}}}}
{{\thicklines{
{\qbezier[166](34.25,0)(34.5,-2)(34.75,0)}}}}
{{\thicklines{
{\qbezier[166](34.75,0)(35.15,2)(35.5,0)}}}}
{{\thicklines{
{\qbezier[166](35.5,0)(35.75,-1)(36,0)}}}}
{{\thicklines{
{\qbezier[160](22,0)(23,1.5)(24,2)}}}}
{{\thicklines{
{\qbezier[6](24,2)(25,2.5)(25.5,2.7)}}}}
{{\thicklines{
{\qbezier[160](36,0)(37.5,2.5)(39,4)}}}}
{{\thicklines{
{\qbezier[160](31,3)(32.5,2)(33,0)}}}}
{{\thicklines{
{\qbezier[12](29,3.5)(30.5,3.3)(31,3)}}}}
{\thicklines
{\qbezier[8](10,0)(10,1.5)(10,3)}}
{\thicklines
{\qbezier[8](13,0)(13,1.5)(13,3)}}
{\thicklines
{\qbezier[8](19,0)(19,1.5)(19,3)}}
{\thicklines
{\qbezier[8](22,0)(22,1.5)(22,3)}}
{\thicklines
{\qbezier[8](33,0)(33,1.5)(33,3)}}
{\thicklines
{\qbezier[8](36,0)(36,1.5)(36,3)}}
\put(10,3){\vector(1,0){2.9}}
\put(13,3){\vector(-1,0){2.9}}
\put(19,3){\vector(1,0){2.9}}
\put(22,3){\vector(-1,0){2.9}}
\put(33,3){\vector(1,0){2.9}}
\put(36,3){\vector(-1,0){2.9}}
\end{picture}
\end{center}
\vspace{1cm}
\caption{\small A path in the set $\widetilde{\cW}_n^\omega$.}
\label{fig-GT1}
\end{figure}
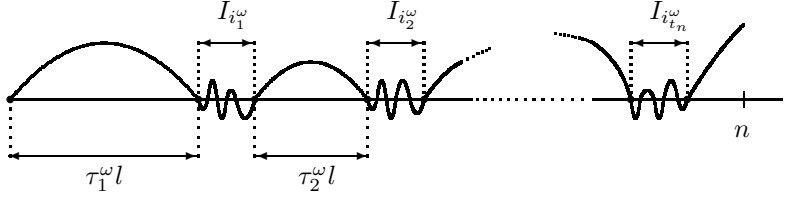


\medskip\noindent
(6) Similarly as in (1), insert the indicator of the set $\widetilde{\cW}_n^\omega$ 
into the definition of the partition function to estimate 
\[
\log Z_n^{\beta,h_c,\omega} \geq \sum_{j=1}^{t_n^\omega} \log R(\tau_j^\omega l)
+ \sum_{j=1}^{t_n^\omega} \log Z_{l}^{\beta,h_c,\,\theta^{i_j^\omega\, l}(\omega)}
+ \log \bar{R}\big(n-(i^\omega_{\,t_n^\omega}+1) l\big),
\]
where $\theta^l(\omega)=(\omega_{i+l})_{i\in\N}$.

\medskip\noindent 
(7) Take the expectation over $\P$ on both sides of (6), divide by $n$ and use (3), 
to obtain 
\[
\begin{aligned}
g^\mathrm{que}(\beta,h_c) 
&\geq  p_{l,\delta} \big[-\tfrac32\,\log(p_{l,\delta}^{-1}-l)+O(\log l)\big]\\ 
&\qquad + \liminf_{n\to \infty} \tfrac1n\, 
\E\left(\sum_{j=1}^{t_n^\omega} 
\log Z_{l}^{\beta,h_c,\,\theta^{i_j^\omega\, l}(\omega)}\right).
\end{aligned}
\]

\medskip\noindent
(8) Use a martingale property to prove that 
\[
\frac1n \E\left(\sum_{j=1}^{t_n^\omega} 
\log Z_{l}^{\beta,h_c,\,\theta^{i_j^\omega\, l}(\omega)}\right)
= E\left(\frac{t_n^\omega}{n}\right)\,\E\left(\log Z_{l}^{\beta, h_c, \omega} 
\mid \Omega_1\leq -\delta l\right),
\]
which gives
\[
g^\mathrm{que}(\beta,h_c)\geq  p_{l,\delta} \left[-\tfrac32\, 
\log(p_{l,\delta}^{-1}-l) + O(\log l)
+ \E\left(\log Z_{l}^{\beta, h_c, \omega} \mid \Omega_1\leq -\delta l\right)\right].
\]

\medskip\noindent 
(9) Deduce from (8) and (3) that 
\[
-\tfrac{3}{2}\Sigma(\delta)
+\tfrac{1}{l}\,\E\left(\log Z_{l}^{\beta,h_c,\omega} \mid \Omega_1 \leq -\delta l\right)
+o(1) \leq 0, \quad  \delta>0,\,l\in 2\N,\, l\to\infty.
\]
For large $l$, considering $l$ i.i.d.\ Gaussian random variables with mean $0$ and 
variance $1$ conditioned to have sum $\leq -\delta l$ is equivalent to considering 
$l$ i.i.d.\ Gaussian random variables with mean $-\delta$ and variance $1$. Therefore
we can replace $\E(\log Z_l^{\beta,h_c,\omega} \mid \Omega_1\leq -\delta l)$ by 
$\E(\log Z_l^{\beta,h_c-\delta,\omega})+o(1)$ and so, after we let $l\to\infty$, the 
inequality in the last display yields
\[
g^\mathrm{que}(\beta,h_c-\delta) \leq \tfrac{3}{2}\Sigma_0(\delta).
\]
Combine the lower bound on $g^\mathrm{que}(\beta,h_c)$ with the upper bound on 
$g^\mathrm{que}(\beta,h_c-\delta)$, and use that $\Sigma_0(\delta)=\tfrac12\delta^2
[1+o(1)]$ as $\delta\downarrow 0$, to obtain that
\[
g^\mathrm{que}(\beta,h_c-\delta)-g(\beta,h_c) \leq \tfrac14\delta^2 \quad \mbox{ for } 
\delta \mbox{ small enough}.
\]
This completes the proof for standard Gaussian disorder.

\medskip\noindent 
(10) It is easy to extend the proof to binary disorder. All that is needed is to show 
that the Gaussian approximation in (9) carries through.



\begin{thebibliography}{999}				

\bibitem{Al07}
K.S.\ Alexander,
Ivy on the ceiling: first-order polymer depinning transitions with quenched disorder,
Markov Proc.\ Relat.\ Fields 13 (2007) 663--680.

\bibitem{Al08}
K.S.\ Alexander,
The effect of disorder on polymer depinning transitions,
Commun.\ Math.\ Phys.\ 279 (2008) 117--146.

\bibitem{AlSi06}
K.S.\ Alexander and V.\ Sidoravicius,
Pinning of polymers and interfaces by random potentials,
Ann.\ Appl.\ Probab.\ 16 (2006) 636--669.

\bibitem{AlZy08}
K.S. Alexander and N.\ Zygouras,
Quenched and annealed critical points in polymer pinning models,
Commun.\ Math.\ Phys.\ 291 (2008) 659--689.

\bibitem{AlZy10}
K.S. Alexander and N.\ Zygouras,
Equality of critical points for polymer depinning transitions with 
loop exponent one, 
Ann.\ Appl.\ Probab.\ 20 (2010) 356--66.  

\bibitem{AuDuEcSy08}
L.\ Auvray, B.\ Duplantier, A.\ Echard and C.\ Sykes, Physique
des polym\`eres et membranes biologiques, Partie 1, Ecole 
Polytechnique, Palaiseau, France. 

\bibitem{BaDuGoSl11}
R.\ Bauerschmidt, H.\ Duminil-Copin, J.\ Goodman, G.\ Slade,
Lectures on self-avoiding walks, in this volume. 

\bibitem{Be11}
V.\ Beffara, Schramm-Loewner evolution and other conformally invariant objects,
in this volume.

\bibitem{BeTo10}
Q.\ Berger and F.L.\ Toninelli, On the critical point of the 
random walk pinning model in dimension $d=3$,
Electr.\ J.\ Probab.\ 15 (2010) 654--683.

\bibitem{BeTiVi08} 
S.\ Bezerra, S.\ Tindel and F.\ Viens, 
Superdiffusivity for a Brownian polymer in a continuous Gaussian environment, 
Ann.\ Probab.\ 36 (2008) 1642--1675.

\bibitem{Bi04}
M.\ Birkner,
A condition for weak disorder for directed polymers in random environment,
Electr.\ Comm.\ Prob.\ 9 (2004) 22--25.

\bibitem{BiGrdHo10}
M.\ Birkner, A.\ Greven and F.\ den Hollander,
Quenched large deviation principle for words in a letter sequence, 
Probab.\ Theory Relat.\ Fields 148 (2010) 403--456. 

\bibitem{BiGrdHo11}
M.\ Birkner, A.\ Greven and F.\ den Hollander,
Collision local time of transient random walks and intermediate phases in 
interacting stochastic systems,
Electr.\ J.\ Probab.\ 16 (2011) 552--586.

\bibitem{BiSu10} 
M.\ Birkner and R.\ Sun, 
Annealed vs quenched critical points for a random walk pinning model, 
Ann.\ Inst.\ H.\ Poincar\'e Probab.\ Stat.\ 46 (2010) 414--441. 
                                                
\bibitem{BiSupr} 
M.\ Birkner and R.\ Sun, 
Disorder relevance for the random walk pinning model in dimension $3$,
Ann.\ Inst.\ H.\ Poincar\'e Probab.\ Stat.\ 47 (2011) 259--293. 

\bibitem{BidHo99}
M.\ Biskup and F.\ den Hollander, 
A heteropolymer near a linear interface, 
Ann.\ Appl.\ Probab.\ 9 (1999) 668--687.

\bibitem{BoGi04}
T.\ Bodineau and G.\ Giacomin,
On the localization transition  of random copolymers near selective interfaces,
J.\ Stat.\ Phys.\ 117 (2004) 801--818.

\bibitem{BoGiLaTo08}
T.\ Bodineau, G.\ Giacomin, H.\ Lacoin and F.L.\ Toninelli,
Copolymers at selective interfaces: new bounds on the phase diagram,
J.\ Stat.\ Phys.\ 132 (2008) 603--626.

\bibitem{Bo89}
E.\ Bolthausen, 
A note on diffusion of directed polymers in a random environment,
Commun.\ Math.\ Phys.\ 123 (1989) 529--534. 

\bibitem{BoCadTi09}
E.\ Bolthausen, F.\ Caravenna and B.\ de Tili\`ere,
The quenched critical point of a diluted disordered polymer model,
Stoch.\ Proc.\ Appl.\ 119 (2009) 1479--1504.

\bibitem{BodHo97}
E.\ Bolthausen and F.\ den Hollander, 
Localization transition for a polymer near an interface, 
Ann.\ Probab.\ 25 (1997) 1334--1366.

\bibitem{BodHoOppr}
E.\ Bolthausen, F.\ den Hollander and A.\ Opoku,
A copolymer near a selective interface: variational characterization of 
the free energy, arXiv.org: 1110.1315 [math.PR], 
manuscript in preparation.

\bibitem{BrGuWh92}
R.\ Brak, A.J.\ Guttmann and S.G.\ Whittington,
A collapse transition in a directed walk model,
J.\ Phys.\ A: Math.\ Gen.\ 25 (1992) 2437--2446.

\bibitem{CaCa09}
A.\ Camanes and P.\ Carmona, 
The critical temperature of a directed polymer in a random environment, 
Markov Proc.\ Relat.\ Fields 15 (2009) 105--116.

\bibitem{CaMaTo08}
P.\ Caputo, F.\ Martinelli and F.L.\ Toninelli,
On the approach to equilibrium for a polymer with adsorption and repulsion,
Electr.\ J.\ Probab.\ 13 (2008) 213--258.

\bibitem{CaLaMaSiTopr}
P.\ Caputo, H.\ Lacoin, F.\ Martinelli, F.\ Simenhaus and F.L.\ Toninelli,
Polymer dynamics in the depinned phase: metastability with logarithmic barriers,
arXiv.org: 1007.4470 [math.PR], to appear in Probab.\ Theory Relat.\ Fields.

\bibitem{CaGi10}
F.\ Caravenna and G.\ Giacomin,
The weak coupling limit of disordered copolymer models,
Ann.\ Probab.\ 38 (2010) 2322--2378.

\bibitem{CaGiGu06}
F.\ Caravenna, G.\ Giacomin and M.\ Gubinelli,
A numerical approach to copolymers at selective interfaces,
J.\ Stat.\ Phys.\ 122 (2006) 799--832.

\bibitem{CaGiTopr}
F.\ Caravenna, G.\ Giacomin and F.L.\ Toninelli,
Copolymers at selective interfaces: settled issues and open problems, In: Probability in complex physical systems. In honour of Erwin Bolthausen and JŸrgen GŠrtner. Edited by J.-D. Deuschel, B. Gentz, W. Kšnig, M. von Renesse, M. Scheutzow, U. Schmock.
Springer Proceedings in Mathematics 11 (2012), 289-312. 

\bibitem{CaGiZa06}
F.\ Caravenna, G.\ Giacomin and L.\ Zambotti, 
Sharp asymptotic behavior for wetting models in (1+1)-dimension,
Electr.\ J.\ Probab.\ 11 (2006) 345--362.

\bibitem{CaPe09a}
F.\ Caravenna and N.\ P\'etr\'elis,
A polymer in a multi-interface medium,
Ann.\ Appl.\ Probab.\ 19 (2009) 1803--1839.

\bibitem{CaPe09b}
F.\ Caravenna and N.\ P\'etr\'elis,
Depinning of a polymer in a multi-interface medium,
Electr.\ J.\ Probab.\ 14 (2009) 2038--2067.

\bibitem{CaHu02}
P.\ Carmona and Y.\ Hu,
On the partition function of a directed polymer in a Gaussian random environment,
Probab.\ Theory Relat.\ Fields 124 (2002) 431--457.

\bibitem{CaHu04}
P.\ Carmona and Y.\ Hu,
Fluctuation exponents and large deviations for directed polymers in a 
random environment,
Stoch.\ Proc.\ Appl.\ 112 (2004) 285--308.

\bibitem{ChdHo10} 
D.\ Cheliotis and F.\ den Hollander, Variational characterization of 
the critical curve for pinning of random polymers, arXiv.org: 1005.3661v1 [math.PR],
to appear in Ann.\ Probab.

\bibitem{ClLiSl07}
N.\ Clisby, R.\ Liang and G.\ Slade,
Self-avoiding walk enumeration via the lace expansion. 
J.\ Phys.\ A: Math.\ Theor.\ 40 (2007) 10973--11017.

\bibitem{CoShYo03} 
F.\ Comets, T.\ Shiga and N.\ Yoshida, 
Directed polymers in random environment: Path localization and 
strong disorder, 
Bernoulli 9 (2003) 705--723.

\bibitem{CoShYo04}
F.\ Comets, T.\ Shiga and N.\ Yoshida, 
Probabilistic analysis of directed polymers in a random environment:
a review, in: \emph{Stochastic Analysis on Large Scale Systems},
Adv.\ Stud.\ Pure Math.\ 39 (2004), pp.\ 115--142.  

\bibitem{CoVa06}
F.\ Comets and V.\ Vargas,
Majorizing multiplicative cascades for directed polymers in random media,
Alea 2 (2006) 267--277. 

\bibitem{CoYo04}
F.\ Comets and N.\ Yoshida,
Brownian directed polymers in random environment,
Commun.\ Math.\ Phys.\ 254 (2004) 257--287.

\bibitem{CoYo04ext}
F.\ Comets and N.\ Yoshida,
Some new results on Brownian directed polymers in random environment,
RIMS Kokyuroku 1386 (2004) 50--66.

\bibitem{CoYo06} 
F.\ Comets and N.\ Yoshida, 
Directed polymers in random environment are diffusive at weak disorder, 
Ann.\ Probab.\ 34 (2006) 1746--1770. 

\bibitem{Co96}
L.N.\ Coyle,
A continuous time version of random walks in a random potential,
Stoch.\ Proc.\ Appl.\ 64 (1996) 209--235.

\bibitem{Co98}
L.N.\ Coyle,
Infinite moments of the partition function for random walks in a random 
potential,
J.\ Math.\ Phys.\ 39 (1998) 2019--2034.

\bibitem{CuHw97}
D.\ Cule and T.\ Hwa,
Denaturation of heterogeneous DNA,
Phys.\ Rev.\ Lett.\ 79 (1997) 2375--2378.

\bibitem{DaGa87}
D.A.\ Dawson and J.\ G\"artner, Large deviations from the McKean-Vlasov 
limit for weakly interacting diffusions, Stochastics 20 (1987) 247--308.

\bibitem{DeZe98} 
A.\ Dembo and O.\ Zeitouni,
\emph{Large Deviations Techniques and Applications}, Springer, 1998.

\bibitem{DeGiLaTo09}
B.\ Derrida, G.\ Giacomin, H.\ Lacoin and F.L.\ Toninelli,
Fractional moment bounds and disorder relevance for pinning models,
Commun.\ Math.\ Phys.\ 287 (2009) 867--887.

\bibitem{DeGiZa05}
J.-D.\ Deuschel, G.\ Giacomin and L.\ Zambotti,
Scaling limits of  equilibrium wetting models in (1+1)-dimension,
Probab.\ Theory Relat.\ Fields 132 (2005) 471--500.

\bibitem{DuSm11}
H.\ Duminil-Copin  and S.\ Smirnov, 
Conformal invariance of lattice models, 
in this volume.

\bibitem{DuSa87}
B.\ Duplantier and H.\ Saleur,
Exact tricritical exponents for polymers at the FTHETA point in
two dimensions,
Phys.\ Rev.\ Lett.\ 59 (1987) 539--542.

\bibitem{EvDe92}
M.R.\ Evans and B.\ Derrida,
Improved bounds for the transition temperature of directed polymers
in a finite-dimensional random medium,
J.\ Stat.\ Phys.\ 69 (1992) 427--437.

\bibitem{Fi84}
M.E.\ Fisher,
Walks, walls, wetting and melting,
J.\ Stat.\ Phys.\ 34 (1984) 667--729. 

\bibitem{GaSt11}
C.\ Garban and J.E.\ Steif,
Noise-sensitivity and percolation,
in this volume.

\bibitem{Gi07}
G.\ Giacomin, \emph{Random Polymer Models}, Imperial College Press,
London, 2007.

\bibitem{GiLaTo10}
G.\ Giacomin, H.\ Lacoin and F.L.\ Toninelli,
Marginal relevance of disorder for pinning models,
Commun.\ Pure Appl.\ Math.\ 63 (2010) 233--265.

\bibitem{GiLaTo11}
G.\ Giacomin, H.\ Lacoin and F.L.\ Toninelli,
Disorder relevance at marginality and critical point shift,
Ann.\ Inst.\ H.\ Poincar\'e Probab.\ Stat.\ 47 (2011) 148--175.

\bibitem{GiTo05}
G.\ Giacomin and F.L.\ Toninelli,
Estimates on path delocalization  for copolymers at selective interfaces,
Probab.\ Theory Relat.\ Fields 133 (2005) 464--482.

\bibitem{GiTo06b}
G.\ Giacomin and F.L.\ Toninelli,
Smoothing of depinning transitions for directed polymers with quenched disorder,
Phys.\ Rev.\ Lett.\ 96 (2006) 070602.

\bibitem{GiTo06c}
G.\ Giacomin and F.L.\ Toninelli, 
Smoothing effect of quenched disorder on polymer depinning transitions, 
Commun.\ Math.\ Phys.\ 266 (2006) 1--16.

\bibitem{GiTo06d}
G.\ Giacomin and F.L.\ Toninelli, 
The localized phase of disordered copolymers with adsorption, 
Alea 1 (2006) 149--180.

\bibitem{GiTo07}
G.\ Giacomin and F.L.\ Toninelli, 
Force-induced depinning of directed polymers,
J.\ Phys.\ A: Math.\ Gen.\ 40 (2007) 5261--5275.

\bibitem{GiTo09}
G.\ Giacomin and F.L.\ Toninelli, 
On the irrelevant disorder regime of pinning models,
Ann.\ Probab.\ 37 (2009) 1841--1875.

\bibitem{Gu09}
A.J.\ Guttmann (Ed.), \emph{Polygons, Polyominoes and Polycubes},
Lecture Notes in Physics 775, Springer and Canopus Academic Publishing 
Ltd., 2009. 

\bibitem{HaSl92a}
T.\ Hara and G.\ Slade, 
Self-avoiding walk in five or more dimensions, I.\ The critical behaviour, 
Commun.\ Math.\ Phys.\ 147 (1992) 101--136.

\bibitem{HaSl92b}
T.\ Hara and G.\ Slade, 
The lace expansion for self-avoiding walk in five or more dimensions, 
Rev.\ Math.\ Phys.\ 4 (1992) 235--327.

\bibitem{vdHoKl01}
R.\ van der Hofstad and A.\ Klenke,
Self-attractive random polymers,
Ann.\ Appl.\ Prob.\ 11 (2001) 1079--1115.

\bibitem{vdHoKlKo02}
R.\ van der Hofstad, A.\ Klenke and W.\ K\"onig,
The critical attractive random polymer in dimension one,
J.\ Stat.\ Phys.\ 106 (2002) 477--520.

\bibitem{dHo00}
F.\ den Hollander, \emph{Large Deviations}, Fields Institute Monographs, 
AMS, Providence RI, 2000.

\bibitem{dHo09}
F.\ den Hollander, \emph{Random Polymers}, Lecture Notes in Ma\-thematics
1974, Springer, Berlin, 2009.

\bibitem{dHoPe09a} 
F.\ den Hollander and N.\ P\'etr\'elis, 
On the localized phase of a copolymer in an emulsion: 
supercritical percolation regime, 
Commun.\ Math.\ Phys.\ 285 (2009) 825--871.

\bibitem{dHoPe09b} 
F.\ den Hollander and N.\ P\'etr\'elis, 
On the localized phase of a copolymer in an emulsion: subcritical 
percolation regime,
J.\ Stat.\ Phys.\ 134 (2009) 209--241.

\bibitem{dHoPe10}
F.\ den Hollander and N.\ P\'etr\'elis, 
A mathematical model for a copolymer in an emulsion, 
J.\ Math.\ Chem.\ 48 (2010) 83--94.

\bibitem{dHoPepr}
F.\ den Hollander and N.\ P\'etr\'elis, 
Free energy of a copolymer in a micro-emulsion, arXiv.org: 1204.1234 [math.PR], 
manuscript in preparation.

\bibitem{dHoWh06}
F.\ den Hollander and S.G.\ Whittington, 
Localization transition for a copolymer in an emulsion, 
Theor.\ Prob.\ Appl.\ 51 (2006) 193--240.

\bibitem{ImSp88}
J.Z.\ Imbrie and T.\ Spencer,
Diffusion of directed polymers in a random environment,
J.\ Stat.\ Phys.\ 52 (1988) 609--626.

\bibitem{IoVe10}
D.\ Ioffe and Y.\ Velenik,
The statistical mechanics of stretched polymers, 
Braz.\ J.\ Probab.\ Stat.\ 24 (2010) 279--299.

\bibitem{IoVe11}
D.\ Ioffe and Y.\ Velenik,
Crossing random walks and stretched polymers at weak disorder,
arXiv.org: 1002.4289v2 [math.PR], to appear in Ann.\ Probab.

\bibitem{IoVepra}
D.\ Ioffe and Y.\ Velenik,
Stretched polymers in random environment,
arXiv.org: 1011.0266v1 [math.PR].

\bibitem{IoVeprb}
D.\ Ioffe and Y.\ Velenik,
Self-attractive random walks: The case of critical drifts,
arXiv.org: 1104.4615v1 [math.PR]. 

\bibitem{Jewww}
I.\ Jensen, homepage ({\tt www.ms.unimelb.edu.au/$\sim$iwan}).

\bibitem{KaMuPe00}
Y.\ Kafri, D.\ Mukamel and L.\ Peliti,
Why is the DNA denaturation transition first order?,
Phys.\ Rev.\ Lett.\ 85 (2000) 4988--4991.

\bibitem{La10}
H.\ Lacoin,
New bounds for the free energy of directed polymer in dimension
$1+1$ and $1+2$, 
Commun.\ Math.\ Phys.\ 294 (2010) 471--503.

\bibitem{La10b}
H.\ Lacoin,
The martingale approach to disorder relevance for pinning models,
Electr.\ Comm.\ Probab.\ 15 (2010) 418--427.

\bibitem{La11}
H.\ Lacoin,
Influence of spatial correlation for directed polymers,
Ann.\ Probab.\ 39 (2011) 139--175.

\bibitem{Le01}
M.\ Ledoux, \emph{The Concentration of Measure Phenomenon}, Mathematical 
Surveys and Monographs 89, American Mathematical Society, 2001.

\bibitem{Me04}
O.\ Mejane,
Upper bound of a volume exponent for directed polymers in a random environment,
Ann.\ Inst.\ H.\ Poincar\'e Probab.\ Stat. 40 (2004) 299--308. 

\bibitem{Pe00}
M.\ Petermann, 
Superdiffusivity of directed polymers in random environment,
Ph.D.\ thesis, University of Z\"urich, 2000.
 
\bibitem{Pe06}
N.\ P\'etr\'elis, 
\emph{Localisation d'un Polym\`ere en Interaction avec une Interface},
Ph.D.\ Thesis, University of Rouen, France, February 2, 2006.

\bibitem{Pe06b}
N.\ P\'etr\'elis, 
Polymer pinning at an interface,
Stoch.\ Proc.\ Appl.\ 116 (2006) 1600--1621.

\bibitem{Pe09}
N.\ P\'etr\'elis, 
Copolymer at selective interfaces and pinning potentials: weak coupling limits, 
Ann.\ Inst.\ H.\ Poincar\'e Probab.\ Stat.\ 45 (2009) 175--200. 

\bibitem{Pi97}
M.S.T.\ Piza,
Directed polymers in a random environment: some results on fluctuations,
J.\ Stat.\ Phys.\ 89 (1997) 581--603.

\bibitem{ScBaBi11}
R.D.\ Schram, G.T.\ Barkema and R.H.\ Bisseling,
Exact enumeration of self-avoiding walks,
arXiv.org: 1104.2184 [physics.math-ph]. 

\bibitem{SeSt88}
F.\ Seno and A.L.\ Stella,
$\theta$ point of a linear polymer in 2 dimensions: a renormalization group
analysis of Monte Carlo enumerations,
J.\ Physique 49 (1988) 739--748.

\bibitem{Si95}
Ya.G.\ Sinai,
A remark concerning random walks with random potentials,
Fund.\ Math.\ 147 (1995) 173--180. 

\bibitem{So09}
J.\ Sohier,
Finite size scaling for homogene<ous pinning models,
ALEA 6 (2009) 163--177.

\bibitem{Sp76}
F.\ Spitzer, \emph{Principles of Random Walk} (2nd.\ ed.), Springer, New York, 
1976.

\bibitem{St89}
J.M.\ Steele, Kingman's subadditive ergodic theorem,
Ann.\ Inst.\ Henri Poincar\'e 25 (1989) 93--98.

\bibitem{TeJvReOrWh96}
M.C.\ Tesi, E.J.\ Janse van Rensburg, E.\ Orlandini and S.G.\ Whittington,
Monte Carlo study of the interacting self-avoiding walk model in three dimensions,
J.\ Stat.\ Phys.\ 82 (1996) 155--181.

\bibitem{To08}
F.L.\ Toninelli,
A replica-coupling approach to disordered pinning models,
Commun.\ Math.\ Phys.\ 280 (2008) 389--401. 

\bibitem{To08a}
F.L.\ Toninelli,
Disordered pinning models and copolymers: beyond annealed bounds,
Ann.\ Appl.\ Probab.\ 18 (2008) 1569--1587.

\bibitem{To08b}
F.L.\ Toninelli,
Coarse graining, fractional moments and the critical slope of random polymers,
Electr.\ J.\ Probab.\ 14 (2009) 531--547. 

\bibitem{To09}
F.L.\ Toninelli,
Localization transition in disordered pinning models. Effect of randomness
on the critical properties, in: Lecture Notes in Mathematics 1970, Springer, 
Berlin, 2009, pp.\ 129-176. 

\bibitem{Zy09}
N.\ Zygouras,
Lyapounov norms for random walks in low disorder and dimension greater than three,
Probab.\ Theory Relat.\ Fields 143 (2009) 667--683.

\bibitem{Zypr}
N.\ Zygouras,
Strong disorder in semidirected random polymers,
arXiv.org: 1009.2693v2 [math.PR].

\end{thebibliography}
\end{document}